\documentclass[11pt, reqno]{amsart}
\usepackage{amsmath}
\usepackage{amssymb}
\usepackage{amsrefs}
\usepackage{amsthm}
\usepackage[english]{babel}
\usepackage[autostyle]{csquotes}
\usepackage{bbm}
\usepackage{caption}
\usepackage[autostyle]{csquotes}
\usepackage{enumerate}
\usepackage[shortlabels]{enumitem}
\newlist{steps}{enumerate}{1}
\setlist[steps, 1]{label = Step \arabic*:}
\usepackage{float}
\usepackage[margin=1.25in]{geometry}
\usepackage[colorlinks=true, urlcolor=blue, linkcolor=black]{hyperref}
\usepackage{tikz}
\usepackage{xcolor}
\usetikzlibrary{backgrounds}
\usepackage{quiver}
\usepackage{ stmaryrd }
  
\theoremstyle{plain}
\newtheorem{thm}{Theorem}[section]
\newtheorem{cor}[thm]{Corollary}
\newtheorem{conj}[thm]{Conjecture}
\newtheorem{lem}[thm]{Lemma}
\newtheorem{obs}[thm]{Observation}
\newtheorem{prop}[thm]{Proposition}

\theoremstyle{definition}
\newtheorem{defn}[thm]{Definition}
\newtheorem{rmk}[thm]{Remark}
\newtheorem{eg}[thm]{Example}

\newtheorem{moral}[thm]{Moral}
\newtheorem{notn}[thm]{Notation}

\usetikzlibrary{cd}

\usepackage{graphicx}

\newcommand{\Alg}{\operatorname{Alg}}

\newcommand{\Br}{\operatorname{Br}}

\newcommand{\Coh}{\operatorname{Coh}}
\newcommand{\Crit}{\operatorname{Crit}}

\newcommand{\IndCoh}{\operatorname{IndCoh}}

\newcommand{\Cone}{\operatorname{Cone}}
\newcommand{\Conf}{\operatorname{Conf}}
\newcommand{\coCone}{\operatorname{coCone}}

\newcommand{\Fuk}{\operatorname{Fuk}}
\newcommand{\Fun}{\operatorname{Fun}}

\newcommand{\Hom}{\operatorname{Hom}}

\newcommand{\Id}{\operatorname{Id}}
\newcommand{\Image}{\operatorname{Im}}

\newcommand{\Loc}{\operatorname{Loc}}

\newcommand{\Map}{\operatorname{Map}}

\newcommand{\Mod}{\operatorname{Mod}}

\newcommand{\Perf}{\operatorname{Perf}}
\newcommand{\Perv}{\operatorname{Perv}}

\newcommand{\QCoh}{\operatorname{QCoh}}

\newcommand{\Sym}{\operatorname{Sym}}

\newcommand{\Tot}{\operatorname{Tot}}

\newcommand{\Vect}{\operatorname{Vect}}

\newcommand{\Sp}{\operatorname{Sp}}

\newcommand{\opname}[1]{\operatorname{#1}}

\newcommand{\colim}{\operatorname{colim}}

\newcommand{\diag}{\operatorname{diag}}

\newcommand{\id}{\operatorname{id}}

\newcommand{\op}{\operatorname{op}}

\newcommand{\pt}{\operatorname{pt}}

\newcommand{\sh}{\operatorname{sh}}

\newcommand{\mush}{\mu\kern -2pt \sh}
\newcommand{\muPerv}{\mu\kern -2pt \Perv}

\newcommand{\bbA}{\mathbb{A}}
\newcommand{\bbC}{\mathbb{C}}

\newcommand{\bbG}{\mathbb{G}}

\newcommand{\bbP}{\mathbb{P}}

\newcommand{\bbR}{\mathbb{R}}

\newcommand{\bbZ}{\mathbb{Z}}

\newcommand{\calA}{\mathcal{A}}
\newcommand{\calB}{\mathcal{B}}
\newcommand{\calC}{\mathcal{C}}
\newcommand{\calD}{\mathcal{D}}
\newcommand{\calE}{\mathcal{E}}
\newcommand{\calEnd}{\mathcal{E}\kern -.5pt nd}
\newcommand{\calF}{\mathcal{F}}
\newcommand{\calFuk}{\mathcal{F}\kern -.5pt uk}
\newcommand{\calG}{\mathcal{G}}

\newcommand{\calHom}{\mathcal{H}\kern -.5pt om}

\newcommand{\calPerv}{\mathcal{P}\kern -.5pt erv}

\newcommand{\calRadon}{\mathcal{R}\kern -.5pt adon}
\newcommand{\calS}{\mathcal{S}}
\newcommand{\calSus}{\mathcal{S}\kern -.5pt us}

\newcommand{\calX}{\mathcal{X}}

\newcommand{\frakS}{\mathfrak{S}}

\newcommand{\actson}{\curvearrowright}

\newcommand{\isomfrom}{\xleftarrow{\sim}}
\newcommand{\isomto}{\xrightarrow{\sim}}

\newcommand{\resto}[1]{\big|_{#1}}
\newcommand{\ts}{\textsuperscript}
\newcommand{\ol}[1]{\overline{#1}}
\newcommand{\ul}[1]{\underline{#1}}

\newcommand{\sfP}{\mathsf{Perv}}
\newcommand{\olsftwoP}{\mathsf{2}\ol{\mathsf{Perv}}}
\newcommand{\sftwoP}{\mathsf{2Perv}}
\newcommand{\sftwoL}{\mathsf{2Loc}}

\newcommand{\sfLoc}{\mathsf{Loc}}

\newcommand{\comment}[1]{}

\newcommand{\xto}[1]{\xrightarrow{#1}}

\newcommand{\inclto}{\hookrightarrow}
\newcommand{\adjto}{\dashv}

\newcommand{\from}{\leftarrow}

\newcommand{\tofrom}{\rightleftarrows}

\newcommand{\pullback}{\arrow[dr, phantom, "\scalebox{1.5}{$\lrcorner$}" , very near start, color=black]}

\newcommand{\FS}{\opname{FS}}

\newcommand{\sfSph}{\mathsf{Sph}}
\newcommand{\sfSphMnd}{\mathsf{SphMnd}}
\newcommand{\sfAdj}{\mathsf{Adj}}
\newcommand{\sfBdj}{\mathsf{Bdj}}
\newcommand{\sfMnd}{\mathsf{Mnd}}
\newcommand{\sfAut}{\mathsf{Aut}}
\newcommand{\sfEnd}{\mathsf{End}}
\newcommand{\sfFun}{\mathsf{Fun}}
\newcommand{\sfSOD}{\mathsf{SOD}}
\newcommand{\sfCat}{\mathsf{Cat}}
\newcommand{\const}{\mathsf{const}}
\newcommand{\sfSt}{\mathsf{St}}

\newcommand{\sfA}{\mathsf{A}}
\newcommand{\sfB}{\mathsf{B}}
\newcommand{\sfC}{\mathsf{C}}
\newcommand{\sfD}{\mathsf{D}}

\newcommand{\sfE}{\mathsf{E}}
\newcommand{\sfH}{\mathsf{H}}

\newcommand{\sfN}{\mathsf{N}}
\newcommand{\sfQ}{\mathsf{Q}}

\newcommand{\sfS}{\mathsf{S}}
\newcommand{\sfT}{\mathsf{T}}

\newcommand{\Cuts}{\opname{Cuts}}
\newcommand{\Cycle}{\opname{Cycle}}
\newcommand{\Complex}{\opname{Complex}}
\newcommand{\pre}{\opname{pre}}

\newcommand{\ver}{\opname{vert}}
\newcommand{\hor}{\opname{hor}}
\newcommand{\dia}{\opname{diag}}
\newcommand{\zero}{(\infty,0)}
\newcommand{\one}{(\infty,1)}
\newcommand{\two}{(\infty,2)}
\newcommand{\wCat}{\widehat{\opname{Cat}}}

\newcommand{\oplax}{\opname{oplax}}

\newcommand{\sm}{\opname{sm}}

\newcommand{\frakd}{\mathfrak{d}}

\newcommand{\frakh}{\mathfrak{h}}

\newcommand{\frakn}{\mathfrak{n}}

\newcommand{\frakt}{\mathfrak{t}}

\newcommand{\frakadj}{\mathfrak{adj}}
\newcommand{\frakmnd}{\mathfrak{mnd}}
\newcommand{\frakend}{\mathfrak{end}}

\newcommand{\act}{\opname{active}}
\newcommand{\inert}{\opname{inert}}

\title{Microlocal Perverse Schobers and Radon Transform}
\date{}
\author{Yuji Okitani}

\begin{document}

\begin{abstract}
Perverse schobers are a categorification of perverse sheaves, originally proposed by Kapranov and Schechtman. The purpose of this paper is to initiate a microlocal study of perverse schobers. We first categorify $\Perv(\bbC,R)/\Loc(\bbC)$, the category of perverse sheaves on a complex line with singular points at $R$, modulo local systems. We use this to propose a general definition for microlocal perverse schobers supported on the open conormal to a germ of a hypersurface, and we conjecture that this is invariant under Radon transform. Here we make the key observation that while the analogous quotient of perverse schobers $\mathsf{2Perv}(\bbC,R)/\mathsf{2Loc}(\bbC)$ is a reasonable categorification, the resulting theory fails to be invariant under Radon transform. In fact, our proposed categorification can be recovered by correcting an instance of this failure in a universal manner. We prove Radon invariance when our hypersurface is the curve $y^m=x^n$ in $\bbC^2_{x,y}$. Along the way, we explain how our theory relates to Fourier transforms of perverse schobers, periodic SODs, and spherical monads.
\end{abstract}

\maketitle

\tableofcontents

\section{Introduction}
Perverse schobers are a categorification of perverse sheaves, initially proposed by Kapranov-Schechtman \cite{kapranovPerverseSchobers2015}. Just as perverse sheaves organize vector space invariants such as those arising in symplectic geometry and representation theory, perverse schobers are expected to organize categorical invariants. Much work has since gone into realizing this proposal in various settings \cites{bondalPerverseSchobersBirational2018,christGinzburgAlgebrasTriangulated2022,kapranovPerverseSchobers2015,dyckerhoffPerverseSchobersCoxeter2025}, although a general definition is still not available.

These categorifications, while ad hoc, follow a common recipe. First, one finds a description of the corresponding category of perverse \textit{sheaves} in terms of diagrams of abelian groups with relations. Then, one can consider diagrams of stable $\one$-categories of the same shape, and categorify relations following certain rules of categorification, established through these examples. Such rules can be found tabulated in \cite{christLaxAdditivity2025}*{\S 1.1}.

Studying the first step more carefully, one observes that these diagrammatic descriptions of perverse sheaf categories are very often of a microlocal origin. That is, the abelian groups are microstalks, and the arrows are communications of microstalks. Gelfand--MacPherson--Vilonen \cite{gelfandMicrolocalPerverseSheaves2005} take advantage of this perspective and present a possible approach to describe any perverse sheaf category as a diagram category, with no reference to derived categories or $t$-structures.

Therefore it is natural to expect that a general and uniform definition of perverse schobers can be achieved by categorifying this approach. The aim of this paper is to begin to develop a microlocal theory of perverse schobers, motivated by this expectation.

\subsection{Microlocal sheaves}
Microlocal sheaf theory \cite{kashiwaraSheavesManifolds1990} is an enhancement of sheaf theory which views sheaves on a smooth real manifold $X$ as living over the symplectic manifold $T^*X$. More precisely, the theory upgrades the support of a sheaf $\calF$ to a microsupport, which is a closed, conic subset inside $T^*X$. Thus given a closed, conic subset $\Lambda$, one can define a subcategory of the stable $\infty$-category $\sh(X)$ of sheaves valued in $\bbZ$,
\[\sh_{\Lambda}(X)\subseteq \sh(X),\]
consisting of those sheaves whose microsupport is contained in $\Lambda$. For example, given a reasonable stratification $\calS$ of $X$, $\sh_{T^*_\calS X}(X)$ is the subcategory of sheaves constructible with respect to $\calS$.

We can localize $\sh_{\Lambda}(X)$ into a presheaf of categories over $T^*X$, denoted
\[\Omega\mapsto \sh_{\Lambda}(X;\Omega) := \sh_{\Omega^c\cup \Lambda}(X)/\sh_{\Omega^c}(X).\]
By sheafification, we obtain a sheaf of categories that we denote $\mush_{\Lambda}(-)$, whose sections we refer to as microlocal sheaves. There is a precise sense in which microlocal sheaves are independent of the particular polarization. This is exhibited by quantized contact transformations as explained in \cite{kashiwaraSheavesManifolds1990}*{Chapter VII}.

For points $p\in \Lambda$ that are in generic position, microlocal cut-offs let us present the stalk $\mush_{\Lambda}(p)$ as a quotient category,
\begin{equation}\label{equation:cutoff}
    \mush_{\Lambda}(p)\simeq \sh_{0_U\cup \Lambda}(U)/\Loc(U),
\end{equation}
where $U$ is a small open neighbourhood of the projection $\pi(p)\in X$, and $0_U$ is the zero section.

\subsection{Microlocal perverse sheaves}\label{subsection:muperv}
If $X$ is a smooth complex manifold and $\Omega$ is $\bbC^\times$-conic then there is a subcategory of microlocal perverse sheaves
\[\muPerv_{\Lambda}(\Omega)\subset \mush_{\Lambda}(\Omega),\]
which organize into a sheaf of categories $\muPerv_{\Lambda}(-)$, as studied in \cites{andronikofMicrolocalVersionRiemannHilbert1994,gelfandMicrolocalPerverseSheaves2005,waschkiesMicrolocalPerverseSheaves2002,waschkiesStackMicrolocalPerverse2004,cotePerverseMicrosheaves2025}.

Microlocal perverse sheaves are particularly tractable near a point $p\in\Lambda$ in generic position -- that is, where
\[\pi^{-1}(\pi(p))\cap \Lambda = \bbC^\times p.\]
Near such $p$, the entire geometry is visible on the base; there is a hypersurface germ $H$ defined on a neighbourhood $U$ of $\pi(p)$, such that $\Lambda=\Lambda_H$ near $\bbC^\times p$, where $\Lambda_H$ denotes the closure of the conormal of the smooth locus of $H$. Waschkies \cite{waschkiesStackMicrolocalPerverse2004} uses microlocal cut-offs to prove an analog of (\ref{equation:cutoff}),
\[\muPerv_{\Lambda}(\bbC^{\times}p)\simeq \Perv_{\Lambda_H}(U)/\Loc(U).\]
We may choose local coordinates $x_1,\dots,x_{N-1},y$ on $U$ such that $\pi(p) = (0,\dots,0)$ and such that $H$ is cut out by an equation of the form
\begin{equation}\label{equation:equation}
    y^m + \text{ higher order terms}.
\end{equation}

Gelfand--MacPherson--Vilonen \cite{gelfandMicrolocalPerverseSheaves2005} show that $\Perv_{\Lambda_H}(U)/\Loc(U)$ has an intrinsically abelian description. Geometrically, the pair $(U,H)$ can be viewed as a family of discs $(\bbC,R)$ parameterized by $(x_1,\dots,x_{N-1})$, where $R\subset \bbC$ is a finite set of points, see Figure~\ref{figure:y2x4}. They explain that $\Perv_{\Lambda_H}(U)/\Loc(U)$ can be constructed out of the categories $\Perv(\bbC,R)/\Loc(\bbC)$. In turn, a quiver description of $\Perv(\bbC,R)/\Loc(\bbC)$ is given in \cite{gelfandPerverseSheavesQuivers1996}, which is intrinsically abelian. Let us refer to this description as the local GMV construction.

\begin{rmk}[Microlocal perverse sheaves at non-generic points]\label{rmk:gmvfull}
    By applying a suitable contact transformation, any point $p\in\Lambda$ can be made generic and thus understood via a pair $(U,H)$ as above. Moreover, if $p\in\Lambda$ is a codimension $\leq i$ singularity, then we can reduce the dimension so that $U$ has dimension $i+1$. The main result of \cite{gelfandMicrolocalPerverseSheaves2005} uses this to give an intrinsically abelian construction of the sheaf of categories $\muPerv_{\Lambda}(-)$ on the open locus $\Lambda^{\leq 1}\subseteq \Lambda$ of singularities of codimension $\leq1$ in terms of pairs of the form $(U\subset \bbC^2,C)$.
\end{rmk}

\begin{figure}
    \caption{Pair $(\bbC^2_{x,y},C)$ together with discs $(\bbC_y,R)$ parametrized by $x$}
    \label{figure:y2x4}
    \begin{tikzpicture}[scale=3]
    \clip (-1, -0.5) rectangle (1.2, 0.7);

    \draw[->, gray, line width=0.6pt] (-1, -0.4) -- (1.2, -0.4) node[above left, gray] {$x$};
    \draw[->, gray, line width=0.6pt] (-0.9, -0.5) -- (-0.9, 0.7) node[below right, gray] {$y$};


    \draw[black, thick, domain=-1:1, samples=100] plot (\x, {0.5*\x*\x}) node[below, black] {$C$};
    \draw[black, thick, domain=-1:1, samples=100] plot (\x, {-0.5*\x*\x});

    \draw[red, thick] (-0.5, -0.5) -- (-0.5, 0.5);
    \draw[red, thick] (-0.3, -0.5) -- (-0.3, 0.5);
    \draw[red, thick] (-0.1, -0.5) -- (-0.1, 0.5);
    \draw[red, thick] (0.1, -0.5) -- (0.1, 0.5);
    \draw[red, thick] (0.3, -0.5) -- (0.3, 0.5);
    \draw[red, thick] (0.5, -0.5) -- (0.5, 0.5);
    \end{tikzpicture}
\end{figure}
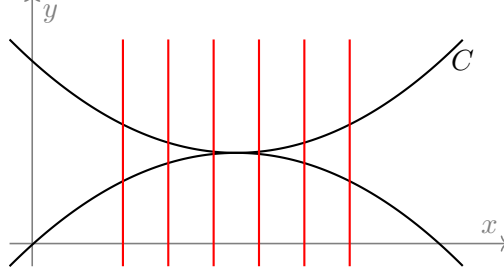

\subsection{The aim of this paper}
The main objective of this paper is to categorify the quotient category $\Perv_{\Lambda_H}(U)/\Loc(U)$ to a $\two$-category of microlocal perverse schobers, which we denote
\[\olsftwoP_{\Lambda_H}(U),\]
where as before, $H$ is a hypersurface germ defined on a small open subset $0\in U\subset \bbC^N_{x_1\dots,x_{N-1},y}$ by an equation of the form (\ref{equation:equation}).

From the discussion in \S\ref{subsection:muperv}, this should be thought of as giving a definition of microlocal perverse schobers defined near $\bbC^\times p$ for any $p\in\Lambda$ generic. For the rest of the introduction, we will assume that the base has dimension 2 so that $U=U_{x,y}$ and $H=C$ is a curve. We broadly partition our objective into three steps.

\begin{itemize}
    \item Step 1: Categorify $\Perv(\bbC,R)/\Loc(\bbC)$ to $\olsftwoP(\bbC,R)$, together with a projection functor $\sftwoP(\bbC,R)\to \olsftwoP(\bbC,R)$.
    \item Step 2: Categorify $\Perv_{\Lambda_C}(U)/\Loc(U)$ to $\olsftwoP_{\Lambda_C}(U)$.
    \item Step 3: Exhibit invariance under contact transformations.
\end{itemize}

Steps 1 and 2 constitute the construction of $\olsftwoP_{\Lambda_C}(U)$ via a categorification of the local GMV construction.

Step 3 is where our categorification is scrutinized. Recall that microlocal perverse sheaves are defined ``top down'' with contact invariance also inherited in this way, while the ``bottom up'' approach of GMV can be thought of as an alternative construction. In contrast, we take the ``bottom up'' approach as the \textit{definition} of microlocal perverse schobers, and to argue that this definition is a good one, we should exhibit contact invariance directly.

In this paper, we study a particularly interesting contact transform. We recall that there is a canonical contact transform
\[T^*\bbP^2 - 0_{\bbP^2}\simeq T^*\check{\bbP}^2 - 0_{\check{\bbP}^2},\]
where $\check{\bbP}^2$ is the dual projective space, parametrizing lines in $\bbP^2$. Moreover, if $C\subset \bbP^2$ is a projective curve, then we may consider its projective dual curve $\check{C}\subset\check{\bbP}^2$, parametrizing (limits of) tangent lines to $C$. Under the contact transform, $\Lambda_{C}$ identifies with $\Lambda_{\check{C}}$. In the perverse sheaf setting, this quantizes to the Radon transform \cite{kiehlLefschetzTheoryBrylinskiRadon2001},
\[\Perv_{0_{\bbP^2}\cup \Lambda_{C}}(\bbP^2)/\Loc(\bbP^2)\simeq \Perv_{0_{\check{\bbP}^2}\cup \Lambda_{\check{C}}}(\check{\bbP}^2)/\Loc(\check{\bbP}^2).\]

More locally, we may associate to our pair $(U_{x,y},C)$, a dual pair $(V_{a,b},\check{C})$, where $\check{C}\subset V_{a,b}$ is the germ of the curve near $(0,0)$ which parametrizes (limits of) tangent lines $y=ax-b$ to $C$. Figure~\ref{figure:radon} illustrates a pair of dual curves. We record below the resulting special case of Step 3.
\begin{itemize}
    \item Step 3': Construct a Radon transform equivalence, $\olsftwoP_{\Lambda_C}(U)\simeq \olsftwoP_{\Lambda_{\check{C}}}(V)$.
\end{itemize}
We note that the Radon transform is in some sense typical, and highlights difficulties that we may encounter for general contact transformations. In a related context, the Radon transform plays a key role in \cites{beilinson2016constructible,saitoSingularSupportsMixed2025} to develop a theory of singular support for \'etale constructible sheaves.

\begin{figure}[ht]
    \caption{A curve $C$ and its dual $\check{C}$}
    \label{figure:radon}

    \begin{tikzpicture}[scale=3]
    \clip (-1, -0.5) rectangle (1.2, 0.7);

    \draw[->, gray, line width=0.6pt] (-1, -0.4) -- (1.2, -0.4) node[above left, gray] {$x$};
    \draw[->, gray, line width=0.6pt] (-0.9, -0.5) -- (-0.9, 0.7) node[below right, gray] {$y$};


    \draw[black, thick, domain=-1:1, samples=100] plot (\x, {0.5*\x*\x*\x}) node[below, black] {$C$};

    \draw[red, thick, domain=0:1, samples=100] plot (\x, {1.5*0.5*0.5*\x-0.5*0.5*0.5});
    \draw[blue, thick, domain=-1:0, samples=100] plot (\x, {1.5*0.5*0.5*\x+0.5*0.5*0.5});
    \end{tikzpicture}
    \begin{tikzpicture}[scale=3]
    \clip (-1, -0.5) rectangle (1.2, 0.7);

    \draw[->, gray, line width=0.6pt] (-1, -0.4) -- (1.2, -0.4) node[above left, gray] {$a$};
    \draw[->, gray, line width=0.6pt] (-0.9, -0.5) -- (-0.9, 0.7) node[below right, gray] {$b$};


    \draw[black, thick, domain=-1:1, samples=100] plot (\x*\x, {0.5*\x*\x*\x}) node[below, black] {$\check{C}$};

    \fill[red] (0.5,0.5*0.5*0.707) circle (0.5pt);
    \fill[blue] (0.5,-0.5*0.5*0.707) circle (0.5pt);
    \end{tikzpicture}
\end{figure}
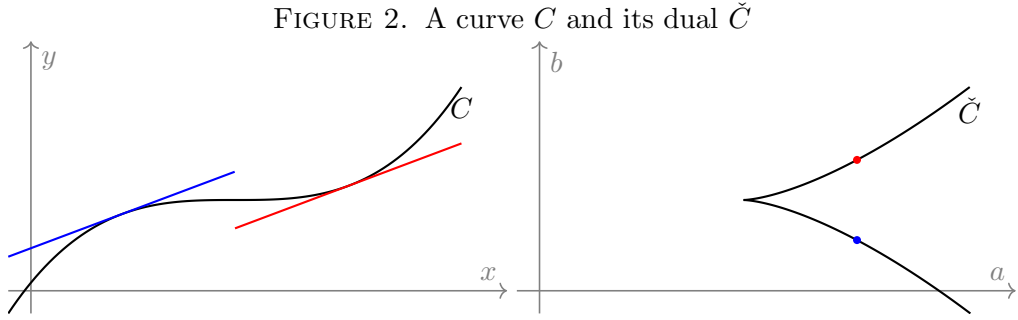

\subsection{Radon transform and symplectic suspension}
We illustrate the main new phenomenon which arises when we attempt this categorification procedure in a key example. We begin by considering two possible categorifications of $\Perv(\bbC,0)/\Loc(\bbC)$ at Step 1. We note that both of these are equipped with projection functors from $\sftwoP(\bbC,0)$.
\begin{enumerate}
    \item The first is the quotient $\two$-category $\olsftwoP^{\opname{monad}}(\bbC,0) := \sftwoP(\bbC,0)/\sftwoL(\bbC)$. It turns out that this is abstractly equivalent to the $\two$-category whose objects are pairs $(\Phi,M)$ where $\Phi$ is a stable $\one$-category, and $M$ is a monad on $\Phi$ such that $\coCone(\id\to M)$ is an autoequivalence.
    \item Recall that $\Perv(\bbC,0)/\Loc(\bbC)$ is abstractly equivalent to the category of pairs $(W,t)$ where $W$ is a $k$-vector space and $t$ is an automorphism of $W$. Our second categorification $\olsftwoP^{\opname{aut}}(\bbC,0)$ is the $\two$-category of pairs $(\Phi,T)$ where $\Phi$ is a stable $\one$-category and $T$ is an autoequivalence of $\Phi$.
\end{enumerate}

Let $C=\{y=x^2\}$. For this curve, Step 2 is easy. Given our categorification $\olsftwoP(\bbC,0)$ from Step 1, the categorified local GMV construction simply defines
\[\olsftwoP_{\Lambda_C}(U) := \olsftwoP(\bbC,0),\]
because $C$ is smooth. The dual curve $\check{C} = \{b=a^2/4\}$ is also smooth so,
\[\olsftwoP_{\Lambda_{\check{C}}}(V) := \olsftwoP(\bbC,0).\]
Thus the Radon transform of Step 3' is identified with an equivalence of $\two$-categories, which we denote
\[\calRadon:\olsftwoP(\bbC,0)\simeq \olsftwoP(\bbC,0).\]
We emphasize that the Radon transform happens on the 2-dimensional geometry of $(U,C)$ and not the 1-dimensional geometry of $(\bbC,0)$.

While these two $\two$-categories are obviously equivalent, we claim that the Radon transform ought to be a different operation which we call symplectic suspension. Assuming that $k=\bbC$, let us denote by $\ol{\calF}(M,f) \in \olsftwoP(\bbC,0)$ the image of the Lefschetz schober associated to a Lefschetz fibration $(M,f)$, under the projection functor constructed in Step 1. We posit that the Radon transform should agree with the symplectic suspension operation,
\[\calRadon:\ol{\calF}(M,f)\mapsto \ol{\calF}(M\times\bbC_z,f+z^2).\]
In \S\ref{subsection:suspension}, we define an algebraic version of the symplectic suspension (general to any $k$), as a functor $\calSus: \sftwoP(\bbC,0)\to \sftwoP(\bbC,0)$, which conjecturally satisfies
\[\calSus:\calF(M,f)\mapsto \calF(M\times\bbC_z,f+z^2).\]
So more precisely, we posit that the Radon transform is compatible with $\calSus$ via projection functors as illustrated in the following commutative diagram.

\[\begin{tikzcd}[ampersand replacement=\&]
	{\sftwoP(\bbC,0)} \& {\sftwoP(\bbC,0)} \\
	{\olsftwoP(\bbC,0)} \& {\olsftwoP(\bbC,0)}
	\arrow["\calSus", from=1-1, to=1-2]
	\arrow[from=1-1, to=2-1]
	\arrow[from=1-2, to=2-2]
	\arrow["{\calRadon}", from=2-1, to=2-2]
\end{tikzcd}\]

If we choose the first categorification, then $\calRadon$ is determined to be the functor
\begin{align*}
    \calSus_{\opname{monad}}:\olsftwoP^{\opname{monad}}(\bbC,0)&\to \olsftwoP^{\opname{monad}}(\bbC,0)\\
    (\Phi,M)&\mapsto (\Phi,\id_\Phi\times_M\id_\Phi)
\end{align*}
However, this is not an equivalence as some data of the monad is lost. If instead we choose the second categorification, then $\calRadon$ is determined to be the functor
\begin{align*}
    \calSus_{\opname{aut}}:\olsftwoP^{\opname{aut}}(\bbC,0)&\to \olsftwoP^{\opname{aut}}(\bbC,0)\\
    (\Phi,T)&\mapsto (\Phi,T[-1])
\end{align*}
which is clearly an equivalence. So we learn that $\olsftwoP^{\opname{aut}}(\bbC,0)$ is the correct categorification at Step 1. We will prove in Theorem~\ref{thm:limit0} that in fact, we may recover the correct categorification from the incorrect one by inverting $\calSus_{\opname{monad}}$ as follows.

\begin{thm}\label{thm:introcolimit}
    There is a natural diagram of $\two$-categories,
    \[\olsftwoP^{\opname{aut}}(\bbC,0)\to \cdots\to \xto{\calSus_{\opname{monad}}}\olsftwoP^{\opname{monad}}(\bbC,0)\xto{\calSus_{\opname{monad}}}\olsftwoP^{\opname{monad}}(\bbC,0)\]
    which exhibits $\olsftwoP^{\opname{aut}}(\bbC,0)$ as a limit.
\end{thm}

Concretely, this says that the part of the monad $M$ stable under $\calSus_{\opname{monad}}$ is precisely the autoequivalence $T:=\coCone(\id\to M)$. We summarize as follows.

\begin{moral}\label{moral:intromain}
    The correct categorification of the quotient ${\Perv}(\bbC,0)/Loc(\bbC)$ is not the quotient of the categorifications, but we can correct this by inverting the symplectic suspension in the manner above.
\end{moral}

This is reflected in our choice to denote the categorification by $\olsftwoP(\bbC,0)$ instead of $\ol{\sftwoP}(\bbC,0)$.

\subsection{Radon transform for \texorpdfstring{$y^m=x^n$}{y\^{}m=x\^{}n}}
The main example of Radon transform that we study in this paper is for the curve $C=\{y^m = x^n\}$ where $0<m<n$. The dual curve is given by
\[\check{C} = \{\left(\frac{m}{n-m}b\right)^{n-m} = \left(\frac{m}{n}a\right)^n\}.\]
Ignoring constants, this is $\{b^{n-m} = a^n\}$, a curve of the same form with $m$ replaced by $n-m$. This curve has been studied in various related contexts, for example, \cites{oblomkovHilbertSchemePlane2012,macphersonPerverseSheavesSingularities1988}.

Let us view $(\bbC^2_{x,y},C)$ as a family of discs. Over $x=1$, we see the disc $(\bbC_y,\{y^m=1\})$. As $x$ makes one revolution around the unit circle, the points cyclically rotate through an angle of $2\pi n/m$, forming the $(n,m)$-torus link. This geometry is reflected in the following description of $\olsftwoP(\bbC^2_{x,y},C)$; its objects are given by
\begin{enumerate}
    \item a stable $\infty$-category $\Phi$,
    \item an $n$-periodic $m$-term $\infty$-admissible SOD $\Phi = \langle\Phi_1,\dots,\Phi_m\rangle$, and
    \item\label{item:ignore} an autoequivalence $T:\Phi\to \Phi$ which is ``conjugate-compatible with respect to this SOD''.
\end{enumerate}
We will ignore (\ref{item:ignore}) for the rest of this discussion. Consider the mutations,
\[\langle\Phi_1,\dots,\Phi_m\rangle = \langle\Phi_2,\dots,\Phi_{m+1}\rangle = \langle\Phi_3,\dots,\Phi_{m+2}\rangle\cdots.\]
The SOD is said to be $n$-periodic if $\Phi_{i} = \Phi_{i+n}$ for all $i\geq1$. When $m=2$, this specializes to the notion introduced by Dyckerhoff--Kapranov--Schechtman \cite{dyckerhoffNsphericalFunctorsCategorification2023}.

Repeating this analysis for the dual curve $\check{C}$ and continuing to ignore autoequivalences, we formulate the following variant of the Radon transform for $C$. The proof can be extracted from that of the original Radon transform for $C$, which appears in the main text as Theorem~\ref{thm:ykxnschober}.
\begin{thm}\label{thm:radonsod}
    Let $\sfSOD(m)^{n\mathrm{-periodic}}$ denote the collection of pairs $(\Phi,\frakS)$ consisting of a stable $\one$-category $\Phi$ and an $m$-term SOD $\frakS$ which is $\infty$-admissible and $n$-periodic. Then, there is a natural bijection,
    \[\sfSOD(m)^{n\mathrm{-periodic}}\simeq \sfSOD(n-m)^{n\mathrm{-periodic}}.\]
\end{thm}

\begin{eg}\label{eg:introx3y2}
    Let $n=3$ and $m=2$. By \cite{dyckerhoffNsphericalFunctorsCategorification2023}, $3$-periodic 2-term SODs are those for which the gluing functor is an equivalence, thus are described by a single category. On the other hand, a 3-periodic 1-term SOD is also just a single category.
\end{eg}

It is proposed in \cite{dyckerhoffNsphericalFunctorsCategorification2023} that forming iterated mutations of a 2-term SOD is a categorical counterpart of the procedure of forming continued fractions, which in turn are related to iterated products of matrices of the form
\[\begin{bmatrix}
0 & -a_0\\
1 & -a_1
\end{bmatrix}.\]
Extending this slogan, Kapranov (in personal communication) proposed that forming iterated \textit{cyclic} mutations of an $m$-term SOD is a categorical counterpart of forming $m$-ary generalized continued fractions as studied by Jacobi \cites{bernstein2006jacobi,lehmerJacobisExtensionContinued1918}, which in turn are related to iterated products of $m\times m$ companion matrices,
\[\begin{bmatrix}
0 & \cdots & 0 & -a_0\\
1 & \cdots & 0 & -a_1\\
\vdots & \ddots & \vdots & \vdots\\
0 & \cdots & 1 & -a_{m-1}
\end{bmatrix}.\]
With this in mind, we have the following decategorification of Theorem~\ref{thm:radonsod}.

\begin{prop}
    Let $k$ be a field. Let $P(m,n)$ denote the set of $n$-tuples $(A_1,\dots,A_n)$ of $m\times m$ companion matrices that satisfy $A_n\cdots A_2A_1 = I$. Then, there is a natural bijection,
    \[P(m,n)\simeq P(n-m,n).\]
\end{prop}

\begin{proof}
    We show that there is a bijection $P(m,n)\simeq \Pi(m,n)$ where $\Pi(m,n)\subset \opname{Gr}(m,n)$ is the open positroid locus. These are those objects for which when represented by an $n\times m$ matrix, any $m$ cyclically consecutive rows have a non-zero minor.
    
    Denote by $e_1\in k^m$ the first standard basis vector and define recursively $e_{i+1} = A_ie_i$ for $i\geq 1$. See that $e_1,\dots,e_m$ is the standard basis, that $e_{i+n}=e_i$ for all $i$, and that any $m$ consecutive vectors in this sequence form a basis. Thus the assignment
    \[(A_1,\dots,A_n)\mapsto [e_1|e_2|\cdots |e_n]^{\intercal}\]
    defines a map $P(m,n)\to \Pi(m,n)$, which is easily seen to be a bijection. We are done by recalling the natural bijection
    \[\Pi(m,n)\simeq \Pi(n-m,n)\]
    by taking the orthogonal subspace under the standard bilinear form on $k^n$.
\end{proof}

We note that $\Pi(m,n)$ also appears as a moduli of microlocal sheaves supported on a Legendrian torus link in $\bbR^2$ as explained in \cite{shendeClusterVarietiesLegendrian2019}. Work of Kuwagaki \cite{kuwagakiCategorificationLegendrianKnots2019} studies perverse schobers in the real microlocal setting, where $n$-periodic SODs naturally arise from the same Legendrian torus links. This picture is related to $(\bbC^2,C)$ via the Fourier transform of Kapranov--Soibelman--Soukhanov \cite{kapranovPerverseSchobersAlgebra2020}. In fact, this is the geometry which implicitly underlies the proof of Theorem~\ref{thm:radonsod}, presented in \S\ref{subsection:ykxn}. Ongoing work explores this relation further.

\subsection{Main constructions and results}
We first outline our main constructions. We recall from \cite{kapranovPerverseSchobers2015} that to define perverse schobers on $(\bbC,R)$, we make a choice of cuts and write down a certain diagram of categories, which we shall call an abstract perverse schober. An action of the braid group on the set of abstract schobers encodes how diagrams corresponding to different choices of cuts are related, leading to a definition independent of the choice of cuts.

Upgrading this, we will construct $\two$-categories together with functors and compatible braid group actions,
\[\sfSph(n)\to\sfSphMnd(n)\to\sfAut(n).\]
such that a $\two$-categorical version of the construction above yields $\two$-categories,
\[\sftwoP(\bbC,R)\to \olsftwoP^{\opname{monad}}(\bbC,R)\to\olsftwoP^{\opname{aut}}(\bbC,R),\]
where $\olsftwoP^{\opname{monad}}(\bbC,R)\simeq \sftwoP(\bbC,R)/\sftwoL(\bbC)$. We refer to objects of these $\two$-categories as (abstract) perverse schobers, prelocalized perverse schobers and localized perverse schobers respectively.

During this process, we will make heavy use of factorization structures on these $\two$-categories as a means to make inductive arguments. To encode these structures, we will in fact construct simplicial $\two$-categories
\[\sfSph(-)\to\sfSphMnd(-)\to\sfAut(-)\to \sfSOD(-)\]
satisfying the 2-Segal property \cites{dyckerhoffHigherSegalSpaces2019,galvez-carrilloDecompositionSpacesIncidence2018}, so that $\sfSph([n]) = \sfSph(n)$ and so on. Here, $\sfSOD(n)$ is a $\two$-category whose objects are given by categories equipped with an $n$-term SOD. 

Our first main result is the following generalization of Theorem~\ref{thm:introcolimit}, which appears as Theorem~\ref{thm:limit0} in the main text.
\begin{thm}\label{thm:introcolimit2}
    There is a natural diagram of $\two$-categories,
    \[\olsftwoP^{\opname{aut}}(\bbC,R)\to \cdots\xto{\calSus_{\opname{monad}}}\olsftwoP^{\opname{monad}}(\bbC,R)\xto{\calSus_{\opname{monad}}}\olsftwoP^{\opname{monad}}(\bbC,R)\]
    which exhibits $\olsftwoP^{\opname{aut}}(\bbC,R)$ as a limit. Here, $\calSus_{\opname{monad}}$ is a certain generalization of the symplectic suspension functor from earlier.
\end{thm}

We use ${\olsftwoP}^{\opname{aut}}(\bbC,R)$ to define $\two$-categories $\olsftwoP_{\Lambda_C}(U)$ by proposing a categorification of the local GMV construction. Our main conjecture is that the resulting $\two$-categories admit a Radon transform.

\begin{conj}\label{conj:introradon}
    Let $C\subset \bbC^2_{x,y}$ be the germ of a curve near $(0,0)$ with defining equation $y^m + \text{ higher order terms}$, and $\check{C}\subset \bbC^2_{a,b}$ its dual. Then, there is an equivalence of $\two$-categories
    \[\olsftwoP_{\Lambda_C}(U_{x,y})\simeq \olsftwoP_{\Lambda_{\check{C}}}(V_{a,b}),\]
    where $U_{x,y}$ and $V_{a,b}$ are small open subsets near $(0,0)$.
\end{conj}

Our second main result is a partial proof of this conjecture for the example $\{y^m=x^n\}$, which appears as Theorem~\ref{thm:ykxnschober} in the main text, and which we discussed a variant of in Theorem~\ref{thm:radonsod}.
\begin{thm}\label{thm:introykxn}
    Let $0<m<n$ and let $C\subset\bbC^2_{x,y}$ be the curve $\{y^m=x^n\}$. For this curve, Conjecture~\ref{conj:introradon} holds at the level of objects.
\end{thm}

\subsection{Relation to other work}
We describe relations to other works on perverse schobers and related areas.

This paper has been greatly inspired by the work of Kapranov--Soibelman--Soukhanov \cite{kapranovPerverseSchobersAlgebra2020}, where localized perverse schobers on $(\bbC,R)$ are studied implicitly. The factorization structure on $\sfSph(-)$ builds on their construction of the Fukaya--Seidel category of a schober on a disc. In addition, we will interpret their Fourier transform in our context by observing that a localized perverse schober on $(\bbC,R)$ is precisely a local system of categories on $\bbC^{\vee}-\{0\}$ together with $R$-Stokes data.

The quotient category $\sftwoP(\bbC,0)/\sftwoL(\bbC)$ is described in terms of spherical monads, which were studied by Christ \cite{christSphericalMonadicAdjunctions2023}. The implicit characterization of spherical monads \cite{christSphericalMonadicAdjunctions2023}*{Theorem 4.1} is sharpened in Theorem~\ref{thm:christstrengthening1}, and generalized to allow for multiple singularities in Appendix~\ref{appendix:sphmonad}. These provide models for prelocalized perverse schobers.

The work of Gammage--Hilburn \cite{gammageHypertoric2categoriesSymplectic2025} also studies microlocal aspects of perverse schobers in the context of 3d mirror symmetry. While we adopt the perspective that the quotient category $\sftwoP(\bbC,R)/\sftwoL(\bbC)$ is incorrect for our purposes, it appears to be the correct category in their context. Indeed, they point out in \cite{gammageHypertoric2categoriesSymplectic2025}*{\S 0.5} that a key feature of their $\two$-category of microlocal perverse schobers is its dependence not just on $\Omega$, but also on the embedding $\Omega\inclto T^*X$. Moreover, it follows from Theorem~\ref{thm:christstrengthening1} that there is an equivalence of $\two$-categories,
\[\sftwoP(\bbC,0)/\sftwoL(\bbC)\simeq \opname{2QCoh}(\bbA^1/\bbG_m).\]
We compare this to the 3d mirror symmetry from their work with Mazel-Gee \cite{gammagePerverseSchobers3d2023}*{Theorem B}, formulated as below in \cite{gammageBettiTatesThesis2025}*{Theorem 1.1}.
\[\sftwoP(\bbC,0)\simeq \opname{2IndCoh}_{\mathbb{L}}(\bbA^1/\bbG_m).\]
We expect that these two statements are compatible in a manner analogous to the compatibility between geometric Langlands and its tempered version, see for example \cite{faergemanNonvanishingGeometricWhittaker2022}*{1.6.2}.

\subsection{Conventions}
Our perverse sheaves are assumed to have $k$-coefficients where $k$ is a commutative ring spectrum. Our perverse schobers will be sheaves of stable presentable $\one$-categories, linear over $k$. In particular, our stable $\one$-categories will be large. This difference is important during one example which uses Proposition~\ref{prop:KLisEM}. However we expect that most of the theory admits a small analog. We outline $\two$-categorical conventions in \S\ref{subsection:catthy}.

\subsection{Organization of paper}
In each of \S\ref{section:sod}--\S\ref{section:radon}, we place the main ideas, definitions and results at the beginning of the section, with technical details placed later. We will also include a brief summary below.

In \S\ref{section:prelim}, we collect preliminaries used throughout the paper. We will recall some higher category theory, and state a theorem of Abell\'an--Haugseng--Martini \cite{abellanFreeFibrationsLax2026} which will provide a source of adjoint functors between $\two$-categories. We also study 2-faithful functors, which will provide our other source of functors. We then discuss simplicial objects, and recall the 2-Segal property which encodes factorization structures. Finally, we discuss the theory of adjunctions and monads focusing on the stable, presentable setting. We prove that there is an analog of the Kleisli construction, which surprisingly agrees with the Eilenberg--Moore construction.

In \S\ref{section:sod}, we set up the conventions for semiorthogonal decompositions, and then construct the $\two$-category $\sfSOD(n)$ whose objects are categories equipped with an $n$-term SOD. We then upgrade this to a 2-Segal simplicial $\two$-category, encoding standard factorization structures on SODs.

In \S\ref{section:localized}, we introduce the theory of localized perverse schobers. We begin by defining \textit{abstract} localized perverse schobers (in the same way that spherical adjunctions can be thought of as abstract perverse schobers) and then upgrade this to a $\two$-category with factorization structures and a braid group action. We use the braid group action to define localized perverse schobers.

In \S\ref{section:adjaut}, we relate localized schobers to ordinary schobers on a disc. In doing so, we will construct factorization structures and a braid group action on ordinary schobers in the $\two$-categorical setting. We use the braid group action to define $\two$-categories of ordinary perverse schobers. We then give an interpretation of the Fourier transform for perverse schobers in our setting.

In \S\ref{section:prelocalized}, we study the quotient category $\sftwoP(\bbC,R)/\sftwoL(\bbC)$, whose objects we will refer to as prelocalized schobers. We then relate this to spherical monads.

In \S\ref{section:pretolocalized}, we define the symplectic suspension operation $\calSus$ on prelocalized schobers, and prove that by stabilizing this, we recover localized perverse schobers, thereby making precise the sense in which localized schobers arise from schobers.

In \S\ref{section:families}, we introduce the local GMV construction, categorifying \cite{gelfandMicrolocalPerverseSheaves2005}*{Proposition 3.3}, in order to define ordinary/localized/prelocalized schobers on $(\bbC^2_{x,y},C)$. We relate these to periodic SODs and study in detail perverse schobers on $(\bbC^2,\{y^2=x^3\})$, comparing them to \cites{macphersonPerverseSheavesSingularities1988,dyckerhoffPerverseSchobersCoxeter2025}, as well as to the $A_2$-configurations of Seidel--Thomas \cite{seidelBraidGroupActions2000}.

In \S\ref{section:radon}, we state the Radon transform conjecture. We study in detail the Radon transform on $(\bbC^2,\{y=x^2\})$, and relate this to the symplectic suspension. We then prove the conjecture for $\{y^m=x^n\}$ in an explicit manner, at the level of objects.

In Appendix~\ref{appendix:sphmonad}, we study spherical monads abstractly. We strengthen a theorem of Merlin Christ \cite{christSphericalMonadicAdjunctions2023}*{Theorem 4.1}, and generalize this theorem to allow multiple singularities instead of just one.

\subsection{Acknowledgments}
This paper is part of an ongoing project in collaboration with Peng Zhou. I thank him for many explanations and interpretations from the perspective of symplectic geometry, which benefitted my understanding greatly.

I thank my PhD advisor David Nadler for many helpful discussions, for proposing the problem of studying perverse schobers microlocally, and for his suggestions which improved the readability of the Introduction.

I thank Merlin Christ for suggesting a simplification in the proof of Theorem~\ref{thm:christstrengthening1}, and for helpful discussions on $\two$-categorical braid group actions. We note that a $\two$-categorical braid group action on $\sfSph(n)$ was constructed first, and in greater generality in his work in preparation with Fernando Abell\'an and Gustavo Jasso \cite{acj}, which was announced during the summer school ``Perverse Sheaves and their Categorification in Geometry and Representation Theory'', Venice, July 2026. This announcement provided the inspiration to make the present work more rigorously $\two$-categorical.

I also thank Benjamin Gammage, Swapnil Garg, Justin Hilburn, Josephine Hlavinka, Mikhail Kapranov, Tatsuki Kuwagaki, and Paul Wedrich for helpful conversations during the preparation of this paper. This work was supported by the Simons Dissertation Fellowship and the James Simons Fellowship in Mathematics.

Google Gemini 3.6 Thinking was used to help identify the 2-Segal property as an appropriate way to encode factorization structures, to generate tikz diagrams, and to check for typographical errors in the paper.

\section{Preliminaries}\label{section:prelim}
We recap some preliminaries on $\two$-categories, simplicial objects, Segal properties, and monads and adjunctions. 

\subsection{Higher category theory}\label{subsection:catthy}
We collect together some higher category theory we use throughout the paper. The properties we need are largely independent of the choice of a model of $\one$ or $\two$-category. We assume to have fixed nested Grothendieck universes which we refer to as small, large and very large.

All $\one$-categories appearing as the coefficients of our perverse schobers will be stable presentable large categories. All functors between $\one$-categories are assumed to be colimit-preserving, or equivalently admit (a possibly non-colimit-preserving) right adjoint. We denote by $\sfSt_k$ the $\two$-category of stable large presentable categories with colimit-preserving functors and natural transformations.

We denote by $\wCat_{\two}$ the $\one$-category of large $\two$-categories. As pointed out in \cite{stefanichPresentable$inftyN$categories2020}, $\sfSt_k$ is only large, roughly because presentable categories are determined by a small amount of data. Thus $\sfSt_k\in \wCat_{\two}$ defines an object.

Given a $\two$-category $\sfA\in\wCat_{\two}$, and objects $\calA_1,\calA_2\in\sfA$, we denote by $\sfA(\calA_1,\calA_2)$ the mapping category. Given also $\sfB\in\wCat_{\two}$, we say that a pair of functors $\opname{L}:\sfA\tofrom \sfB:\opname{R}$ are adjoint if there exist functorial equivalences of mapping categories,
\[\sfA(\calA,\opname{R}(\calA))\simeq \sfB(\opname{L}(\calA),\calB).\]

Given a functor $\opname{I}:\sfA\to \sfB$, we say that
\begin{itemize}
    \item $\opname{I}$ is 2-fully faithful if $\opname{I}$ induces an equivalence on all mapping categories,
    \item $\opname{I}$ is 2-faithful if $\opname{I}$ induces a fully faithful functor on all mapping categories,
    \item $\opname{I}$ is a locally full inclusion if $\opname{I}$ is 2-faithful and induces a monomorphism on equivalence classes of objects.
\end{itemize}

Whenever $\frakd$ is a small $\two$-category, there is a $\two$-category $\sfFun(\frakd,\sfSt_k)\in \wCat_{\two}$ whose objects are functors $\frakd\to \sfSt_k$. Whenever $f:\frakd_1\to \frakd_2$ is a functor between $\two$-categories, there is a restriction functor,
\[f^*:\sfFun(\frakd_2,\sfSt_k)\to \sfFun(\frakd_1,\sfSt_k).\]

All $\two$-categories of schobers will be constructed out of these diagram categories $\sfFun(\frakd,\sfSt_k)\in \wCat_{\two}$ by taking locally full sub-$\two$-categories. All functors will be constructed using either Theorem~\ref{thm:laxfibrations} or Proposition~\ref{prop:functorconstruction}, which we discuss below.

\begin{rmk}
    We expect that all $\two$-categories of schobers are ``$k$-linear stable presentable $\two$-categories'' and that all functors respect this structure. A potential candidate for such theory is laid out in \cite{stefanichPresentable$inftyN$categories2020}. In this paper, we will content ourselves with working inside $\wCat_{\two}$.
\end{rmk}

\begin{thm}[Abell\'an--Haugseng--Martini, \cite{abellanFreeFibrationsLax2026}*{Theorem 5.6.5}]\label{thm:laxfibrations}
    The functor $f^*$ admits both adjoints $f_!\adjto f^*\adjto f_*$ which each have explicit descriptions on objects.
\end{thm}

Our other method of constructing functors is as follows.

\begin{prop}\label{prop:functorconstruction}
    Let $\sfC,\sfD,\sfS\in\wCat_{\two}$, and let $\opname{F}:\sfC\to\sfS$, and $\opname{G}:\sfD\to\sfS$ be functors, so that we have lifts $\sfC,\sfD\in (\wCat_{\two})_{/\sfS}$. Assume that $\opname{G}$ is 2-faithful. Then, the mapping space
    \[\Map_{(\wCat_{\two})_{/\sfS}}(\sfC,\sfD)\]
    is discrete. Moreover, any functor $\sfC\to\sfD$ over $\sfS$ is determined by the map it induces on the sets of equivalence classes of objects.
\end{prop}

We note that the proof of the proposition will also give us a criterion under which a given map on equivalence classes of objects comes from a functor. We will need the following standard facts.

\begin{lem}\label{lem:ftoff}
    If a functor $\opname{G}:\sfD\to\sfS$ is 2-faithful, then the diagonal functor $\opname{Diag}:\sfD\to\sfD\times_{\sfS}\sfD$ is 2-fully faithful.
\end{lem}

\begin{proof}
    Let $\calD_1,\calD_2\in\sfD$. We have an identification
    \[(\sfD\times_{\sfS}\sfD)(\opname{Diag}\calD_1,\opname{Diag}\calD_2)\simeq \sfD(\calD_1,\calD_2)\times_{\sfS(\opname{G}\calD_1,\opname{G}\calD_2)} \sfD(\calD_1,\calD_2),\]
    and the functor $\sfD(\calD_1,\calD_2)\to (\sfD\times_{\sfS}\sfD)(\opname{Diag}\calD_1,\opname{Diag}\calD_2)$ is the diagonal functor. The diagonal of a fully faithful functor of $\one$-categories is an equivalence, so we are done.
\end{proof}

Lemma~\ref{lem:ftoff} allows us to reduce statements about 2-faithful functors to those about 2-fully faithful ones which are much easier to think about. In the following, we recall that for $k\geq -1$ a map $f$ between spaces is said to be $k$-truncated if the diagonal of $f$ is $(k-1)$-truncated, and that a map $f$ is said to be $(-2)$-truncated if it is an equivalence. So a $(-1)$-truncated map is an inclusion of connected components, and a $0$-truncated map is one which is relatively discrete.

\begin{lem}\label{lem:discrete1}
    Let $\frakt\in\wCat_{\two}$ be any $\two$-category. Then, the post-composition map on mapping categories
    \[\opname{G}\circ-:\Map(\frakt,\sfD)\to \Map(\frakt,\sfS)\]
    is $0$-truncated.
\end{lem}

\begin{proof}
    This is equivalent to showing that the diagonal map of $\opname{G}\circ-$ is $(-1)$-truncated. This diagonal map is identified with $\opname{Diag}\circ-$, and by Lemma~\ref{lem:ftoff}, $\opname{Diag}$ is 2-fully faithful. We are done by observing that $\opname{Diag}\circ-$ is $(-1)$-truncated. In words, a functor $\frakt\to \sfD\times_{\sfS}\sfD$, if it factors through the full sub-$\two$-category $\sfD$, does so uniquely.
\end{proof}

\begin{lem}\label{lem:discrete2}
    Let $\frakt'\to\frakt$ be a functor of $\two$-categories which is essentially surjective. Then, the natural map
    \[\Map(\frakt,\sfD)\to \Map(\frakt',\sfD)\times_{\Map(\frakt',\sfS)}\Map(\frakt,\sfS)\]
    is $(-1)$-truncated.
\end{lem}

\begin{proof}
    This is equivalent to showing that the diagonal is $(-2)$-truncated, i.e. an equivalence. This is identified with the natural map
    \[\Map(\frakt,\sfD)\to \Map(\frakt',\sfD)\times_{\Map(\frakt',\sfD\times_{\sfS}\sfD)}\Map(\frakt,\sfD\times_{\sfS}\sfD).\]
    By Lemma~\ref{lem:ftoff}, $\opname{Diag}$ is 2-fully faithful, so that this is indeed an equivalence. In words, a functor $\frakt\to \sfD\times_{\sfS}\sfD$ factors through $\sfD$ after restricting to $\frakt'$, if and only if it factors through $\sfD$.
\end{proof}

\begin{proof}[Proof of Proposition~\ref{prop:functorconstruction}]
    By Yoneda lemma, to give a functor $\sfC\to \sfD$ over $\sfS$ is equivalent to giving for each $\frakt\in\wCat_{\two}$ a commuting diagram of spaces,
    \[\begin{tikzcd}[ampersand replacement=\&]
        {\Map(\frakt,\sfC)} \& {\Map(\frakt,\sfD)} \\
        \& {\Map(\frakt,\sfS)}
        \arrow[dashed, from=1-1, to=1-2]
        \arrow["\opname{F}\circ-"', from=1-1, to=2-2]
        \arrow["\opname{G}\circ-", from=1-2, to=2-2]
    \end{tikzcd}\]
    which is functorial in $\frakt$.
    
    By Lemma~\ref{lem:discrete1}, the arrow $\opname{G}\circ-$ is relatively discrete, and hence the space of lifts of $\opname{F}\circ-$ is discrete.

    Suppose we exhibit a lift for the trivial $\two$-category $\frakt=\mathfrak{pt}$. This determines the lift for $\frakt$ any discrete space, as well as the functoriality for any map between discrete spaces.
    
    We show that in fact, the lift is uniquely determined for all $\frakt\in \wCat_{\two}$, if it exists. To see this, let $\frakt'\to\frakt$ be any essentially surjective functor from any discrete space $\frakt'$. Functoriality gives us a commutative diagram,
    \[\begin{tikzcd}[ampersand replacement=\&]
        \& {\Map(\frakt',\sfC)} \& {\Map(\frakt',\sfD)} \\
        {\Map(\frakt,\sfC)} \& {\Map(\frakt,\sfD)} \& {\Map(\frakt',\sfS)} \\
        \& {\Map(\frakt,\sfS)}
        \arrow[from=1-2, to=1-3]
        \arrow[from=1-2, to=2-3]
        \arrow[from=1-3, to=2-3]
        \arrow[from=2-1, to=1-2]
        \arrow[dashed, from=2-1, to=2-2]
        \arrow[from=2-1, to=3-2]
        \arrow[from=2-2, to=1-3]
        \arrow[from=2-2, to=3-2]
        \arrow[from=3-2, to=2-3]
    \end{tikzcd}\]
    The desired lift denoted by the dotted arrow is uniquely determined if it exists, by Lemma~\ref{lem:discrete2} applied to the square on the right of this diagram.
    
    Suppose now that lifts exist for all $\frakt\in \wCat_{\two}$. We claim that the functoriality of these lifts is automatic. Let $\frakt_1\to \frakt_2$ be any map of $\two$-categories. Choose any map $\frakt_1'\to \frakt_2'$ of discrete spaces together with a commutative diagram as follows, such that the vertical maps are essentially surjective.
    \[\begin{tikzcd}[ampersand replacement=\&]
        {\frakt_1'} \& {\frakt_2'} \\
        {\frakt_1} \& {\frakt_2}
        \arrow[from=1-1, to=1-2]
        \arrow[from=1-1, to=2-1]
        \arrow[from=1-2, to=2-2]
        \arrow[from=2-1, to=2-2]
    \end{tikzcd}\]
    A similar argument to before produces the desired commuting diagram
    \[\begin{tikzcd}[ampersand replacement=\&]
        {\Map(\frakt_1,\sfC)} \& {\Map(\frakt_1,\sfD)} \\
        {\Map(\frakt_2,\sfC)} \& {\Map(\frakt_2,\sfD).}
        \arrow[from=1-1, to=1-2]
        \arrow[from=2-1, to=1-1]
        \arrow[from=2-1, to=2-2]
        \arrow[from=2-2, to=1-2]
    \end{tikzcd}\]
    Higher functoriality is similarly automatic. We deduce that $\Map_{(\wCat_{\two})_{/\sfS}}(\sfC,\sfD)$ embeds as a subset of the set of lifts for $\frakt=\mathfrak{pt}$, proving the proposition.
\end{proof}

\subsection{Simplicial objects}\label{subsection:simplicialintro}
We denote by $\Delta$ the simplex category, whose objects are $[n] := \{0,\dots,n\}$ for $n\geq0$ and whose morphisms are non-decreasing functions $f:[m]\to[n]$. We consider the following standard set of generating morphisms.
\begin{itemize}
    \item For each $n\geq1$ let $\delta_{n,j}:[n-1]\to[n]$ be the injective map skipping $j$.
    \item For each $n\geq0$ let $\sigma_{n,j}:[n+1]\to[n]$ be the surjective map hitting $j$ twice.
\end{itemize}

We say that a morphism $f:[m]\to [n]$ is active if $f(0)=0$ and $f(m)=n$. We say that a morphism $f:[m]\to [n]$ is inert if $f$ is injective and its image is convex under the partial order. These two classes of morphisms are each closed under composition and so define subcategories $\Delta_{\act}$ and $\Delta_{\inert}$ respectively. Note that in $\Delta_{\act}$, $[1]$ is an initial object. These two classes of morphisms in fact form a factorization system of $\Delta$, that is, every morphism $f:[m]\to [n]$ uniquely factors as $i\circ a$ where $a$ is active and $i$ is inert.

Let $\calC$ be a $\one$-category with fibered products. A simplicial object $C(-)$ in $\calC$ is a functor
\[C:\Delta^{\op}\to \calC.\]
We often write $C(n) := C([n])$, as well as
\begin{align*}
    d_{n,j} &:= C(\delta_{n,j}): C(n)\to C(n-1)\\
    s_{n,j} &:= C(\sigma_{n,j}): C(n)\to C(n+1).
\end{align*}

A simplicial object $C(-)$ is said to be 1-Segal if for every pushout square of inert maps
\[\begin{tikzcd}[ampersand replacement=\&]
	{[n]}\pullback \& {[n-m]} \\
	{[m]} \& {[0]}
	\arrow[from=1-2, to=1-1]
	\arrow[from=2-1, to=1-1]
	\arrow[from=2-2, to=1-2]
	\arrow[from=2-2, to=2-1]
\end{tikzcd}\]
the resulting square of objects in $\calC$ is a pullback square. By definition, this upgrades $C(1)$ to a monoidal object in the slice category $\calC_{/C(0)}$.

Many of our simplicial objects will be 2-Segal, a notion introduced by Dyckerhoff--Kapranov \cite{dyckerhoffHigherSegalSpaces2019}, and independently under the name decomposition spaces by G\'alvez-Carrillo--Kock--Tonks \cite{galvez-carrilloDecompositionSpacesIncidence2018}.

A simplicial object $C(-)$ is said to be 2-Segal if for every pushout square of maps with $a$ active and $i$ inert (which implies $a'$ active and $i'$ inert)
\[\begin{tikzcd}[ampersand replacement=\&]
	{[n]} \& {[n+1-m]} \\
	{[m]} \& {[1]}
	\arrow["{i'}"', from=1-2, to=1-1]
	\arrow["{a'}", from=2-1, to=1-1]
	\arrow["a", from=2-2, to=1-2]
	\arrow["i"', from=2-2, to=2-1]
\end{tikzcd}\]
the resulting square of objects in $\calC$ is a pullback square.

\begin{eg}\label{eg:waldhausen}
    The prototypical example of such an object arises from the Waldhausen S-construction, of which the following is a simple instance. Let $\calC$ be the category of sets, and let $C(n)$ denote the set of filtered vector spaces $0=V_0\subset V_1\subset\cdots \subset V_n$. We first define the face maps
    \[d_{n,i}:C(n)\to C(n-1).\]
    For $i=0$, this outputs $V_1/V_1\subset V_2/V_1\subset\cdots\subset V_n/V_1$, for $i=n$, this outputs $0=V_0\subset V_1\cdots \subset V_{n-1}$, and for $0<i<n$, this outputs the filtered vector space where we remove $V_i$ from the filtration. We define the degeneracy maps
    \[s_{n,i}:C(n)\to C(n+1).\]
    For each $0\leq i\leq n$, this outputs the result of repeating $V_i$ twice. The 2-Segal condition reduces to some elementary linear algebra.
\end{eg}

We describe a mechanism we will frequently use in order to produce simplicial $\two$-categories and maps between simplicial $\two$-categories from an existing one.

\begin{lem}\label{lem:descendingsimplicialobjects}
    Let $\sfC'(-):\Delta^{\op}\to \wCat_{\two}$ be a simplicial $\two$-category, and suppose for each $n$, we have a sub-$\two$-category $\opname{I}_n:\sfC(n)\inclto \sfC'(n)$ admitting a right adjoint $\opname{I}_n^R$.

    For each arrow $a:[m]\to[n]$ in $\Delta$, define
    \[\sfC(a) := \opname{I}_m^R\circ\sfC'(a)\circ \opname{I}_n,\]
    so that we have a lax-commuting square as follows.
    \[\begin{tikzcd}[ampersand replacement=\&]
        {\sfC'(n)} \& {\sfC'(m)} \\
        {\sfC(n)} \& {\sfC(m)}
        \arrow["{\sfC'(a)}", from=1-1, to=1-2]
        \arrow["{\opname{I}_n^R}"', from=1-1, to=2-1]
        \arrow["{\opname{I}_n^R}", from=1-2, to=2-2]
        \arrow[between={0}{0.8}, Rightarrow, from=2-1, to=1-2]
        \arrow["{\sfC(a)}"', from=2-1, to=2-2]
    \end{tikzcd}\]
    If these actually commute for all arrows, or equivalently a generating set of arrows, then $\sfC(-)$ defines a simplicial $\two$-category,
    \[\sfC(-):\Delta^{\op}\to\wCat_{\two},\]
    equipped with a functor, $\sfC'(-)\to \sfC(-)$.
\end{lem}

\begin{proof}
    We allow ourselves to enhance the target and consider the $\two$-category of $\two$-categories, $\widehat{\sfCat}_{\two}$. As defined, we have a lax-functor $\sfC:\Delta^{\op}\to \widehat{\sfCat}_{\two}$. The conditions ensure that this lax-functor factors through $\wCat_{\two}$. Indeed, let $[m]\xto{a}[n]\xto{b}[p]$ be morphisms with composition $c$. Then the 2-morphism $\sfC(a)\circ\sfC(b) \Rightarrow \sfC(c)$ can be written as
    \begin{align*}
        \sfC(a)\circ\sfC(b) &\simeq \sfC(a)\circ\sfC(b)\circ\opname{I}_p^R\circ \opname{I}_p\\
        &\simeq \sfC(a)\circ\opname{I}_n^R\circ\sfC'(b)\circ\opname{I}_p\\
        &\simeq \opname{I}_m^R\circ\sfC'(a)\circ\sfC'(b)\circ\opname{I}_p\\
        &\simeq \sfC(c).
    \end{align*}
    Since $\Delta^{\op}$ is a 1-category, $\sfC$ factors through the maximal $\one$-category,
    \[\wCat_{\two}\inclto \widehat{\sfCat}_{\two}.\]
\end{proof}

A similar mechanism will allow us to produce simplicial maps from existing ones.

\begin{lem}\label{lem:descendingsimplicialmaps}
    Let $\opname{F}'(-):\sfC'(-)\to \sfD'(-)$ be a map between two simplicial $\two$-categories, and suppose that for each $n$, we have sub-$\two$-categories
    \[\opname{I}_n:\sfC(n)\inclto \sfC'(n), \opname{J}_n:\sfD(n)\inclto \sfD'(n)\]
    admitting right adjoints $\opname{I}_n^R,\opname{J}_n^R$.

    Suppose that the conditions of Lemma~\ref{lem:descendingsimplicialobjects} are satisfied for both $\sfC'(-)$ and $\sfD'(-)$ so that we have simplicial $\two$-categories $\sfC(-),\sfD(-)$ as well as maps $\sfC'(-)\to \sfC(-)$ and $\sfD'(-)\to \sfD(-)$.
    
    For each arrow $n$, define
    \[\opname{F}(n):= \opname{J}_n^R\circ\opname{F}'(n)\circ \opname{I}_n: \sfC(n)\to\sfD(n)\]
    so that we have a lax-commuting square as follows.
    \[\begin{tikzcd}[ampersand replacement=\&]
        {\sfC'(n)} \& {\sfD'(n)} \\
        {\sfC(n)} \& {\sfD(n)}
        \arrow["{\opname{F}'(n)}", from=1-1, to=1-2]
        \arrow[from=1-1, to=2-1]
        \arrow[from=1-2, to=2-2]
        \arrow[between={0}{0.8}, Rightarrow, from=2-1, to=1-2]
        \arrow["{\opname{F}(n)}"', from=2-1, to=2-2]
    \end{tikzcd}\]
    If these actually commute for all $n$, then $\opname{F}(-)$ defines a map of simplicial $\two$-categories,
    \[\opname{F}(-):\sfC(-)\to\sfD(-),\]
    and there is a commuting diagram of simplicial $\two$-categories, as follows.
    \[\begin{tikzcd}[ampersand replacement=\&]
        {\sfC'(-)} \& {\sfD'(-)} \\
        {\sfC(-)} \& {\sfD(-)}
        \arrow["{\opname{F}'(-)}", from=1-1, to=1-2]
        \arrow[from=1-1, to=2-1]
        \arrow[from=1-2, to=2-2]
        \arrow["{\opname{F}(-)}"', from=2-1, to=2-2]
    \end{tikzcd}\]
\end{lem}

\begin{proof}
    The proof is similar to that of Lemma~\ref{lem:descendingsimplicialobjects}.
\end{proof}

\subsection{Monads and adjunctions}\label{subsection:mndadj}
We review the theory of monads and adjunctions following \cites{riehlHomotopyCoherentAdjunctions2015,haugsengLaxTransformationsAdjunctions2021}.

We will define the $\two$-category $\sfAdj$ of adjunctions in $\sfSt_k$ as a locally full subcategory of $\sfFun(\frakn(1),\sfSt_k)$, where $\frakn(1)$ is the small $\two$-category $1\from \infty$. This locally full subcategory consists of objects
\[\calA_{1}\xto{F}\calA_{\infty}\]
where $F$ is a left adjoint functor, and morphisms
\[\begin{tikzcd}[ampersand replacement=\&]
	{\calA_1} \& {\calA_\infty} \\
	{\calA_1'} \& {\calA_\infty'}
	\arrow["F", from=1-1, to=1-2]
	\arrow["{\opname{G}_1}"', from=1-1, to=2-1]
	\arrow["{\opname{G}_\infty}", from=1-2, to=2-2]
	\arrow["{F'}", from=2-1, to=2-2]
\end{tikzcd}\]
where the commuting square is adjointable, that is, the natural transformation $\opname{G}_1F^R\to (F')^R\opname{G}_\infty$ is an equivalence.

It is shown in \cite{riehlHomotopyCoherentAdjunctions2015} that there is a small $\two$-category $\frakadj$ called the walking adjunction, such that $\sfFun(\frakadj,\sfSt_k)$ also parametrizes adjunctions in $\sfSt_k$. It follows from \cite{haugsengLaxTransformationsAdjunctions2021} that these two $\two$-categories in fact agree i.e. the functor obtained by restriction to the left adjoint $f:\frakn(1)\to \frakadj$ defines an equivalence,
\[f^*:\sfFun(\frakadj,\sfSt_k)\isomto \sfAdj.\]

This second description of $\sfAdj$ is useful for studying monads. Let $i:\frakmnd\inclto \frakadj$ be the full sub-$\two$-category of $\frakadj$ generated by the source of the left adjoint. Then,
\[\sfMnd:=\sfFun(\frakmnd,\sfSt_k)\]
parametrizes pairs $(\Phi,M)$ where $\Phi$ is an object of $\sfSt_k$ and $M$ is a monad.

\begin{rmk}\label{rmk:monadstrictness}
    We note that morphisms $(\Phi_1,M_1)\to (\Phi_2,M_2)$ are given by a functor $F_{21}:\Phi_1\to \Phi_2$ together with an identification $F_{21}M_1\simeq M_2 F_{21}$. In particular the fiber over $\Phi\in\sfSt_k$ of $\sfMnd$ has no non-invertible morphisms. To see such morphisms, we must allow morphisms in $\sfFun$ to be given by lax natural transformations, as studied in \cite{haugsengLaxTransformationsAdjunctions2021}. This distinction becomes important in \S\ref{subsection:suspension}.
\end{rmk}

There is a pullback functor $i^*:\sfAdj\to \sfMnd$ which on objects sends the adjunction $S:\Phi\tofrom\Psi:R$ to $(\Phi,RS)$. By Theorem~\ref{thm:laxfibrations}, this admits a left and right adjoint, 
\[i_!,i_*:\sfMnd\to \sfAdj.\]
These are the stable, presentable analogs of the Kleisli and Eilenberg--Moore adjunctions respectively. Note that
\[i^*i_*\isomto \Id\]
so that these each embed $\sfMnd$ as a full sub-$\two$-category of $\sfAdj$.

We recall that classically, the Kleisli adjunction is constructed by restricting to the essential image of the left adjoint in the Eilenberg--Moore adjunction. A key observation in \cite{christSphericalMonadicAdjunctions2023} is that the essential image of a functor of stable categories need not be stable. By passing to a stable closure, we obtain a good notion of Kleisli adjunction. It is shown in \cite{gammagePerverseSchobers3d2023}*{Proposition B.4} that this notion agrees with the left adjoint $i_!$.

In our presentable setting, we must also account for the fact that the essential image of a continuous functor of presentable categories need not be presentable. We formulate an analogous recognition principle in our setting.

\begin{lem}\label{lem:kleislirecognition}
    An adjunction $S:\Phi\tofrom\Psi:R$ is in the sub-$\two$-category $i_!(\sfMnd)$ if and only if the smallest stable, presentable subcategory which contains $\opname{Im}(S)$ is equal to $\Psi$.
\end{lem}

\begin{proof}
    The proof is identical to the proof of \cite{gammagePerverseSchobers3d2023}*{Proposition B.4}.
\end{proof}

From $i^*i_*\simeq \Id$, we obtain a natural transformation
\[i_!\to i_*,\]
which is an analog of the classical comparison between the Kleisli and Eilenberg--Moore adjunctions. In our presentable setting, this natural transformation is an equivalence.

\begin{prop}\label{prop:KLisEM}
    The natural transformation $i_!\to i_*$ is an equivalence.
\end{prop}

\begin{proof}
    Let $(\Phi,M)\in\sfMnd$ and denote $i_!(\Phi,M)\to i_*(\Phi,M)$ as follows.
    \[\begin{tikzcd}
        \Phi & {\opname{KL}(M)} \\
        \Phi & {\Mod(M)}
        \arrow["{S_{\opname{KL}}}", shift left, from=1-1, to=1-2]
        \arrow["{\id_{\Phi}}"', from=1-1, to=2-1]
        \arrow["{R_{\opname{KL}}}", shift left, from=1-2, to=1-1]
        \arrow["I", from=1-2, to=2-2]
        \arrow["{S_{\opname{EM}}}", shift left, from=2-1, to=2-2]
        \arrow["{R_{\opname{EM}}}", shift left, from=2-2, to=2-1]
    \end{tikzcd}\]
    By Lemma~\ref{lem:kleislirecognition}, $I$ is fully faithful, and thus $R_{\opname{KL}}$ is conservative. Since we are working inside $\sfSt_k$, $R_{\opname{KL}}$ is automatically colimit-preserving. Thus it is monadic, and agrees with the Eilenberg--Moore adjunction as desired.
\end{proof}

\begin{eg}
    It is useful to consider geometric examples of this. Let $Y$ be a smooth variety equipped with a map $f:Y\to \bbA^1/\bbG_m$. Write $D=Y\times_{\bbA^1/\bbG_m}(B\bbG_m)$, the associated generalized effective Cartier divisor.
    
    On the level of small categories, this yields an adjunction
    \[i^*:\Coh(Y)\tofrom \Coh(D):i_*\]
    which coincides with the Eilenberg--Moore construction as $i_*$ is seen to be monadic. However, the stable essential image of $i^*$ is $\Perf(D)$, noting that since $Y$ is smooth, $\Perf(Y)\isomto \Coh(Y)$. So the Kleisli adjunction associated to the monad $i^*i_*$ is
    \[i^*:\Perf(Y)\tofrom \Perf(D):i_*.\]
    So if $D$ is not smooth (for example take $Y=\pt$ and $f=0$) then the two adjunctions differ.

    On the other hand, in the large setting, both constructions output the adjunction,
    \[i^*:\QCoh(Y)\tofrom \QCoh(D):i_*.\]
    So this exhibits Proposition~\ref{prop:KLisEM} for the object $(\QCoh(Y),M_D := i_*i^*)$. An interesting consequence is that any other adjunction lifting this monad admits a universal map both to and from this one. Take for example the adjunction,
    \[i^*:\QCoh(Y)\tofrom \IndCoh(D):i_*,\]
    which yields the same monad $M_D$ on $\QCoh(Y)$. Then there are maps,
    \[\QCoh(D)\to \IndCoh(D)\to \QCoh(D).\]
    These are the adjoint functors $\Xi_D\adjto \Psi_D$ of \cite{gaitsgoryIndcoherentSheaves2012}.
\end{eg}

\subsection{Families of \texorpdfstring{$\two$}{2}-categories}\label{subsection:twocatfamily}
In this subsection, we prove some category theory results used in Section~\ref{section:families} and in Section~\ref{section:radon}.

\begin{defn}
    A local system of $\two$-categories on a space $X$ is a functor $X\to \wCat_{\two}$, where $X$ is viewed as a $\one$-category. A map between two local systems is a natural transformation of functors. We write $\Fun(X,\wCat_{\two})$ for the category of functors with morphisms given by natural transformations.
\end{defn}

If $X$ is connected and we fix a basepoint $x\in X$, then there is an identification,
\[X\simeq B\Omega_xX,\]
where $\Omega_xX$ be the loop group based at $x\in X$. In this way, a local system of $\two$-categories on $X$ is the same thing as a $\two$-category $\sfC_x$ with an action of $\Omega_xX$.

Consider the map of spaces, $\pi:X\to \pt$. Then, there is a pair of adjoint functors,
\[\pi^*:\wCat_{\two}\tofrom \Fun(X,\wCat_{\two}):\pi_*.\]
We will write $\pi_* := \Gamma(X,-)$. Under the above identification, this is the functor of $\Omega_xX$-fixed points.

\begin{prop}\label{prop:faithful}
    Let $X$ be a connected space and suppose we are given local systems of 2-categories, $\ul{\sfC},\ul{\sfD}$, together with a map $\opname{F}:\ul{\sfC}\to\ul{\sfD}$ between them. Fix $x\in X$, and suppose that $\opname{F}_x:\sfC_x\to\sfD_x$ is 2-faithful. Then, the natural functor
    \[\opname{I}:\Gamma(X,\ul{\sfC})\to \sfC_x\times_{\sfD_x}\Gamma(X,\ul{\sfD})\]
    is 2-fully faithful.

    Consider an object $\calE$ of the target given by $\calD\in \Gamma(X,\ul{\sfD})$ together with a lift of $\calD_x\in\sfD_x$ to some $\calC_x\in\sfC_x$. Then, $\calE$ lies in the essential image of $\opname{I}$ if and only if $\calC_x$ is a fixed point in the set of all lifts of $\calD_x$ along $\opname{F}_x$, under its natural $\pi_1(X,x)$-action.
\end{prop}

\begin{proof}
    As discussed above, we have an identification,
    \[\Gamma(X,\ul{\sfC})\simeq \sfC_x^{\Omega_xX}.\]
    We denote its objects by pairs $(\calC,\Xi)$ consisting of an object $\calC$ equipped with fixed point data $\Xi$. Given objects $(\calC_1,\Xi_1),(\calC_2,\Xi_2)$, the mapping category ${\sfC}(\calC_1,\calC_2)$ obtains a $\Omega_xX$-action, and
    \[{\sfC^{\Omega_xX}}((\calC_1,\Xi_1),(\calC_2,\Xi_2))\simeq {\sfC}(\calC_1,\calC_2)^{\Omega_xX}.\]
    We have similarly for $\Gamma(X,\ul{\sfD})\simeq \sfD^{\Omega_xX}$. We may organize our categories in the following diagram.
    \[\begin{tikzcd}
        {\sfC_x^{\Omega_xX}} & \sfE\pullback & {\sfD_x^{\Omega_xX}} \\
        & {\sfC_x} & {\sfD_x}
        \arrow["{\opname{I}}", from=1-1, to=1-2]
        \arrow[curve={height=12pt}, from=1-1, to=2-2]
        \arrow[from=1-2, to=1-3]
        \arrow[from=1-2, to=2-2]
        \arrow[from=1-3, to=2-3]
        \arrow["{\opname{F}_x}", from=2-2, to=2-3]
    \end{tikzcd}\]

    Now let $\calX_1=(\calC_1,\Xi_1),\calX_2=(\calC_2,\Xi_2)\in \sfC^{\Omega_xX}$. Then, we get the following diagram of hom categories.
    \[\begin{tikzcd}
        {{\sfC_x^{\Omega_xX}}(\calX_1,\calX_2)\simeq{\sfC_x}(\calC_1,\calC_2)^{\Omega_xX}} & {{\sfE}(\opname{I}(\calX_1),\opname{I}(\calX_2))}\pullback & {{\sfD_x}(\opname{F}(\calC_1),\opname{F}(\calC_2))^{\Omega_xX}} \\
        & {{\sfC_x}(\calC_1,\calC_2)} & {{\sfD_x}(\opname{F}(\calC_1),\opname{F}(\calC_2))}
        \arrow["{I}", from=1-1, to=1-2]
        \arrow[curve={height=12pt}, from=1-1, to=2-2]
        \arrow[from=1-2, to=1-3]
        \arrow[from=1-2, to=2-2]
        \arrow[from=1-3, to=2-3]
        \arrow["", from=2-2, to=2-3]
    \end{tikzcd}\]
    By assumption, the bottom horizontal arrow is fully faithful. By an analogous $\one$-categorical statement, $I$ is an equivalence of categories, so we deduce the 2-fully faithfulness.

    We now calculate the objects in the essential image of the functor $\opname{I}$. Let $\calD\in \Gamma(X,\ul{\sfD})$. Since $\opname{F}_x$ is faithful, the $(\infty,2)$-category of lifts of $\calD_x$ along $\opname{F}_x$ is in fact a partially ordered set (viewed as an $(\infty,2)$-category). Thus the action of the group $\Omega_xX$ factors through $\Omega_xX\to \pi_1(X,x)$. A lift $\calC_x$ of $\calD_x$ defines an object $\calE$ of $\sfC_x\times_{\sfD_x}\Gamma(X,\ul{\sfD})$.
    
    Suppose that $\calC_x$ is a fixed point for the $\pi_1(X,x)$-action. Then, for each $y\in X$, we can choose a path $\gamma$ from $x$ to $y$ in $X$, and obtain a lift $\calC_y^\gamma$ of $\calD_y$. The fixed point condition ensures that this is independent of $\gamma$, and so we obtain a preimage $\calC\in\sfC$. The converse is clear.
\end{proof}

\begin{cor}\label{cor:surjectivepi1}
    In the setting of Proposition~\ref{prop:faithful}, suppose that $U\to X$ is a map of connected spaces such that $x\in U$ and the induced map $\pi_1(U,x)\to \pi_1(X,x)$ is surjective. Then, the natural functor
    \[\opname{P}:\Gamma(X,\ul{\sfC})\to \Gamma(U,\ul{\sfC})\times_{\Gamma(U,\ul{\sfD})}\Gamma(X,\ul{\sfD})\]
    is an equivalence.
\end{cor}

\begin{proof}
    The essential image of the functor $\opname{I}$ of Proposition~\ref{prop:faithful} is given by fixed points for a $\pi_1(X,x)$-action, which is the same as being a fixed point under an associated $\pi_1(U,x)$-action. So the result follows.
\end{proof}

\begin{lem}\label{lem:faithfulpullback}
    Suppose we are given a pullback square of 2-categories as follows.
    \[\begin{tikzcd}
        {\sfC}\pullback & {\sfD} \\
        \sfC' & \sfD'
        \arrow["{\opname{F}}", from=1-1, to=1-2]
        \arrow["\opname{G}", from=1-1, to=2-1]
        \arrow["\opname{H}", from=1-2, to=2-2]
        \arrow["{\opname{F}'}", from=2-1, to=2-2]
    \end{tikzcd}\]
    If $\opname{F}'$ is 2-faithful, then so is $\opname{F}$.
\end{lem}

\begin{proof}
    Let $\calC_1,\calC_2\in\sfC$ be objects. Then, we have a pullback square of hom categories,
    \[\begin{tikzcd}
        {\sfC(\calC_1,\calC_2)}\pullback & {\sfD}(\opname{F}(\calC_1),\opname{F}(\calC_2)) \\
        {{\sfC'}(\opname{G}(\calC_1),\opname{G}(\calC_2))} & {{\sfD'}(\opname{HF}(\calC_1),\opname{HF}(\calC_2)).}
        \arrow[from=1-1, to=1-2]
        \arrow[from=1-1, to=2-1]
        \arrow[from=1-2, to=2-2]
        \arrow[from=2-1, to=2-2]
    \end{tikzcd}\]
    Since the bottom horizontal functor is an equivalence of categories, we deduce that the top horizontal functor is an equivalence, thus $\opname{F}$ is 2-faithful.
\end{proof}

\section{Semiorthogonal decompositions}\label{section:sod}
In this section, we review the theory of semiorthogonal decompositions, or SODs.

In \S\ref{subsection:radmissible}, we will define right-admissible SODs and show that we may ``merge'' adjacent components together. We note that it is possible to define SODs without the admissibility assumption, for example as in \cite{christGinzburgAlgebrasTriangulated2022}*{\S 2.6}, but we will not need to work in this level of generality. We will then recall the notion of $\infty$-admissible SODs, and the braid group action by mutation of such SODs.

In \S\ref{subsection:simplicialsod}, we construct $\two$-categories $\sfSOD(n)$ whose objects are given by categories equipped with an SOD. We will then define functors of $\two$-categories $d_{n,i}:\sfSOD(n)\to \sfSOD(n-1)$ for $0<i<n$, which merges the $i$\ts{th} and $(i+1)\ts{th}$ components of the SOD. We further upgrade the assignment
\[[n]\mapsto\sfSOD(n)\]
to a simplicial $\two$-category $\sfSOD(-)$, where the inner face maps are given by the functors $d_{n,i}$. The remaining face maps $d_{n,0},d_{n,n}$ are given by forgetting the first and last components respectively, whilst the degeneracy maps $s_{n,i}$ insert a 0 subcategory between the $i$\ts{th} and $(i+1)\ts{th}$ components.

We prove that $\sfSOD(-)$ satisfies the 2-Segal property as reviewed in \S\ref{subsection:simplicialintro}. Viewing semiorthogonal decompositions as filtrations of a category, this is similar in spirit to the Waldhausen S-construction, which is the prototypical example of a 2-Segal object, as discussed in Example~\ref{eg:waldhausen}.

\subsection{Right-admissible semiorthogonal decompositions}\label{subsection:radmissible}
We fix notation and conventions regarding semi-orthogonal decompositions or SODs. 

\begin{defn}
    A full subcategory $\Phi'$ of $\Phi$ is said to be right-admissible if the inclusion functor admits a right adjoint, which we recall is assumed colimit-preserving.
\end{defn}

\begin{defn}\label{defn:sod}
    Let $\Phi$ be a category. Let $\frakS=(\Phi_1,\dots,\Phi_n)$ be a list of right-admissible subcategories and denote by $I_i:\Phi_i\to\Phi$ the inclusion functors. We say that $\frakS$ forms an $n$-term right-admissible semi-orthogonal decomposition of $\Phi$ if the following two properties hold.
    \begin{enumerate}
        \item\label{item:semi} Whenever $i<j$, and $\phi_i\in\Phi_i,\phi_j\in\Phi_j$ are objects, $\Hom_{\Phi}(\phi_j,\phi_i)=0$.
        \item\label{item:cons} The map $\oplus_i I_i^R:\Phi\to \oplus_i\Phi_i$ is conservative.
    \end{enumerate}
    In this case, we write
    \[\Phi=\langle\Phi_1,\dots,\Phi_n\rangle.\]
\end{defn}

\begin{lem}\label{lem:SODcharacterizations}
    The second condition is equivalent to the following condition.
    \begin{enumerate}[resume]
        \item\label{item:product} The composition $\Cone(I_1I_1^R\to 1)\circ\Cone(I_2I_2^R\to 1)\circ\cdots\circ\Cone(I_nI_n^R\to 1)$ is zero.
    \end{enumerate}
\end{lem}

\begin{proof}
    We show (\ref{item:cons})$\Rightarrow$(\ref{item:product}). Let $\phi\in\Phi$ and consider the element $\phi' := \Cone(I_1I_1^R\to 1)\circ\Cone(I_2I_2^R\to 1)\circ\cdots\circ\Cone(I_nI_n^R\to 1)\phi$. It is easy to see that $I_1^R\phi'=0$ because $I_1^R\Cone(I_1I_1^R\to 1)=0$. Next, $I_2^R\phi'=0$ because $I_2^R\Cone(I_1I_1^R\to 1) \simeq I_2^R$ by semi-orthogonality. Continuing as such, we deduce that $I_i^R\phi'=0$ for all $i$ so $\phi'=0$.
    
    For the converse, we need to show that given $\phi\in\Phi$, if $I_i^R\phi = 0$ for all $i$, then $\phi=0$. By (\ref{item:product}),
    \[0 = \Cone(I_1I_1^R\to 1)\circ\Cone(I_2I_2^R\to 1)\circ\cdots\circ\Cone(I_nI_n^R\to 1)\phi.\]
    In the meantime, if we expand the product, every term vanishes except $\phi$. So $\phi=0$ as desired.
\end{proof}

\begin{defn}
    Let $\Phi$ be a category, $\frakS$ an $n$-term right-admissible SOD of $\Phi$, and $1\leq j< n$ an index. Suppose that $\frakS' = (\Phi_1',\dots,\Phi_n')$ is another right-admissible SOD of $\Phi$.
    \begin{enumerate}
        \item If $\Phi'_i=\Phi_i$ for all $i\neq j, j+1$ and $\Phi'_j = \Phi_{j+1}$ then we say that $\frakS'$ is the mutation of $\frakS$ by the element $\sigma_j\in\Br_n$, and write $\frakS' = \sigma_j\cdot\frakS$
        \item If $\Phi'_i=\Phi_i$ for all $i\neq j, j+1$ and $\Phi'_{j+1} = \Phi_j$ then we say that $\frakS'$ is the mutation of $\frakS$ by the element $\sigma_j^{-1}\in\Br_n$, and write $\frakS' = \sigma_j^{-1}\cdot\frakS$
    \end{enumerate}
\end{defn}

The following lemmas are standard.

\begin{lem}
    Let $\Phi$ be a category, $\frakS$ an $n$-term right-admissible SOD of $\Phi$, and $1\leq j<n$ an index. If $\frakS$ admits a mutation by $\sigma_j\in\Br_n$, then it does so uniquely. We have likewise for $\sigma_j^{-1}$.
\end{lem}

\begin{lem}\label{lem:relations}
    Let $\Phi$ be a category, $\frakS$ an $n$-term right-admissible SOD of $\Phi$. The following hold whenever the relevant mutations exist.
    \begin{enumerate}
        \item $\sigma_i\cdot(\sigma_i^{-1}\cdot\frakS) = \frakS$, $\sigma_i^{-1}\cdot(\sigma_i\cdot\frakS) = \frakS$, for any $i$.
        \item $\sigma_i\cdot(\sigma_j\cdot\frakS) = \sigma_j\cdot(\sigma_i\cdot\frakS)$, for $|i-j|>1$.
        \item $\sigma_i\cdot(\sigma_j\cdot(\sigma_i\cdot\frakS)) = \sigma_j\cdot(\sigma_i\cdot(\sigma_j\cdot\frakS))$, for $|i-j|=1$.
    \end{enumerate}
\end{lem}

For example, when $n=5$ the mutation by $\sigma_2$ can be represented by the following diagram, where we have drawn the original SOD on the bottom and the mutated one on the top.
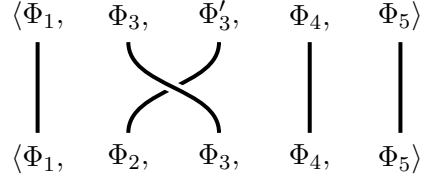
\[\begin{tikzpicture}[scale=1.2]
    \tikzset{
        strand/.style={black, line width=1.5pt},
        mask/.style={white, line width=4.5pt}
    }

    \draw[strand] (1, 2) -- (1, 1) node[below, black] {$\langle\Phi_1,$} node[at start, above] {$\langle\Phi_1,$};
    \draw[strand] (3, 2) .. controls (3, 1.5) and (2,1.5) .. (2, 1) node[below, black] {$\Phi_2,$} node[at start, above] {$\Phi_3',$};
    \draw[strand] (4, 2) -- (4, 1) node[below, black] {$\Phi_4,$} node[at start, above] {$\Phi_4,$};
    \draw[strand] (5, 2) -- (5, 1) node[below, black] {$\Phi_5\rangle$} node[at start, above] {$\Phi_5\rangle$};

    \draw[mask] (2, 2) .. controls (2, 1.5) and (3,1.5) .. (3, 1);
    \draw[strand] (2, 2) .. controls (2, 1.5) and (3,1.5) .. (3, 1) node[below, black] {$\Phi_3,$} node[at start, above] {$\Phi_3,$};
\end{tikzpicture}\]

\begin{rmk}\label{rmk:braidconvention}
    Given braids $\alpha,\beta\in \Br_n$, $\alpha\beta$ is the braid obtained by stacking $\alpha$ on top of $\beta$. Under this convention, the action of $\Br_n$ in Lemma~\ref{lem:actionSOD} is a left action.
\end{rmk}

\begin{defn}
    A semi-orthogonal decomposition $\frakS$ of $\Phi$ is said to be $\infty$-admissible if it admits all repeated mutations.
\end{defn}

\begin{lem}\label{lem:actionSOD}
    If $\frakS$ is an $n$-term $\infty$-admissible SOD of $\Phi$, then the braid group acts on the set of all repeated mutations.
\end{lem}

\begin{defn}\label{defn:periodic}
    Let $\frakS$ be an $n$-term $\infty$-admissible SOD. For an element $\sigma\in\Br_n$, we say that $\frakS$ is $\sigma$-periodic if $\sigma\cdot\frakS=\frakS$. For a subgroup $H\subseteq \Br_n$, say that $\frakS$ is periodic for $H$ if for all $\sigma\in H$, $\frakS$ is $\sigma$-periodic.
\end{defn}

The following says that in a right-admissible SOD, adjacent terms may be merged together.

\begin{lem}\label{lem:twosubcats}
    Suppose that $\Phi_i\xto{I_i}\Phi$ for $i=1,2$ are two right-admissible subcategories satisfying semiorthogonality, $I_2^RI_1=0$. Then, the subcategory $\Phi_{12}$ they generate is also right-admissible and has an SOD $\langle\Phi_1,\Phi_2\rangle$.
\end{lem}

\begin{proof}

    Let $\Phi_{12}$ denote the category of diagrams of the form
        \[\begin{tikzcd}
        {I_2\phi_2} & \\
        {I_1I_1^RI_2\phi_2} & {I_1\phi_1}
        \arrow[from=2-1, to=1-1]
        \arrow[from=2-1, to=2-2]
    \end{tikzcd}\]
    where the vertical arrow is a counit for $I_1\adjto I_1^R$.

    We can construct an adjoint pair of functors $I_{12}:\Phi_{12}\tofrom \Phi:I_{12}^R$, by constructing a functorial identification $\Hom(I_{12}\phi_{12},\phi)\simeq \Hom(\phi_{12},I_{12}^R\phi)$. The functor $I_{12}$ is given by the colimit of the diagram in $\Phi$, whilst $I_{12}^R$ maps $\phi$ to the diagram
    \[\begin{tikzcd}
        {I_2I_2^R\phi} & \\
        {I_1I_1^RI_2I_2^R\phi} & {I_1I_1^R\phi}
        \arrow[from=2-1, to=1-1]
        \arrow[from=2-1, to=2-2]
    \end{tikzcd}\]
    where both arrows are counits.

    We exhibit this adjunction in the proof of Lemma~\ref{lem:twocats}. It is easy to verify that $I_{12}$ is fully faithful and that it admits an SOD $\langle\Phi_1,\Phi_2\rangle$.
\end{proof}

\subsection{\texorpdfstring{$\sfSOD(n)$}{SOD(n)} as a simplicial \texorpdfstring{$\two$}{2}-category}\label{subsection:simplicialsod}
We construct the $\two$-category $\sfSOD(n)$ as follows. We will do this by considering the simpler $\two$-category of partial SODs, and by taking advantage of 2-faithfulness to exhibit higher coherences for free.

Let $\frakn(n)$ be the ordinary 1-category with objects labelled $1,\dots,n,\infty$ and an arrow $i\to \infty$ for each $i$. We let
\[\sfAdj(n)\subseteq \sfFun(\frakn(n),\sfSt_k)\]
be the locally full sub-$\two$-category of those diagrams $i\mapsto \Phi_i$ where each morphism $\Phi_i\to\Phi_{\infty}$ admits a right adjoint, and of those morphisms which are adjointable with respect to these right adjoints. See \S\ref{subsection:mndadj} for details when $n=1$.

Let $\sfSOD(n)$ be the full sub-$\two$-category cut out by the conditions in Definition~\ref{defn:sod}. In order to upgrade this to a simplicial category, we will eventually construct some functors. For example, the middle face map $[3]\to [4]$ will correspond to some functor $\sfSOD(4)\to \sfSOD(3)$ which merges the middle two terms in the SOD. 

We define partial SODs.

\begin{defn}\label{defn:partialsod}
    Let $\Phi$ be a category. Let $\frakS=(\Phi_1,\dots,\Phi_n)$ a list of right-admissible subcategories and denote by $I_i:\Phi_i\to\Phi$ the inclusion functors. We say that $\frakS$ forms an $n$-term right-admissible partial semi-orthogonal decomposition of $\Phi$ if whenever $i<j$, and $\phi_i\in\Phi_i,\phi_j\in\Phi_j$ are objects, we have $\Hom_{\Phi}(\phi_j,\phi_i)=0$.
\end{defn}

Let $\sfSOD'(n)$ be the full sub-$\two$-category of $\sfAdj(n)$ which parametrizes $n$-term partial SODs. We will denote objects of $\sfSOD'(n)$ by $(\frakS;\Phi)$. There are functors $\opname{U}:\sfSOD'(n)\to\sfSt_k$ given by
\[(\frakS;\Phi)\mapsto\Phi.\]
These categories have the advantage that all functors in the simplicial structure lie over $\sfSt_k$ via $\opname{U}$. We show now that moreover, $\opname{U}$ is 2-faithful. Thus roughly speaking, relative to $\sfSt_k$, we are working with partially ordered \textit{sets} as opposed to $\two$-categories. This is made precise in Proposition~\ref{prop:functorconstruction}.

\begin{lem}\label{lem:Uisfaithful}
    The functor $\opname{U}:\sfSOD'(n)\to\sfSt_k$ is 2-faithful.
\end{lem}

\begin{proof}
    Let $\calF := (\frakS;\Phi),\calF' := (\frakS';\Phi')\in\sfSOD'(n)=:\sfS$ be two stable categories equipped with $n$-term partial SODs $\frakS=(\Phi_i)$ and $\frakS'=(\Phi_i')$. The mapping category $\sfS(\calF,\calF')$ is a full subcategory of the limit of the following diagram (cut out by the condition that squares be adjointable).
    \[\begin{tikzcd}[ampersand replacement=\&]
        \& {\sfSt_k(\Phi,\Phi')} \& \\
        {\sfSt_k(\Phi_1,\Phi')} \& \cdots \& {\sfSt_k(\Phi_n,\Phi')} \\
        {\sfSt_k(\Phi_1,\Phi'_1)} \& \cdots \& {\sfSt_k(\Phi_n,\Phi'_n)}
        \arrow["{-\circ I_1}"', from=1-2, to=2-1]
        \arrow[from=1-2, to=2-2]
        \arrow["{-\circ I_n}"', from=1-2, to=2-3]
        \arrow["{I_1^R\circ-}", from=3-1, to=2-1]
        \arrow[from=3-2, to=2-2]
        \arrow["{I_n^R\circ-}", from=3-3, to=2-3]
    \end{tikzcd}\]
    For each $i$, the functor $I_i^R\circ-$ is fully faithful, so the limit is a full subcategory inside $\sfSt_k(\Phi,\Phi')$. The resulting fully faithful composition
    \[\sfS(\calF,\calF')\to \sfSt_k(\Phi,\Phi')\]
    is exactly that which comes from $\opname{U}$ so we are done.
\end{proof}

\begin{prop}\label{prop:simplicialsodprime}
    For each $n\geq0$, there are functors of $\two$-categories over $\sfSt_k$,
    \[s_{n,0},\dots,s_{n,n}:\sfSOD'(n)\to\sfSOD'(n+1)\]
    where $s_{n,i}$ inserts a 0 subcategory between $\Phi_i$ and $\Phi_{i+1}$. There are also functors,
    \[d_{n,0},\dots,d_{n,n}:\sfSOD'(n)\to \sfSOD'(n-1)\]
    where $d_{n,0}$ removes $\Phi_1$, $d_{n,n}$ removes $\Phi_n$, and the remaining $d_{n,i}$ merges $\Phi_i$ and $\Phi_{i+1}$.

    These functors define a simplicial $\two$-category
    \[\sfSOD'(-):\Delta^{\op}\to\wCat_{\two}.\]
\end{prop}

\begin{proof}
    Since we are defining our functors between $\two$-categories faithful over $\sfSt_k$, we can apply Proposition~\ref{prop:functorconstruction}. This implies that to specify a functor $F:\sfSOD'(n)\to\sfSOD'(m)$ over $\sfSt_k$, it is enough to write down a map on equivalence classes of objects, which is given in the statement.

    The proof of Proposition~\ref{prop:functorconstruction} provides a criterion under which this extends to a functor of $\two$-categories, we omit the verification here.

    Similarly, given two functors $a,b: \sfSOD'(n)\to \sfSOD'(m)$ over $\sfSt_k$, by Proposition~\ref{prop:functorconstruction}, it is a property to ask if $a=b$. We need only ask they agree over each $\Phi\in\sfSt_k$. So to check that these functors extend to a simplicial object over $\sfSt_k$, it suffices to verify the usual simplicial identities as if we were defining a simplicial \textit{set}. This is easy to verify.

    In this way, we obtain a simplicial object,
    \[\sfSOD'(-):\Delta^{\op}\to(\wCat_{\two})_{/\sfSt_k}.\]
    Composing with the forgetful functor $(\wCat_{\two})_{/\sfSt_k}\to \wCat_{\two}$, we are done.
\end{proof}

The following will let us pass from $\sfSOD'(n)$ to $\sfSOD(n)$.

\begin{lem}\label{lem:product}
    There is an equivalence of $\two$-categories, $\sfSOD'(n)\isomto \sfSOD(n)\times_{\sfSt_k}\sfSOD'(1)$, over $\sfSt_k$ via the forgetful functor
    \[\sfSOD(n)\times_{\sfSt_k}\sfSOD'(1)\to \sfSOD'(1)\xto{\opname{U}}\sfSt_k.\]
    On objects, this is given by
    \[(\frakS;\Phi)\mapsto ((\frakS;\Phi_{1,\dots,n}),(\Phi_{1,\dots,n};\Phi)),\]
    where $\Phi_{1,\dots,n} := \langle\Phi_1,\dots,\Phi_n\rangle$ and $\frakS$ also denotes its $n$-term SOD.
\end{lem}

\begin{proof}
    As in the proof of Proposition~\ref{prop:simplicialsodprime}, the given datum extends to a functor uniquely if at all. We once again omit the verification.
\end{proof}

\begin{lem}\label{lem:keyrightadjoint}
    The inclusion $\opname{I}_n:\sfSOD(n)\inclto \sfSOD'(n)$ of $\two$-categories admits a right adjoint, which on objects, replaces $\Phi$ by the subcategory generated by $\Phi_1,\dots,\Phi_n$.
\end{lem}

\begin{proof}
    The inclusion identifies under the identification of Lemma~\ref{lem:product} with the result of applying $\sfSOD(n)\times_{\sfSt_k}-$ to $\opname{I}_1:\sfSOD(1)\inclto \sfSOD'(1)$. So we reduce to the case $n=1$.
    
    Via the inclusion $\sfSOD'(1)\inclto\sfAdj$, $\opname{I}_1$ is identified with $p^*$, where $p:\frakadj\to\mathfrak{pt}$ is the trivial map. We also consider the functor $i^*:\sfAdj\to\sfSt_k$ by restricting to the left category, via $i:\mathfrak{pt}\to\frakadj$. It is clear that $i^*p^*$ is an equivalence. We can organize these into the commuting diagram on the left.
    \[\begin{tikzcd}[ampersand replacement=\&]
        {\sfSOD(1)} \& {\sfSOD'(1)} \& {\sfSOD(1)} \& {\sfSOD'(1)} \\
        {\sfSt_k} \& \sfAdj \& {\sfSt_k} \& \sfAdj
        \arrow["{\opname{I}_1}", hook, from=1-1, to=1-2]
        \arrow["\simeq"', from=1-1, to=2-1]
        \arrow["{p^*}", hook, from=1-1, to=2-2]
        \arrow["{\opname{J}}", hook, from=1-2, to=2-2]
        \arrow["\simeq"', from=1-3, to=2-3]
        \arrow["{\opname{I}_1^R}"', from=1-4, to=1-3]
        \arrow["{\opname{J}}", hook, from=1-4, to=2-4]
        \arrow["{i^*}", from=2-2, to=2-1]
        \arrow["{p_*}"', from=2-4, to=1-3]
        \arrow[""{name=0, anchor=center, inner sep=0}, "{i^*}", from=2-4, to=2-3]
        \arrow[between={0}{0.8}, Rightarrow, from=1-3, to=0]
    \end{tikzcd}\]
    
    By Theorem~\ref{thm:laxfibrations}, $p^*$ admits a right adjoint $p_*$, thus $\opname{I}_1$ admits a right adjoint $p_*\circ \opname{J}$. The resulting lax-commuting diagram of $\two$-categories is depicted on the right above.

    What results is a natural transformation,
    \[\opname{I}_1^R\simeq p_*\circ \opname{J}\to i^*\circ\opname{J}.\]
    We claim that this is an equivalence. It suffices to show that for any $\calF\in\sfSOD'(1)$ and $\calA\in\sfSt_k$, the result of applying $\sfSt_k(\calA,-(\calF))$ to this is an equivalence of $\one$-categories. This functor is identified with
    \[\sfAdj(p^*(\calA),\opname{J}(\calF))\to \sfSt_k(\calA,i^*\circ\opname{J}(\calF)).\]
    This is an equivalence because any commuting square of the form
    \[\begin{tikzcd}[ampersand replacement=\&]
        \calA \& \calA \\
        {\Phi_1} \& \Phi
        \arrow["{=}", from=1-1, to=1-2]
        \arrow[from=1-1, to=2-1]
        \arrow[from=1-2, to=2-2]
        \arrow["{I_1}", hook, from=2-1, to=2-2]
    \end{tikzcd}\]
    is automatically adjointable, i.e. it remains a commuting square after passing to right adjoints of the horizontal arrows (see \S\ref{subsection:mndadj}).
\end{proof}

\begin{prop}\label{prop:simplicialsod}
    The assignment $[n]\mapsto \sfSOD(n)$ upgrades to a simplicial $\two$-category. For $0\leq i\leq n$, the degeneracy map $s_{n,i}$ inserts a $0$ between $\Phi_i$ and $\Phi_{i+1}$;
    \begin{align*}
        s_{n,i}:\sfSOD(n)&\to\sfSOD(n+1)\\
        (\frakS;\Phi)&\mapsto ((\Phi_1,\dots,\Phi_i,0,\Phi_{i+1},\dots,\Phi_n);\Phi).
    \end{align*}
    The inert face map $d_{n,0}$ removes $\Phi_1$ from the SOD to obtain $\Phi_{2,\dots,n} := \langle\Phi_2,\dots,\Phi_n\rangle$;
    \begin{align*}
        d_{n,0}:\sfSOD(n)&\to\sfSOD(n-1)\\
        (\frakS;\Phi)&\mapsto ((\Phi_2,\dots,\Phi_n);\Phi_{2,\dots,n}),
    \end{align*}
    and the other inert face map $d_{n,n}$ is similar, removing $\Phi_n$ from the SOD. For $0<i<n$, the active face map is given by merging $\Phi_{i}$ and $\Phi_{i+1}$ in the SOD into $\Phi_{i,i+1} := \langle\Phi_i,\Phi_{i+1}\rangle$;
    \begin{align*}
        d_{n,i}:\sfSOD(n)&\to\sfSOD(n-1)\\
        (\frakS;\Phi)&\mapsto ((\Phi_1,\dots,\Phi_{i-1},\Phi_{i,i+1},\Phi_{i+2},\dots\Phi_n);\Phi).
    \end{align*}
\end{prop}

\begin{proof}
    We apply Lemma~\ref{lem:descendingsimplicialobjects} to the inclusion $\sfSOD(n)\inclto\sfSOD'(n)$, which has a right adjoint as computed in Lemma~\ref{lem:keyrightadjoint}. It suffices to check commutativity (of the a priori only lax-commuting squares) for the face and degeneracy maps. For example, consider the face map $d_{4,4}:\sfSOD(4)\to\sfSOD(3)$. This functor removes the last term in the SOD. In this case, the following lax-commuting square actually commutes; the displayed 2-arrow is an equivalence.
    \[\begin{tikzcd}[ampersand replacement=\&]
        {\sfSOD'(4)} \& {\sfSOD'(3)} \\
        {\sfSOD(4)} \& {\sfSOD(3)}
        \arrow[from=1-1, to=1-2]
        \arrow[from=1-1, to=2-1]
        \arrow[from=1-2, to=2-2]
        \arrow[between={0.2}{1}, Rightarrow, nfold, from=2-1, to=1-2]
        \arrow[from=2-1, to=2-2]
    \end{tikzcd}\]
    This is because both directions are given by
    \[(\Phi_1,\Phi_2,\Phi_3,\Phi_4;\Phi)\mapsto(\Phi_1,\Phi_2,\Phi_3;\langle\Phi_1,\Phi_2,\Phi_3\rangle).\]
\end{proof}

\begin{lem}\label{lem:activecartesian}
    The map of simplicial $\two$-categories $\sfSOD'(-)\to\sfSOD(-)$ is Cartesian over $\Delta^{\op}_{\act}$.
\end{lem}

\begin{proof}
    Since $[1]$ is intial in $\Delta_{\act}$, it suffices to verify that the following square is Cartesian for each $n$, where $a:[1]\to[n]$ is the unique active morphism.
    \[\begin{tikzcd}[ampersand replacement=\&]
        {\sfSOD'(n)} \& {\sfSOD'(1)} \\
        {\sfSOD(n)} \& {\sfSOD(1)}
        \arrow[from=1-1, to=1-2]
        \arrow[from=1-1, to=2-1]
        \arrow[from=1-2, to=2-2]
        \arrow[from=2-1, to=2-2]
    \end{tikzcd}\]
    This is precisely the statement of Lemma~\ref{lem:product}.
\end{proof}

\begin{prop}\label{prop:segalsod}
    The simplicial $\two$-categories $\sfSOD'(-)$ and $\sfSOD(-)$ are 2-Segal.
\end{prop}

\begin{proof}
    As recalled in \S\ref{subsection:simplicialintro}, 2-Segal means that for each pushout square in $\Delta$,
    \begin{equation}\label{equation:twosegal}
        \begin{tikzcd}[ampersand replacement=\&]
            {[n]}\pullback \& {[n+1-m]} \\
            {[m]} \& {[1]}
            \arrow["a"', from=1-2, to=1-1]
            \arrow["i", from=2-1, to=1-1]
            \arrow["{i'}", from=2-2, to=1-2]
            \arrow["{a'}"', from=2-2, to=2-1]
        \end{tikzcd}
\end{equation}
    where $a,a'$ are active and where $i,i'$ are inert, the resulting square of $\two$-categories is a pullback.

    To verify that $\sfSOD'(-)$ is 2-Segal we work over $\sfSt_k$ via the 2-faithful forgetful functors $\opname{U}:\sfSOD'(n)\to\sfSt_k$. Applying Proposition~\ref{prop:functorconstruction}, we may directly construct an inverse functor $\sfSOD'(m)\times_{\sfSOD'(1)}\sfSOD'(n+1-m)\to\sfSOD'(n)$ by working at the level of objects. We omit the details.

    We now verify that $\sfSOD(-)$ is 2-Segal. We need the outer square in the following diagram to commute.
    \[\begin{tikzcd}[ampersand replacement=\&]
        {\sfSOD(n)} \& {\sfSOD'(n)} \& {\sfSOD'(n+1-m)} \& {\sfSOD(n+1-m)} \\
        {\sfSOD(m)} \& {\sfSOD'(m)} \& {\sfSOD'(1)} \& {\sfSOD(1)}
        \arrow["{\opname{I}_n}", hook, from=1-1, to=1-2]
        \arrow[from=1-1, to=2-1]
        \arrow[from=1-2, to=1-3]
        \arrow[from=1-2, to=2-2]
        \arrow["{\opname{I}_{n+1-m}^R}", from=1-3, to=1-4]
        \arrow[from=1-3, to=2-3]
        \arrow[from=1-4, to=2-4]
        \arrow["{\opname{I}_m}"', hook, from=2-1, to=2-2]
        \arrow[from=2-2, to=2-3]
        \arrow["{\opname{I}_{1}^R}"', from=2-3, to=2-4]
    \end{tikzcd}\]
    The middle square is a pullback by the 2-Segal property for $\sfSOD'(-)$. The right square is a pullback by Lemma~\ref{lem:activecartesian} using that $a'$ is active. T see that the left square is also a pullback square, consider the following diagram.
    \[\begin{tikzcd}[ampersand replacement=\&]
        {\sfSOD(n)} \& {\sfSOD'(n)} \& {\sfSOD(n)} \\
        {\sfSOD(m)} \& {\sfSOD'(m)} \& {\sfSOD(m)}
        \arrow["{\opname{I}_n}", hook, from=1-1, to=1-2]
        \arrow[from=1-1, to=2-1]
        \arrow["{\opname{I}_{n}^R}", from=1-2, to=1-3]
        \arrow[from=1-2, to=2-2]
        \arrow[from=1-3, to=2-3]
        \arrow["{\opname{I}_m}", hook, from=2-1, to=2-2]
        \arrow["{\opname{I}_{m}^R}", from=2-2, to=2-3]
    \end{tikzcd}\]
    The left square is the one in question, the right square is a pullback by Lemma~\ref{lem:activecartesian} applied to the active morphism $a$, and the outer square is a pullback because the horizontal compositions are equivalences.
\end{proof}

\section{Localized schobers on a disc}\label{section:localized}
In this section we will define localized perverse schobers on a disc. We note that we will make no reference to non-localized schobers until \S\ref{section:adjaut}.

\begin{defn}\label{defn:compatible}
    Let $\Phi=\langle\Phi_1,\dots,\Phi_n\rangle$ be a category equipped with a right-admissible SOD $\frakS$. An autoequivalence $T$ of $\Phi$ is said to be conjugate-compatible with $\frakS$ if the composition
    \[\Phi_j\xto{I_j}\langle\Phi_1,\dots,\Phi_n\rangle\xto{T}\langle\Phi_1,\dots,\Phi_n\rangle\xto{I_i^R}\Phi_i\]
    is zero whenever $i<j$, and invertible whenever $i=j$.
\end{defn}

We point out a symmetry in this definition with the observation that the composition
\[\Phi_j\xto{I_j}\langle\Phi_1,\dots,\Phi_n\rangle\xto{I_i^R}\Phi_i\]
is zero whenever $i>j$ and identity when $i=j$.

\begin{defn}\label{defn:localized2}
    An abstract localized schober on a disc with $n$ singularities is a triple $(\frakS;\Phi,T)$ where $\Phi$ is a stable $\infty$-category, $\frakS$ is an SOD of $\Phi$, and $T$ is an autoequivalence of $\Phi$ such that $T$ is conjugate-compatible with $\frakS$.
\end{defn}

We temporarily denote the collection of abstract localized schobers on a disc with $n$ singularities by $\sfAut(n)$ and also write $\sfAut = \sfAut(1)$. We will write $\FS_{\sfAut}:\sfAut(n)\to\sfAut$ for the forgetful map
\[(\frakS;\Phi,T)\mapsto (\Phi,T).\]
This will be compatible with the ``enhanced Fukaya--Seidel functor'' which we define in \S\ref{section:adjaut}.

\begin{rmk}
    Definition~\ref{defn:localized2} needs no functorial cones, and in fact makes sense for non-stable settings. This suggests that our theory of localized schobers admits a classical counterpart in terms of triangulated categories, as well as a non-linear counterpart. Similar phenomena have been observed in \cite{kapranovPerverseSchobers2015}*{Remark 3.6}, with the notion of spherical pair, and more generally perverse schobers on a real hyperplane arrangement.
\end{rmk}

\begin{eg}\label{eg:waldhausen1}
    Let $F:\calA\tofrom \calB:G$ be a spherical adjunction and $n\geq 1$. It is shown in \cite{christGinzburgAlgebrasTriangulated2022}*{Lemma 3.8} that one can construct from this a new spherical adjunction,
    \[F':S_{n-1}(F)\tofrom \calB^{\oplus n}:G',\]
    where $S_{n-1}(F)$ is a term in the relative Waldhausen construction of \cite{dyckerhoffSphericalAdjunctionsStable2021}, and where $F',G'$ are described by explicit formulae. The category $S_{n-1}(F)$ admits an $n$-term SOD $\frakS$ of the form
    \[S_{n-1}(F) \simeq \langle\calA,\calB,\dots,\calB\rangle.\]

    In this situation, $T' := \coCone(1\to F'G')$ is conjugate-compatible with $\frakS$, and so $(\frakS;S_{n-1}(F),T')$ defines an object of $\sfAut(n)$. In Example~\ref{eg:waldhausen2}, we will explain that this actually comes from a schober with $n$ singularities.
\end{eg}

In \S\ref{subsection:comparison}, we show that this definition decategorifies to localized perverse sheaves as studied in \cite{gelfandPerverseSheavesQuivers1996}.

In \S\ref{subsection:simplicialend} and \S\ref{subsection:simplicialaut}, we will upgrade $\sfAut(n)$ to a $\two$-category, $\FS_{\sfAut}$ to a functor of $\two$-categories, and the assignment
\[\sfAut(-):[n]\mapsto\sfAut(n)\]
to a simplicial $\two$-category in such a way that the unique active map $[1]\to [n]$ induces $\FS_{\sfAut}$. More generally, the degeneracy maps $s_{n,i}:\sfAut(n)\to\sfAut(n+1)$ insert a zero into the SOD between $\Phi_i$ and $\Phi_{i+1}$, the inert face maps $d_{n,0},d_{n,n}:\sfAut(n)\to\sfAut(n-1)$ are given by forgetting $\Phi_1,\Phi_n$ respectively, and the active face maps $d_{n,i}$ merge $\Phi_i,\Phi_{i+1}$ in the SOD. We will prove that $\sfAut(n)$ is 2-Segal, which we loosely interpret as saying that localized schobers have good factorization properties.

In \S\ref{subsection:admissible}, we will prove the following recharacterization of abstract localized schobers.

\begin{thm}\label{thm:alternate}
    Let $\Phi$ be a stable category equipped with a right-admissible SOD $\frakS$ and autoequivalence $T$. Then, $(\Phi,\frakS,T)$ is an abstract localized schober if and only if $\frakS$ is $\infty$-admissible, and
    \[T(\frakS) = \delta^{-2}\cdot \frakS,\]
    where $T(\frakS)$ is the result of applying $T$ to each subcategory, and $\delta^2 \in\Br_n$ is the positive full-twist element.
\end{thm}

The infinite admissibility lets us construct an action of the braid group $\Br_n$ on the $\two$-category $\sfAut(n)$. We will use this to define $\two$-categories of non-abstract localized perverse schobers, denoted $\olsftwoP(\bbC,R)$, where $R\subset\bbC$ is a finite subset.

In \S\ref{subsection:interpretation}, we conclude the section by giving two related interpretations of localized schobers, the first in terms of Serre functors and the second in terms of the global monodromy of the Fukaya--Seidel category of a Lefschetz fibration.

\subsection{Comparison to localized perverse sheaves}\label{subsection:comparison}
We compare the definition of abstract localized schobers with that of localized perverse sheaves. We recall the following.

\begin{prop}[Gelfand--MacPherson--Vilonen, \cite{gelfandPerverseSheavesQuivers1996}*{Proposition 2.3}]\label{prop:gmvquiver}
    Let $R\subset \bbC$ a finite subset of size $n$. Let $K_n$ be the quiver with $n$ vertices and arrows $j\to i$ for all $1\leq i,j\leq n$. Let $\opname{Q}(n)$ be the category of quiver representations $(M_i,m_{ij})$ of $K_n$ such that for all $i$, $1-m_{ii}$ is invertible.

    After making a choice of cuts on $(\bbC,R)$, there is an equivalence of categories,
    \[\Perv(\bbC,R)/\Loc(\bbC)\simeq \opname{Q}(n).\]
\end{prop}

We unravel the definition of abstract localized schober to resemble this quiver description. Recall that by the theory of gluing functors (which arises here as the monadicity of $\oplus_iI_i^R$), the data of the SOD $\Phi=\langle\Phi_1,\dots,\Phi_n\rangle$ is equivalent to giving the components $\Phi_i$, together with gluing functors $F_{ij} := I_i^RI_j$ and its monad structure. By orthogonality, $F_{ij}=0$ when $i>j$, so this monad is upper triangular. Moreover $F_{ii}\simeq 1$. This can be seen as the ``upper triangular half'' of the categorical quiver representation.

The ``lower triangular half'' comes from the autoequivalence $T$. Define $T_{ij} := I_i^RTI_j:\Phi_j\to \Phi_i$. Conjugate-compatibility says that $T_{ij}=0$ when $i<j$ and that $T_{ii}$ is invertible. This collection of functors admits a ``bimodule structure'' for the lower triangular monad $F := (F_{ij})_{i\leq j}$. We see that the data of the functor $T$ is equivalent to the functors $(T_{ij})_{i\leq j}$ together with its bimodule structure. So we have the following.

\begin{prop}\label{prop:decategorification1}
    Given an abstract localized perverse schober $(\Phi,\Phi_i,T)\in\sfAut(n)$, we may form an object of $\opname{Q}(n)$ by taking $M_i := K_0(\Phi_i)$ and $m_{ij} = K_0(F_{ij}\oplus T_{ij})$.
\end{prop}

\begin{rmk}
    Additional structures such as the monad structure and the bimodule structure disappear when we decategorify, as these are encoded by 2-morphisms in $\sfSt_k$. In plain terms, if $f,g:V\to W$ are linear maps of vector spaces, we can't make sense of a (non-invertible) morphism $f\to g$.
\end{rmk}

\subsection{\texorpdfstring{$\sfEnd(n)$}{End(n)} as a simplicial \texorpdfstring{$\two$}{2}-category}\label{subsection:simplicialend}

We will consider a variant of $\sfAut(n)$ where we don't impose invertibility conditions, and construct out of this a simplicial $\two$-category.

\begin{defn}
    Let $\Phi$ be a category equipped with a (possibly partial) SOD $\frakS=(\Phi_1,\dots,\Phi_n)$. Denote by $I_i$ the inclusion functors. An \textit{endofunctor} $T$ of $\Phi$ is said to be weakly conjugate-compatible with $\frakS$ if the composition
    \[\Phi_j\xto{I_j}\langle\Phi_1,\dots,\Phi_n\rangle\xto{T}\langle\Phi_1,\dots,\Phi_n\rangle\xto{I_i^R}\Phi_i\]
    is zero whenever $i<j$.
\end{defn}

\begin{notn}
    We denote by $\sfEnd$ the $\two$-category of pairs $(\Phi,T)$ where $T:\Phi\to\Phi$ is an endofunctor. We denote by $\sfEnd(n)$ the $\two$-category of data $(\Phi,\frakS,T)$ where $\Phi$ is a category, $\frakS$ is an SOD of $\Phi$, and $T:\Phi\to\Phi$ is an \textit{endofunctor}, weakly conjugate-compatible with $\frakS$. These $\two$-categories are constructed similarly to $\sfAut(n)$.
\end{notn}

We upgrade the assignment $[n]\mapsto \sfEnd(n)$ to a simplicial $\two$-category which is 2-Segal. Consider the simplicial $\two$-category
\[[n]\mapsto \sfSOD'(n)\times_{\sfSt_k}\sfEnd,\]
where the functors to $\sfSt_k$ are given by $\opname{U}:\sfSOD'(n)\to\sfSt_k$, and $\sfEnd\to\sfSt_k$ is the forgetful functor. We denote objects by $(\Phi_1,\dots,\Phi_n;\Phi,T)$ and write $I_i:\Phi_i\inclto \Phi$ for inclusions of subcategories.

Define the sub-$\two$-category
\[\sfEnd'(n)\subseteq \sfSOD'(n)\times_{\sfSt_k}\sfEnd\]
consisting of those objects $(\frakS;\Phi,T)$ for which the endofunctor $T$ is weakly conjugate-compatible with the partial SOD $\frakS$ of $\Phi$.

\begin{prop}\label{prop:segalend}
    The assignment $\sfEnd'(-):[n]\mapsto \sfEnd'(n)$ is a simplicial sub-$\two$-category of $\sfSOD'(-)\times_{\sfSt_k}\sfEnd$. Moreover, $\sfEnd'(-)$ is 2-Segal.
\end{prop}

\begin{proof}
    We show that $\sfEnd'(n)$ maps to $\sfEnd'(n-1)$ under the face maps $d_{n,i}$ for $0<i<n$. Stability under the remaining face maps and degeneracy maps is easy to verify.

    Let $(\Phi_1,\dots,\Phi_n;\Phi,T)\in \sfEnd'(n)$. The operation $d_{n,i}$ collapses $\Phi_i,\Phi_{i+1}$ into $\langle\Phi_i,\Phi_{i+1}\rangle$, whose inclusion we denote by $I_{i,i+1}$. So we must verify that for all $j>i+1$,
    \[I_{i,i+1}^R T I_{j} = 0\]
    and that for all $j<i$,
    \[I_j^R T I_{i,i+1} = 0.\]
    The first of these follows by post-composing with $\langle\Phi_i,\Phi_{i+1}\rangle\to \Phi_i\oplus\Phi_{i+1}$ which is conservative. The second of these follows because $\langle\Phi_i,\Phi_{i+1}\rangle$ is generated by $\Phi_i$ and $\Phi_{i+1}$.

    We now verify the 2-Segal condition. By Proposition~\ref{prop:segalsod}, $\sfSOD'(n)$ is 2-Segal. It follows that for each square in $\Delta$ as in (\ref{equation:twosegal}), the natural map
    \[\sfEnd'(n)\to \sfEnd'(m)\times_{\sfEnd'(1)}\sfEnd'(n+1-m)\]
    is 2-fully faithful. So we can work on the level of objects as follows.
    
    Let $(\Phi_1,\dots,\Phi_n;\Phi,T)\in \sfSOD'(-)\times_{\sfSt_k}\sfEnd$. Given $0<i<j<n$, suppose that
    \[(\Phi_i,\dots,\Phi_j;\Phi,T), (\Phi_1,\dots,\Phi_{i-1},\Phi_{i,\dots,j},\Phi_{j+1},\dots,\Phi_n;\Phi,T)\]
    individually lie in $\sfEnd'(-)$. Here, $\Phi_{i,\dots,j} := \langle\Phi_i,\dots,\Phi_j\rangle$, and we write $I_{i,\dots,j}$ for its inclusion as a subcategory inside $\Phi$. It will suffice to prove that $(\Phi_1,\dots,\Phi_n;\Phi,T)$ also lies in $\sfEnd'(-)$.

    We must verify that for all $1\leq i' < j'\leq n$, $I_{i'}^RTI_{j'}=0$. If both $i',j'$ lie inside $\{i,\dots,j\}$ or if they both lie outside $\{i,\dots,j\}$, then this is clear. If just $j'$ lies inside then $I_{j'}$ factors as $\Phi_{j'}\xto{I_{j'}^{i,\dots,j}} \Phi_{i,\dots,j}\xto{I_{i,\dots,j}} \Phi$, so
    \[I_{i'}^RTI_{j'} = I_{i'}^RTI_{i,\dots,j}I_{j'}^{i,\dots,j}  = 0.\]
    The argument is similar if just $i'$ lies in $\{i,\dots,j\}$ so we are done.
\end{proof}

For each $n$, there is an inclusion functor, $\opname{I}_n:\sfEnd(n) \inclto \sfEnd'(n)$ identifying $\sfEnd(n)$ with the full sub-$\two$-category of those objects where the partial SOD is an SOD.

\begin{lem}\label{lem:rightadjointend}
    The inclusion functor $\opname{I}_n:\sfEnd(n) \inclto \sfEnd'(n)$ admits a right adjoint $\opname{I}_n^R$, which is given on objects by
    \[(\frakS;\Phi,T)\mapsto (\frakS;\langle\Phi_1,\dots,\Phi_n\rangle,I^RTI),\]
    where $I:\langle\Phi_1,\dots,\Phi_n\rangle\inclto \Phi$.
\end{lem}

\begin{proof}
    Consider first the case $n=1$. Then, the right adjoint exists and is given on objects by
    \[(\Phi_1;\Phi,T)\mapsto (\Phi_1,I_1^RTI_1),\]
    by a proof similar to that of Lemma~\ref{lem:keyrightadjoint}.

    For general $n$, consider the following commuting diagram of $\two$-categories, where each arrow is 2-fully faithful.
    \[\begin{tikzcd}[ampersand replacement=\&]
        {\sfEnd(n)} \& {\sfEnd'(n)} \\
        {\sfSOD(n)\times_{\sfSt_k}\sfEnd(1)} \& {\sfSOD(n)\times_{\sfSt_k}\sfEnd'(1)}
        \arrow["{\opname{I}_n}", hook, from=1-1, to=1-2]
        \arrow[hook, from=1-1, to=2-1]
        \arrow[hook, from=1-2, to=2-2]
        \arrow["{\id\times\opname{I_1}}"', hook, from=2-1, to=2-2]
    \end{tikzcd}\]
    The bottom horizontal arrow has a right adjoint by the case $n=1$. This right adjoint preserves the full sub-$\two$-categories, and hence induces a right adjoint of $\opname{I}_n$, given by the formula in the statement.
\end{proof}

\begin{cor}\label{prop:simplicialend}
    The assignment $[n]\mapsto\sfEnd(n)$ upgrades to a simplicial $\two$-category. For $0\leq i\leq n$, the degeneracy map
    \[s_{n,i}:\sfEnd(n)\to \sfEnd(n+1)\]
    inserts a $0$ between $\Phi_i$ and $\Phi_{i+1}$. For $0< i< n$, the active face map
    \[d_{n,i}:\sfEnd(n)\to \sfEnd(n-1)\]
    merges $\Phi_i,\Phi_{i+1}$ into $\langle\Phi_i,\Phi_{i+1}\rangle$. The inert face maps $d_{n,0}$ and $d_{n,n}$ remove $\Phi_1$ and $\Phi_n$ respectively. 

    The assignment $[n]\mapsto \sfEnd(n)$ upgrades to a simplicial $\two$-category. For $0\leq i\leq n$, the degeneracy map $s_{n,i}$ inserts a $0$ between $\Phi_i$ and $\Phi_{i+1}$;
    \begin{align*}
        s_{n,i}:\sfEnd(n)&\to\sfEnd(n+1)\\
        (\frakS;\Phi,T)&\mapsto ((\Phi_1,\dots,\Phi_i,0,\Phi_{i+1},\dots\Phi_n);\Phi,T).
    \end{align*}
    The inert face map $d_{n,0}$ removes $\Phi_1$ from the SOD to get $\Phi_{2,\dots,n} := \langle\Phi_2,\dots,\Phi_n\rangle\xto{I_{2,\dots,n}}\Phi$;
    \begin{align*}
        d_{n,0}:\sfEnd(n)&\to\sfEnd(n-1)\\
        (\frakS;\Phi,T)&\mapsto ((\Phi_2,\dots,\Phi_n);\Phi_{2,\dots,n},I_{2,\dots,n}^RTI_{2,\dots,n}),
    \end{align*}
    and the other inert face map $d_{n,n}$ similarly removes $\Phi_n$ from the SOD. For $0<i<n$, the active face map is given by merging $\Phi_{i}$ and $\Phi_{i+1}$ in the SOD into $\Phi_{i,i+1} := \langle\Phi_i,\Phi_{i+1}\rangle$;
    \begin{align*}
        d_{n,i}:\sfEnd(n)&\to\sfEnd(n-1)\\
        (\frakS;\Phi,T)&\mapsto ((\Phi_1,\dots,\Phi_{i-1},\Phi_{i,i+1},\Phi_{i+2},\dots\Phi_n);\Phi,T).
    \end{align*}
\end{cor}

\begin{proof}
    We apply Lemma~\ref{lem:descendingsimplicialobjects}. We omit the verification that the resulting lax-commuting squares are commuting squares.
\end{proof}

\begin{prop}
    The simplicial $\two$-category $\sfEnd(-)$ is 2-Segal.
\end{prop}

\begin{proof}
    The strategy used in the proof of Proposition~\ref{prop:segalsod} works by observing that the map $\sfEnd'(-)\to\sfEnd(-)$ is Cartesian over $\Delta_{\act}$.
\end{proof}

\subsection{\texorpdfstring{$\sfAut(n)$}{Aut(n)} as a simplicial \texorpdfstring{$\two$}{2}-category}\label{subsection:simplicialaut}
We now reinstate the invertibility conditions and construct the $\two$-category of abstract localized schobers, and upgrade this to a simplicial $\two$-category which is 2-Segal.

\begin{notn}
    We denote by $\sfAut$ the $\two$-category of data $(\Phi,T)$ consisting of a stable category equipped with an autoequivalence. This admits a forgetful functor, $\sfAut\to\sfSt_k$. We denote by $\sfAut(n)$ the full sub-$\two$-category of $\sfAut\times_{\sfSt_k}\sfSOD(n)$ of abstract localized schobers.
\end{notn}

We now exhibit $[n]\mapsto \sfAut(n)$ as a 2-Segal simplicial $\two$-category.

\begin{prop}\label{prop:simplicialaut}
    The full sub-$\two$-categories $\sfAut(n)\inclto \sfEnd(n)$ form a sub-simplicial $\two$-category which is also 2-Segal.
\end{prop}

We begin with two technical lemmas.
\begin{lem}\label{lem:technical}
    Let $\Phi$ be a category, and $\frakS=(\Phi_1,\Phi_2)$ be a right-admissible SOD. Let $T:\Phi\to \Phi$ be an autoequivalence such that $I_1^RTI_2=0$. Write $T_i := I_i^RTI_i$. If $T_1$ is invertible, then there is a natural identification,
    \[\Cone(I_2I_2^{R}\to \id)\simeq I_1T_1^{-1}I_1^RT.\]
    If $T_2$ is invertible, then there is a natural identification,
    \[\Cone(I_1I_1^{R}\to \id)\simeq TI_2T_2^{-1}I_2^R.\]
\end{lem}

\begin{proof}
    Let us prove the second statement. Consider the diagram of functors
    \[\id\from I_2I_2^R\to TI_2T_2^{-1}I_2^R,\]
    where the first arrow is the counit and the second comes from $\id \simeq I_2^RTI_2(T_2^{-1})$. We postcompose the diagram by $I_1I_1^R\to \id$ to obtain the following commuting diagram.
    \[\begin{tikzcd}[ampersand replacement=\&]
	{I_1I_1^R} \& {I_1I_1^RI_2I_2^R} \& 0 \\
        \id \& {I_2I_2^R} \& {TI_2T_2^{-1}I_2^R}
        \arrow[from=1-1, to=2-1]
        \arrow[from=1-2, to=1-1]
        \arrow[from=1-2, to=1-3]
        \arrow[from=1-2, to=2-2]
        \arrow[from=1-3, to=2-3]
        \arrow[from=2-2, to=2-1]
        \arrow[from=2-2, to=2-3]
    \end{tikzcd}\]
    The top right is zero by $I_1^RTI_2=0$. The left square is Cartesian by Lemma~\ref{lem:SODcharacterizations}. The right square is also Cartesian because if we apply either $I_1^R$ or $I_2^R$, it becomes Cartesian. It follows that the cones of the vertical maps all agree, giving the desired identification.

    The first statement is similar beginning with
    \[\id\from I_1I_1^R\to I_1T_1^{-1}I_1^RT,\]
    and precomposing with $I_2I_2^R\to \id$.
\end{proof}

\begin{lem}\label{lem:invertibletwo}
    Let $\Phi$ be a category, and $\frakS=(\Phi_1,\Phi_2)$ be a right-admissible SOD. Let $T:\Phi\to \Phi$ be an autoequivalence such that $I_1^RTI_2=0$. Write $T_i := I_i^RTI_i$. Then, $T_1$ is an autoequivalence if and only if $T_2$ is an autoequivalence.
\end{lem}

\begin{proof}
    Suppose that $T_1$ is an autoequivalence. We claim that
    \[V_2 := I_2^RT^{-1}\Cone(I_1I_1^R\to \id)I_2\]
    is an inverse to $T_2$. This is a formal but tedious verification. We have
    \begin{align*}
        V_2T_2 &= I_2^RT^{-1}\Cone(I_1I_1^R\to \id)I_2\circ I_2^RTI_2\\
        &\simeq I_2^RT^{-1}\Cone(I_1I_1^R\to \id)TI_2\\
        &\simeq I_2^RT^{-1}TI_2\\
        &\simeq \id_{\Phi_2}.
    \end{align*}
    
    To evaluate $T_2V_2$, see that
    \begin{align*}
        T_2V_2 = I_2^RTI_2\circ I_2^RT^{-1}\Cone(I_1I_1^R\to \id)I_2
    \end{align*}
    is the cone of a map between the following functors,
    \begin{align*}
        &I_2^RT \circ T^{-1}\Cone(I_1I_1^R\to \id)I_2,\\
        &I_2^RT \circ \Cone(I_2I_2^R\to \id)\circ T^{-1}\Cone(I_1I_1^R\to \id)I_2.
    \end{align*}
    The first of these is identified with $\id_{\Phi_2}$. The second is evaluated using Lemma~\ref{lem:technical} as follows.
    \begin{align*}
        &I_2^RT \circ\Cone(I_2I_2^R\to \id)\circ T^{-1}\Cone(I_1I_1^R\to \id)I_2\\
        \simeq&  I_2^RT \circ I_1T_1^{-1}I_1^RT\circ T^{-1}\Cone(I_1I_1^R\to \id)I_2\\
        \simeq& I_2^RT \circ I_1T_1^{-1}I_1^R \Cone(I_1I_1^R\to \id)I_2\\
        \simeq &0.
    \end{align*}
    So $T_2V_2\simeq \id_{\Phi_2}$.

    The proof in the other direction is similar, with inverse given by
    \[V_1 := I_1^R\Cone(I_2I_2^R\to \id)T^{-1}I_1.\]
\end{proof}

\begin{proof}[Proof of Proposition~\ref{prop:simplicialaut}]
    It is clear that the degeneracy maps and the inert face maps $d_{n,0},d_{n,n}$ preserve $\sfAut(-)$. We verify that the active face maps $d_{n,i}:\sfEnd(n)\to\sfEnd(n-1)$ for $0<i<n$ preserve $\sfAut(-)$. Recall these functors are given by
    \[(\Phi_1,\dots,\Phi_n;\Phi,T)\mapsto (\Phi_1,\dots\Phi_{i,i+1}\dots,\Phi_n;\Phi,T),\]
    where $I_{i,i+1}:\Phi_{i,i+1} \inclto \Phi$ is subcategory generated by $\Phi_i,\Phi_{i+1}$. It suffices to verify that $I_{i,i+1}^RTI_{i,i+1}$ is invertible. This follows by repeatedly applying Lemma~\ref{lem:invertibletwo}. We thus deduce that $\sfAut(-)$ is a sub-simplicial $\two$-category.

    We now verify the 2-Segal condition. Let $(\Phi_1,\dots,\Phi_n;\Phi,T)\in \sfEnd(n)$. Given $0<i<j<n$, suppose that
    \[(\Phi_i,\dots,\Phi_j;\Phi,T_{i,\dots,j}), (\Phi_1,\dots,\Phi_{i-1},\Phi_{i,\dots,j},\Phi_{j+1},\dots,\Phi_n;\Phi,T)\]
    individually lie in $\sfAut(-)$. Here, $\Phi_{i,\dots,j} := \langle\Phi_i,\dots,\Phi_j\rangle$, $I_{i,\dots,j}$ is its inclusion as a subcategory inside $\Phi$ and $T_{i,\dots,j} = I_{i,\dots,j}^RTI_{i,\dots,j}$. It will suffice to prove that $(\Phi_1,\dots,\Phi_n;\Phi,T)$ lies in $\sfAut(n)$.

    We must verify that for all $1\leq i' \leq n$, $I_{i'}^RTI_{i'}$ is invertible. If $i'$ lies outside $\{i,\dots,j\}$ then this is clear. Otherwise, $I_{i'}$ factors as $\Phi_{i'}\xto{I_{i'}^{i,\dots,j}} \Phi_{i,\dots,j}\xto{I_{i,\dots,j}} \Phi$, so
    \[I_{i'}^RTI_{i'} = (I_{i'}^{i,\dots,j})^RT_{i,\dots,j}I_{i'}^{i,\dots,j}\]
    is invertible. This completes the proof.
\end{proof}

\subsection{Braid group action on abstract localized schobers}\label{subsection:admissible}
We establish that abstract localized perverse schobers enjoy a braid group action by mutation of its SOD. We then prove Theorem~\ref{thm:alternate} which gives an alternate characterization of abstract localized schobers.

\begin{thm}\label{thm:inftyadmissible}
    Let $(\Phi,\frakS,T)\in\sfAut(n)$ be an abstract localized schober. Then, $\frakS$ is an $\infty$-admissible SOD. Moreover, for any $\beta\in\Br_n$, $(\Phi,\beta\cdot\frakS,T)$ is an abstract localized schober.
\end{thm}

\begin{cor}\label{cor:action}
    There is a natural action of the braid group $\Br_n$ on $\sfAut(n)$. Moreover, $\FS_{\sfAut}:\sfAut(n)\to\sfAut$ is naturally $\Br_n$-equivariant where $\sfAut$ is given the trivial action.
\end{cor}

\begin{proof}
    It follows from Lemma~\ref{lem:Uisfaithful} that $\FS_{\sfAut}:\sfAut(n)\to\sfAut$ is 2-faithful. So functors $\sfAut(n)\to \sfAut(n)$ over $\sfAut$ may be defined at the level of objects by Proposition~\ref{prop:functorconstruction}. Using this, we can check that for each $\beta\in\Br_n$, the assignment of Theorem~\ref{thm:inftyadmissible} extends to a functor
    \[\mathop{A}(\beta): \sfAut(n)\to\sfAut(n),\]
    and so we obtain a functor of $\one$-categories,
    \begin{align*}
        \Br_n&\to \Hom_{(\wCat_{\two})_{/\sfAut}}(\sfAut(n),\sfAut(n))\\
        \beta&\mapsto \mathop{A}(\beta).
    \end{align*}
    
    To extend this to a group action, we must lift this to a monoidal functor. It follows by Proposition~\ref{prop:functorconstruction} that the target $\one$-category is in fact discrete, so it is a property for this functor to be monoidal. This amounts to checking that the functors $\mathop{A}(\beta_1)\circ \mathop{A}(\beta_2)$ and $\mathop{A}(\beta_1\beta_2)$ agree at the level of objects, which in turn follows from Lemma~\ref{lem:relations}.
\end{proof}

We begin by studying the smallest non-trivial case of Theorem~\ref{thm:inftyadmissible}.
\begin{lem}\label{lem:nistwo}
    Theorem~\ref{thm:inftyadmissible} holds for $n=2$.
\end{lem}

\begin{proof}
    Let $\sigma_1\in\Br_2$ be the positive generator. By induction on the length of the braid group element, it suffices to verify the following.
    \begin{enumerate}
        \item $\frakS$ admits a mutation by $\sigma_1$, and $T$ is conjugate-compatible with $\sigma_1\cdot\frakS$.
        \item $\frakS$ admits a mutation by $\sigma_1^{-1}$, and $T$ is conjugate-compatible with $\sigma_1^{-1}\cdot\frakS$.
    \end{enumerate}

    We will prove the second statement. We write $\frakS=(\Phi_1,\Phi_2)$ and $I_i:\Phi_i\inclto\Phi$ as usual. Write also,
    \[T_i := I_iTI_i^R,\]
    which we recall is an autoequivalence of $\Phi_i$. We show that the subcategories $(T(\Phi_2),\Phi_1)$ define a right-admissible SOD of $\Phi$, thereby exhibiting the mutation $\sigma_1^{-1}\cdot\frakS$. It is clear that the inclusion functors $I_1' := TI_2, I_2' := I_1$ admit right adjoints, and that we have semi-orthogonality,
    \[I_2'^{R}I_1' = I_1^RTI_2=0.\]
    By Lemma~\ref{lem:SODcharacterizations}, it suffices to show that
    \[\Cone(I_1'I_1'^{R}\to \id)\Cone(I_2'I_2'^{R}\to \id)=0.\]
    By Lemma~\ref{lem:technical},
    \[\Cone(I_2'I_2'^{R}\to \id)\simeq I_1'T_2^{-1}I_1'^RT^{-1}.\]
    We deduce therefore that
    \[\Cone(I_1'I_1'^{R}\to \id)\Cone(I_2'I_2'^{R}\to \id) = 0.\]

    Let us now verify that $T$ is conjugate-compatible with the mutation $\sigma_1\cdot\frakS$. We see that $I_1'^RTI_2' \simeq I_2^RI_1=0$, and also that $I_1'^RTI_1' \simeq T_2$, $I_2'^RTI_2' \simeq T_1$ are invertible, as required.

    The first statement is proven similarly, starting with the claim that the list of subcategories $(\Phi_2,T^{-1}(\Phi_1))$ defines the mutation by $\sigma_1$.
\end{proof}

In order to prove Theorem~\ref{thm:inftyadmissible} for general $n$, we will take advantage of the 2-Segal property of $\sfAut(-)$.

\begin{proof}[Proof of Theorem~\ref{thm:inftyadmissible}]
    We show that for each $i=1,\dots,n-1$ and sign $\epsilon\in\{1,-1\}$, $\frakS$ admits a mutation by $\sigma_i^{\epsilon}$, and $T$ is conjugate-compatible for the mutation $\sigma_i^{\epsilon}\cdot\frakS$.

    From the 2-Segal property of $\sfAut(-)$ which was proven in Proposition~\ref{prop:simplicialaut}, we have a pullback square as follows.
    \[\begin{tikzcd}[ampersand replacement=\&]
        {\sfAut(n)}\pullback \& {\sfAut(2)} \\
        {\sfAut(n-1)} \& {\sfAut(1)}
        \arrow[from=1-1, to=1-2]
        \arrow[from=1-1, to=2-1]
        \arrow[from=1-2, to=2-2]
        \arrow[from=2-1, to=2-2]
    \end{tikzcd}\]
    Here, the top horizontal arrow restricts to the 2-term SOD comprised of $\Phi_i,\Phi_{i+1}$, and the left vertical arrow contracts these two components. The relevant mutation on $\sfAut(2)$ induces via this pullback square, the desired mutation.
\end{proof}

We conclude by proving Theorem~\ref{thm:alternate}, which provides an alternate characterization of abstract localized schobers.

\begin{proof}[Proof of Theorem~\ref{thm:alternate}]
    Suppose that $(\Phi,\frakS,T)$ is an abstract localized schober. Then by Theorem~\ref{thm:inftyadmissible}, $\frakS$ is $\infty$-admissible. We verify that $T(\frakS) = \delta^{-2}\cdot \frakS$. Generalizing the proof of Lemma~\ref{lem:nistwo}, we see that
    \[\tau^{-1}\cdot\frakS = (T(\Phi_n),\Phi_1,\Phi_2,\dots,\Phi_{n-1})\]
    where $\tau = \sigma_{n-1}\cdots\sigma_1$ is one of the positive cyclic rotations. This is illustrated below for $n=5$. Since $\tau^n = \delta^2$, we are done.
    \[\begin{tikzpicture}[scale=1.2]
        \tikzset{
            strand/.style={black, line width=1.5pt},
            mask/.style={white, line width=4.5pt}
        }

        \draw[strand] (5, 0) .. controls (5, 0.5) and (1,0.5) .. (1, 1) node[at start, below] {$\Phi_5\rangle$} node[above, black] {$\langle T(\Phi_5),$} ;

        \draw[mask] (1, 0) .. controls (1, 0.5) and (2,0.5) .. (2, 1);
        \draw[mask] (2, 0) .. controls (2, 0.5) and (3,0.5) .. (3, 1);
        \draw[mask] (3, 0) .. controls (3, 0.5) and (4,0.5) .. (4, 1);
        \draw[mask] (4, 0) .. controls (4, 0.5) and (5,0.5) .. (5, 1);
        \draw[strand] (1, 0) .. controls (1, 0.5) and (2,0.5) .. (2, 1) node[at start, below] {$\langle\Phi_1,$} node[above, black] {$\Phi_1,$};
        \draw[strand] (2, 0) .. controls (2, 0.5) and (3,0.5) .. (3, 1) node[at start, below] {$\Phi_2,$} node[above, black] {$\Phi_2,$} ;
        \draw[strand] (3, 0) .. controls (3, 0.5) and (4,0.5) .. (4, 1) node[at start, below] {$\Phi_3,$} node[above, black] {$\Phi_3,$} ;
        \draw[strand] (4, 0) .. controls (4, 0.5) and (5,0.5) .. (5, 1) node[at start, below] {$\Phi_4,$} node[above, black] {$\Phi_4\rangle$} ;
    \end{tikzpicture}\]

    We now prove the converse. For each $i$, let $\Phi_i'$ be the right-orthogonal of $\oplus_{j\neq i}\Phi_j$. These subcategories have the property that $I_i^R$ is zero on $\Phi_{j}$ for $j\neq i$ and induces an equivalence on $\Phi_i'$. These are the subcategories that form the mutation,
    \[\delta^{-1}\cdot\frakS := (\Phi_n',\dots,\Phi_1').\]
    where for example, $\Phi_1' = \Phi_1$. 
    
    Now by assumption, $T(\Phi_n) = \Phi_n'$. By the property above, $I_i^RTI_n = 0$ for all $i<n$, and is an equivalence for $i=n$.

    Next, observe that as subcategories of $\Phi$,
    \[\langle T(\Phi_{n-1}),T(\Phi_n)\rangle = \langle \Phi_n',\Phi_{n-1}'\rangle.\]
    So by the property above, $I_i^RTI_{n-1} = 0$ for all $i<n-1$, and is an equivalence for $i=n-1$. Continuing in this fashion, we deduce that $T$ is conjugate-compatible for $\frakS$, as desired.
\end{proof}

Let $R\subset\bbC$ be a finite subset. We define non-abstract localized schobers on $(\bbC,R)$ by using the braid group action and a $\two$-categorical version of the construction in \cite{kapranovPerverseSchobers2015}. 

\begin{defn}\label{defn:cuts}
    A system of cuts $K$ on $(\bbC,R)$ is a choice of arcs connecting the points of $R$ to a point $\infty\in\bbC$ far away on the positive real axis. We let $\Cuts(\bbC,R)$ denote the space of systems of cuts. 
\end{defn}

The space $\Cuts(\bbC,R)$ is a torsor under a homotopical right action by the group $\Br_n$, acting by mutations of cuts, see \cite{kapranovPerverseSchobersAlgebra2020}*{Figure 8}.

Consider now the trivial local system $\ul{\sfAut}(n)$ of $\two$-categories over $\Cuts(\bbC,R)$ where each stalk is given by $\sfAut(n)$. The action of $\Br_n$ on $\sfAut(n)$ upgrades $\ul{\sfAut}(n)$ to a $\Br_n$-equivariant local system.

\begin{defn}\label{defn:nonabstract}
    We let $\olsftwoP(\bbC,R)$ be the $\two$-category of sections of $\ul{\sfAut}(n)$ that are invariant under $\Br_n$, and refer to an object of $\sftwoP(\bbC,R)$ as a localized perverse schober on $(\bbC,R)$.
\end{defn}

\subsection{Serre functors and global monodromy}\label{subsection:interpretation}
We relate our notion of abstract localized schobers to Serre functors on the one hand, and to the global monodromy of the Fukaya--Seidel category of a Lefschetz fibration on the other.

\begin{eg}\label{eg:aut2}
    We give a more explicit description of objects in $\sfAut(2)$. Recall from \cite{dyckerhoffSphericalAdjunctionsStable2021}*{Corollary 2.5.3} that a 2-term $\infty$-admissible SOD $\langle\calA,\calB\rangle$, is equivalent to a diagram $F:\calB\to\calA$ of categories where $F$ admits all adjoints. The mutation of this SOD by the full twist is abstractly given by the diagram $F^{RR}:\calB\to\calA$. Thus a conjugate-compatible autoequivalence $T$ is equivalent to the data of autoequivalences $T_\calA,T_\calB$ of $\calA$, $\calB$ respectively, together with an equivalence, $F\circ T_\calB\simeq T_{\calA}\circ F^{RR}$.

    A source of such objects come from Serre functors. Let $\calA = \QCoh(X),\calB = \QCoh(Y)$ where $X$ and $Y$ are smooth projective varieties over $k$. Let $\omega_X := \pi_X^!k$ denote the shifted dualizing sheaf on $X$ and $\omega_Y$ likewise, so that
    \[T_\calA := -\otimes\omega_X^{-1},T_\calB := -\otimes\omega_Y^{-1}\]
    are the inverse Serre functors. Let $F = f^*$ for some smooth morphism $f:X\to Y$. Then by Grothendieck duality, this gives an object of $\sfAut(2)$.
\end{eg}

An algebraic relationship between $\delta^2$ and the Serre functor generalizing Example~\ref{eg:aut2} can be found in \cite{kapranovPerverseSchobersAlgebra2020}.

\begin{prop}\label{prop:serrefunctor}
    Suppose that $\Phi$ admits a Serre functor $S$. Then, for any right-admissible SOD $\frakS$ of $\Phi$, $(\frakS;\Phi,S^{-1})$ is an abstract localized schober.
\end{prop}

\begin{proof}
    By \cite{kapranovPerverseSchobersAlgebra2020}*{Proposition 3.3.5}, $\frakS$ is $\infty$-admissible. By \cite{kapranovPerverseSchobersAlgebra2020}*{Proposition 3.3.11},
    \[\delta^2\cdot\frakS = S(\frakS),\]
    so we are done by Theorem~\ref{thm:alternate}.
\end{proof}

In this way, we can think of the characterization of Theorem~\ref{thm:alternate} as saying that $(\Phi,\frakS,T)$ is an abstract localized schober whenever $T$ behaves like an inverse Serre functor with respect to the SOD $\frakS$. We can find related observations in earlier works such as \cite{addingtonNewDerivedSymmetries2016}*{Proposition 1.1} and \cite{halpern-leistnerAutoequivalencesDerivedCategories2016}*{Example 4.17}.

Examples of localized perverse schobers also arise from Lefschetz fibrations. Given a Lefschetz fibration $f:X\to \bbC$ with $n$ critical values $R\subset\bbC$, consider its Fukaya--Seidel category $\Fuk(X,f)$. After making a choice of cuts, this admits an autoequivalence $T$ by simultaneously moving Lefschetz thimbles once around the boundary of $\bbC$. The resulting mutation of our choice of cuts is precisely described by the full twist element $\delta^2\in\Br_n$, so that $(\frakS;\Fuk(X,f),T)$ defines an abstract localized schober. One expects that this is independent of the choice of cuts, lifting this to an object of $\olsftwoP(\bbC,R)$.

This is described for example by Seidel in \cite{seidelSymplecticHomologyHochschild2008}*{\S 3, 4}, where $T$ is referred to as the global monodromy. Moreover, it is proposed that $T$ should be the Serre functor of $\Fuk(X,f)$ so that it is also an instance of Proposition~\ref{prop:serrefunctor}. In \S\ref{subsection:fourier}, we study the same topology in the context of Fourier transform and Stokes structures as introduced in \cite{kapranovPerverseSchobersAlgebra2020}.

\section{Perverse schobers on a disc}\label{section:adjaut}
We recall from \cite{kapranovPerverseSchobers2015} the following definition, which we will choose to refer to as abstract perverse schober on a disc.

\begin{defn}
    An abstract perverse schober on a disc with one singularity is a pair of stable categories, $\Phi,\Psi$ together with a spherical adjunction,
    \[S:\Phi\tofrom\Psi:R.\]
    An abstract perverse schober on a disc with $n$ singularities is a list of stable categories, $\Phi_1,\dots,\Phi_n,\Psi$, together with spherical adjunctions,
    \[S_i:\Phi_i\tofrom\Psi:R_i.\]
\end{defn}

\begin{notn}
    We denote by $\sfSph(n)$ the collection of abstract perverse schobers, which we later upgrade to a $\two$-category. We also write $\sfSph = \sfSph(1)$.
\end{notn}

We recall in \S\ref{subsection:enhancedFS} an operation defined in \cites{kapranovPerverseSchobersAlgebra2020,barbacoviCompositionTwoSpherical2021,harderPerverseSheavesCategories2019} which we will refer to as the enhanced Fukaya--Seidel functor,
\[\FS_{\sfSph}:\sfSph(n)\to\sfSph.\]
Geometrically, this operation can be thought of as merging the $n$ singularities into one. In the case that our abstract schober comes from a Lefschetz fibration, it is shown in \cite{christRelativeCalabiYauStructures2026}*{Theorem 6.1} that the vanishing cycle of $\FS_{\sfSph}$ agrees with the usual notion of Fukaya--Seidel category.

In \S\ref{subsection:simplicialadj} and \S\ref{subsection:simplicialsph}, we will upgrade $\sfSph(n)$ to a $\two$-category, $\FS_{\sfSph}$ to a functor of $\two$-categories, and the assignment
\[\sfSph(-):[n]\mapsto\sfSph(n)\]
to a simplicial $\two$-category in such a way that the unique active map $[1]\to [n]$ induces $\FS_{\sfSph}$. Moreover, the inert face maps $d_{n,0},d_{n,n}:\sfSph(n)\to\sfSph(n-1)$ are given by forgetting the first and last spherical adjunctions respectively so that $\sfSph(-)$ is evidently 1-Segal. In other words, we may view
\[\FS_{\sfSph}:\sfSph(2)\simeq\sfSph(1)\times_{\sfSph(0)}\sfSph(1)\to \sfSph(1)\]
as an associative binary operation.

In parallel, we construct a map of simplicial categories $\sfSph(-)\to\sfAut(-)$, and prove in Proposition~\ref{prop:activecartesian2} that over the active arrow $[1]\to[n]$, this induces a pullback square of $\two$-categories as follows.
\begin{equation}\label{equation:importantpullbacksquare}\begin{tikzcd}
    {\sfSph(n)} \pullback & \sfSph \\
    {\sfAut(n)} & \sfAut
    \arrow["{\FS_{\sfSph}}", from=1-1, to=1-2]
    \arrow[from=1-1, to=2-1]
    \arrow[from=1-2, to=2-2]
    \arrow["{\FS_{\sfAut}}", from=2-1, to=2-2]
\end{tikzcd}\end{equation}

In \S\ref{subsection:mutationsph}, we will use this to transfer the action of $\Br_n$ on $\sfAut(n)$ to one on $\sfSph(n)$, which is compatible with $\FS_{\sfSph}$, recovering a special case of the construction in the upcoming work \cite{acj}. We show that on objects, this agrees with the $\Br_n$-action on abstract perverse schobers defined in \cite{kapranovPerverseSchobers2015}*{(2.8)}. This leads to a definition of a $\two$-category of non-abstract perverse schobers, $\sftwoP(\bbC,R)$, together with a projection functor
\[\sftwoP(\bbC,R)\to\olsftwoP(\bbC,R)\]
and an enhanced Fukaya--Seidel functor
\[\FS_{\sftwoP}:\sftwoP(\bbC,R)\to \sftwoP(\bbC,0).\]

We record the following restatement of Proposition~\ref{prop:activecartesian2} for non-abstract schobers.
\begin{thm}\label{thm:merge}
    Let $R\subset\bbC$ be a finite subset. Then, there is a natural pullback square of $\two$-categories as follows.
    \[\begin{tikzcd}
        {\sftwoP(\bbC,R)} \pullback & \sftwoP(\bbC,0) \\
        {\olsftwoP(\bbC,R)} & \olsftwoP(\bbC,0)
        \arrow["{\FS_{\sftwoP}}", from=1-1, to=1-2]
        \arrow[from=1-1, to=2-1]
        \arrow[from=1-2, to=2-2]
        \arrow["{\FS_{\olsftwoP}}", from=2-1, to=2-2]
    \end{tikzcd}\]
\end{thm}

Finally, in \S\ref{subsection:fourier}, we will explain how Theorem~\ref{thm:merge} can be interpreted as the Fourier transform of perverse schobers studied in \cite{kapranovPerverseSchobersAlgebra2020}.

\begin{eg}\label{eg:waldhausen2}
    We continue with notation from Example~\ref{eg:waldhausen1}, where we constructed an object $(\frakS;S_{n-1}(F),T')$ of $\sfAut(n)$ from a spherical adjunction $F\adjto G$. Moreover, $T'$ was the twist of a spherical adjunction $F'\adjto G'$, so the square (\ref{equation:importantpullbacksquare}) implies that our object in fact lifts to $\sfSph(n)$.

    In \cite{gargGITRootStacks2026}, the construction $\sfSph\to \sfSph$ sending $F\adjto G\mapsto F'\adjto G'$ is interpreted as a pushforward
    \[f_*:\sftwoP(\bbC_z,0)\to \sftwoP(\bbC_y,0)\]
    along the map $f:\bbC_z\xto{z^n} \bbC_y$.

    We may perturb $z^n$ by a linear term so that it is Morse, $f_x(z) := z^n-xz$, where $x\in\bbC^\times$. This Morse function has $n$ critical values $R_x = \{y=0, y^{n-1} = x^n\}$. We interpret the lift to $\sfSph(n)$ as a factorization,
    \[\sftwoP(\bbC_z,0)\xto{(f_x)_*} \sftwoP(\bbC_y,R_x)\xto{\FS_{\sftwoP}} \sftwoP(\bbC_y,0).\]
    
    In Example~\ref{eg:waldhausen3}, we will explain how we may consider all $x$, including $x=0$ at once, and package this into a perverse schober on $\bbC^2_{x,y}$.
\end{eg}

\subsection{The enhanced Fukaya--Seidel functor}\label{subsection:enhancedFS}
In \cite{barbacoviCompositionTwoSpherical2021} and \cite{kapranovPerverseSchobersAlgebra2020}*{\S 3.5}, it is shown that given an object $(\Phi_i,\Psi,S_i,R_i)\in\sfSph(n)$, we may produce a category $\Phi$ called its Fukaya--Seidel category, equipped with an SOD of the form $\langle\Phi_1,\dots,\Phi_n\rangle$. This category is part of a spherical adjunction
\[S:\Phi\tofrom\Psi:R,\]
giving us a mapping of objects that we denote $\FS_{\sfSph}:\sfSph(n)\to\sfSph$, and that we refer to as the enhanced Fukaya--Seidel functor. Moreover, there are identifications $S\circ I_i\simeq S_i$.

All together, we have an operation,
\[\opname{Ind}:\begin{tikzcd}[ampersand replacement=\&]
	\& \Psi \& \\
	{\Phi_1} \& \cdots \& {\Phi_n}
	\arrow["{R_1}", shift left, from=1-2, to=2-1]
	\arrow["{R_n}", shift left, from=1-2, to=2-3]
	\arrow["{S_1}", shift left, from=2-1, to=1-2]
	\arrow["{S_n}", shift left, from=2-3, to=1-2]
\end{tikzcd}
\mapsto
\begin{tikzcd}[ampersand replacement=\&]
	\& \Psi \& \\
	\& {\langle\Phi_1,\dots,\Phi_n\rangle} \\
	{\Phi_1} \& \cdots \& {\Phi_n}
	\arrow["R"', shift right, from=1-2, to=2-2]
	\arrow["S"', shift right, from=2-2, to=1-2]
	\arrow["{I_1^R}"', shift right, from=2-2, to=3-1]
	\arrow["{I_n^R}"', shift right, from=2-2, to=3-3]
	\arrow["{I_1}"', shift right, from=3-1, to=2-2]
	\arrow["{I_n}"', shift right, from=3-3, to=2-2]
\end{tikzcd}\]

We will now construct the operation $\opname{Ind}$ carefully for $n=2$. We will find it helpful to perform the construction without any sphericity/invertibility conditions so that we obtain a mapping of objects
\[\opname{Ind}:\sfAdj(2)\to \sfSOD(2)\times_{\sfSt_k}\sfAdj.\]

Suppose that $S_i:\Phi_i\tofrom\Psi:R_i$ for $i=1,2$ are two adjunctions, defining an object of $\sfAdj(2)$. Define $\Phi_{12}$ to be the category of diagrams of the form,
\[\begin{tikzcd}
    {S_2\phi_2} & \\
    {S_1R_1S_2\phi_2} & {S_1\phi_1}
    \arrow[from=2-1, to=1-1]
    \arrow[from=2-1, to=2-2]
\end{tikzcd}\]
together with a lift of the horizontal arrow to a diagram $R_1S_2\phi_2\to \phi_1$ in $\Phi_1$, and such that the vertical arrow is a counit for $S_1\adjto R_1$.

\begin{rmk}\label{rmk:laxcolimit}
    This diagram category is described in a somewhat unwieldy manner. Observe that the objects of this diagram category can also be written as a triple $(\phi_1,\phi_2,\alpha:R_1S_2\phi_2\to \phi_1)$ where $\phi_i\in\Phi_i$. Thus this is equivalent to the lax colimit of diagram $\Phi_2\xto{F_{12}}\Phi_1$ in the spirit of \cite{christLaxAdditivity2025}. In the language of \cite{kapranovPerverseSchobersAlgebra2020}, it is the category of modules for a unitriangular monad. That said, the diagram category allows for a nice description of $S_{12}$ below.
\end{rmk}

Consider the pair of maps
\[S_{12}:\Phi_{12}\tofrom\Psi:R_{12},\]
where 
\[S_{12}\left(\begin{tikzcd}
    {S_2\phi_2} & \\
    {S_1R_1S_2\phi_2} & {S_1\phi_1}
    \arrow[from=2-1, to=1-1]
    \arrow[from=2-1, to=2-2]
\end{tikzcd}\right) := \colim\left(\begin{tikzcd}
    {S_2\phi_2} & \\
    {S_1R_1S_2\phi_2} & {S_1\phi_1}
    \arrow[from=2-1, to=1-1]
    \arrow[from=2-1, to=2-2]
\end{tikzcd}\right),\]
and where
\[R_{12}(\psi) := \begin{tikzcd}
    {S_2R_2\psi} & \\
    {S_1R_1S_2R_2\psi} & {S_1R_1\psi}
    \arrow[from=2-1, to=1-1]
    \arrow[from=2-1, to=2-2]
\end{tikzcd}\]
where the horizontal arrow is the counit for $S_2\adjto R_2$.

\begin{lem}\label{lem:twocats}
    There is a natural adjunction, $S_{12}\adjto R_{12}$.
\end{lem}

\begin{proof}
    Given a diagram $\phi_{12}$ as above, we have functorial identifications
    \begin{align*}
        \Hom(S_{12}\phi_{12},\psi)&\simeq \Hom(S_1\phi_1,\psi)\times_{\Hom(S_1R_1S_2\phi_2,\psi)} \Hom(S_2\phi_2,\psi)\\
        &\simeq \Hom(\phi_1,R_1\psi)\times_{\Hom(R_1S_2\phi_2,R_1\psi)} \Hom(\phi_2,R_2\psi)\\
        &\simeq \Hom(\phi_{12},R_{12}\psi).
    \end{align*}
\end{proof}

\begin{lem}\label{lem:twocatssod}
    The category $\Phi_{12}$ admits a right-admissible SOD via the following subcategories. Here, we denote elements of $\Phi_{12}$ by triples $(\phi_1,\phi_2,\alpha:R_1S_2\phi_2\to \phi_1)$, as discussed in Remark~\ref{rmk:laxcolimit}.
    \begin{align*}
        I_1: \Phi_1&\to \Psi\\
        \phi_1&\mapsto (\phi_1,0,0)\\
        I_2: \Phi_2&\to \Psi\\
        \phi_2&\mapsto (R_1S_2\phi_2,\phi_2,\id_{R_1S_2\phi_2})
    \end{align*}
    The right-adjoints are given by projections $I_i^R:(\phi_1,\phi_2,\alpha)\mapsto \phi_i$. Moreover, there are canonical identifications $S_{12}\circ I_i\simeq S_i$.
\end{lem}

\begin{proof}
    Explicit calculation.
\end{proof}

Recall that for $\calF\in\sfAdj(2)$, we may recover $\calF$ from $\opname{Ind}(\calF)$. Thus $\opname{Ind}$ is an injective mapping. The following computes the image of this mapping.

\begin{lem}\label{lem:twocatsequivalence}
    The image of $\opname{Ind}:\sfAdj(2)\to\sfSOD(2)\times_{\sfSt_k}\sfAdj$ consists of adjunctions $S:\Phi\tofrom\Psi:R$ equipped with a 2-term SOD $\frakS$ of $\Phi$ such that $\coCone(1\to RS)$ is weakly conjugate-compatible with respect to $\frakS$.
\end{lem}

\begin{proof}
    Given an object $\calF\in\sfAdj(2)$, we show that $\coCone(\id\to R_{12}S_{12})$ is weakly conjugate-compatible with respect to the SOD $\frakS$ on $\Phi_{12}$ as constructed in Lemma~\ref{lem:twocatssod}. This is by direct calculation - see that
    \[I_1^R\coCone(\id\to R_{12}S_{12})I_2\simeq \coCone(I_1^RI_2\to R_1S_2) = 0.\]

    We now prove the other containment. We consider an object of $\sfSOD(2)\times_{\sfSt_k}\sfAdj$ given as follows.
    \[\begin{tikzcd}[ampersand replacement=\&]
        \& \Psi \& \\
        \& {\langle\Phi_1,\Phi_2\rangle} \\
        {\Phi_1} \&  \& {\Phi_n}
        \arrow["R"', shift right, from=1-2, to=2-2]
        \arrow["S"', shift right, from=2-2, to=1-2]
        \arrow["{I_1^R}"', shift right, from=2-2, to=3-1]
        \arrow["{I_2^R}"', shift right, from=2-2, to=3-3]
        \arrow["{I_1}"', shift right, from=3-1, to=2-2]
        \arrow["{I_2}"', shift right, from=3-3, to=2-2]
    \end{tikzcd}\]
    Consider the object $\calF\in\sfAdj(2)$ given by $(\Phi_i,\Psi,SI_i,I_i^RR)$. Because $\coCone(\id\to RS)$ is conjugate-compatible with respect to $\frakS$, the natural map
    \[I_1^RI_2\to I_1^RRSI_2\]
    is an equivalence. Using this, we deduce that $\Phi_{12}$ is equivalent to the category of diagrams of form
    \[\begin{tikzcd}
        {SI_2\phi_2} & \\
        {SI_1I_1^RI_2\phi_2} & {SI_1\phi_1}
        \arrow[from=2-1, to=1-1]
        \arrow[from=2-1, to=2-2]
    \end{tikzcd}\]
    together with a lift of the horizontal arrow to a diagram $I_1I_1^RI_2\phi_2\to I_1\phi_1$ in $\Phi$, and such that the vertical arrow is a counit for $I_1\adjto I_1^R$. Thus $S\circ-$ can be factored out and this is equivalent to the diagram category described in Lemma~\ref{lem:twosubcats}, so that $\Phi_{12}\simeq \Phi$. Moreover, under this identification, the adjunction of Lemma~\ref{lem:twocats} is identified with $S\adjto R$ and the SOD of Lemma~\ref{lem:twocatssod} is identified with $\frakS$.
\end{proof}

\begin{lem}\label{lem:twocatsspherical}
    The image of $\opname{Ind}$ restricted to $\sfSph(2)$ consists of \textit{spherical} adjunctions $S:\Phi\tofrom\Psi:R$ equipped with a 2-term SOD $\frakS$ of $\Phi$ such that $\coCone(\id\to RS)$ is conjugate-compatible with respect to $\frakS$.
\end{lem}

\begin{proof}
    We first prove that given $\calF\in\sfSph(2)$, $\opname{Ind}(\calF)$ has the claimed properties. The case $n=2$ of \cite{kapranovPerverseSchobersAlgebra2020}*{Proposition 3.5.6} shows that $S_{12}\adjto R_{12}$ is spherical. This also follows by a direct calculation. See that 
    \[I_i^R\coCone(\id\to R_{12}S_{12})I_i\simeq \coCone(\id\to R_iS_i)\]
    is invertible. Together with Lemma~\ref{lem:twocatsequivalence}, this proves conjugate-compatibility.

    For the other containment, we must show that if $(\frakS;\Phi,\Psi,S,R)\in \sfSOD(2)\times_{\sfSt_k}\sfAdj$ has the claimed properties, then $SI_i:\Phi_i\tofrom \Psi:R_i$ is spherical. The computation above shows that $\coCone(\id\to R_iS_i)$ is invertible for each $i$. We extract from the proof of \cite{kapranovPerverseSchobersAlgebra2020}*{Proposition 3.5.6}, or show by direct calculation, that
    \[\Cone(SR\to \id)\simeq \Cone(S_1R_1\to \id)\Cone(S_2R_2\to \id).\]
    Since $\Cone(SR\to \id)$ is invertible, this shows that $\Cone(S_1R_1\to \id)$ has a right-inverse, and $\Cone(S_2R_2\to \id)$ has a left-inverse.

    By Theorem~\ref{thm:inftyadmissible}, we may mutate $\frakS$ to $\sigma\cdot\frakS = (\Phi_2',\Phi_1)$, and we denote the new inclusion by $I_2'$. This defines a new object $(\sigma\cdot\frakS;\Phi,\Psi,S,R)$ and the same argument above shows that $\Cone(S_1R_1\to \id)$ has a right-inverse. So we conclude that $\Cone(S_1R_1\to \id)$ is invertible. The other mutation shows that $\Cone(S_2R_2\to \id)$ is invertible.
\end{proof}

\begin{rmk}
    A version of this (in one direction) appears in \cite{halpern-leistnerAutoequivalencesDerivedCategories2016}*{Theorem 4.4}. This statement contains a condition equivalent to our conjugate-compatibility condition.
\end{rmk}

For $n>2$, we will construct $\opname{Ind}$ by abstract means in Theorem~\ref{prop:bdjisadj}.

\subsection{\texorpdfstring{$\sfAdj(n)$}{Adj(n)} as a simplicial \texorpdfstring{$\two$}{2}-category}\label{subsection:simplicialadj}
We construct the simplicial $\two$-category $\sfAdj(-)$. We do so by constructing an a priori different simplicial $\two$-category $\sfBdj(-)$, and then comparing $\sfBdj(n)$ to $\sfAdj(n)$.

Consider the functor
\[\sfAdj \to \sfEnd,\]
given by $(\Phi,\Psi,S,R)\mapsto (\Phi,\coCone(\id\to RS))$. Consider the simplicial $\two$-category
\[\sfBdj'(-) := \sfEnd'(-)\times_{\sfEnd}\sfAdj.\]
The objects of $\sfBdj'(n)$ are given by a category $\Phi$ with a partial SOD $\frakS$ and an adjunction $(\Phi,\Psi,S,R)$ such that $\coCone(\id\to RS)$ is weakly conjugate-compatible with respect to $\frakS$.

For each $n$, consider the full sub-$\two$-category $\opname{I}_n:\sfBdj(n)\inclto \sfBdj'(n)$ given by those objects where the partial SOD is an SOD.

\begin{lem}\label{lem:rightadjointbdj}
    The inclusion functor $\opname{I}_n:\sfBdj(n) \inclto \sfBdj'(n)$ admits a right adjoint $\opname{I}_n^R$, which is given on objects by
    \[(\frakS;\Phi,\Psi,S,R)\mapsto (\frakS;\langle\Phi_1,\dots,\Phi_n\rangle,\Psi,SI,I^RR),\]
    where $I:\langle\Phi_1,\dots,\Phi_n\rangle\inclto \Phi$.
\end{lem}

\begin{proof}
    This is proven similarly to Lemma~\ref{lem:keyrightadjoint} and Lemma~\ref{lem:rightadjointend}, by reducing to the case $n=1$.
\end{proof}

\begin{cor}\label{prop:simplicialbdj}
    The assignment $[n]\mapsto \sfBdj(n)$ upgrades to a simplicial $\two$-category. For $0\leq i\leq n$, the degeneracy map $s_{n,i}$ inserts a $0$ between $\Phi_i$ and $\Phi_{i+1}$;
    \begin{align*}
        s_{n,i}:\sfBdj(n)&\to\sfBdj(n+1)\\
        (\frakS;\Phi,\Psi,S,R)&\mapsto ((\Phi_1,\dots,\Phi_i,0,\Phi_{i+1},\dots\Phi_n);\Phi,\Psi,S,R).
    \end{align*}
    The inert face map $d_{n,0}$ removes $\Phi_1$ from the SOD to get $\Phi_{2,\dots,n} := \langle\Phi_2,\dots,\Phi_n\rangle\xto{I_{2,\dots,n}}\Phi$;
    \begin{align*}
        d_{n,0}:\sfBdj(n)&\to\sfBdj(n-1)\\
        (\frakS;\Phi,\Psi,S,R)&\mapsto ((\Phi_2,\dots,\Phi_n);\Phi_{2,\dots,n},\Psi,SI_{2,\dots,n},I_{2,\dots,n}^RR),
    \end{align*}
    and the other inert face map $d_{n,n}$ similarly removes $\Phi_n$ from the SOD. For $0<i<n$, the active face map is given by merging $\Phi_{i}$ and $\Phi_{i+1}$ in the SOD into $\Phi_{i,i+1} := \langle\Phi_i,\Phi_{i+1}\rangle$;
    \begin{align*}
        d_{n,i}:\sfBdj(n)&\to\sfBdj(n-1)\\
        (\frakS;\Phi,\Psi,S,R)&\mapsto ((\Phi_1,\dots,\Phi_{i-1},\Phi_{i,i+1},\Phi_{i+2},\dots\Phi_n);\Phi,\Psi,S,R).
    \end{align*}
\end{cor}

\begin{proof}
    We apply Lemma~\ref{lem:descendingsimplicialobjects}. We omit the verification that the resulting lax-commuting squares are commuting squares.
\end{proof}

\begin{prop}
    The simplicial $\two$-category $\sfBdj(-)$ is 2-Segal.
\end{prop}

\begin{proof}
    The strategy used in the proof of Proposition~\ref{prop:segalsod} works, using that the map $\sfBdj'(-)\to\sfBdj(-)$ is Cartesian over $\Delta_{\act}$.
\end{proof}

We now construct a projection $\sfBdj(-)\to\sfEnd(-)$ of simplicial $\two$-categories.

\begin{prop}\label{prop:simplicialbdjtoend}
    There exists a map $\pi(-): \sfBdj(-)\to\sfEnd(-)$ of simplicial $\two$-categories, which for each $n$ is given object-wise by
    \[(\frakS;\Phi,\Psi,R,S)\mapsto (\frakS;\Phi,\coCone(\id\to RS)).\]
\end{prop}

\begin{proof}
    We begin with the simplicial map $\pi'(-):\sfBdj'(-)\to\sfEnd'(-)$. We apply Lemma~\ref{lem:descendingsimplicialmaps} to define functors
    \[\pi(n):\sfBdj(n)\to \sfEnd(n)\]
    which are given on objects by the formulae stated. For each $n$, the relevant lax-commuting square evidently commutes so we obtain the desired map $\pi(-)$.
\end{proof}

\begin{prop}\label{prop:activecartesian4}
    The map $\pi(-): \sfBdj(-)\to\sfEnd(-)$ of simplicial $\two$-categories is Cartesian over $\Delta_{\act}$.
\end{prop}

\begin{proof}
    We are to show that for each $n$, the following commutative square is a pullback.
    \[\begin{tikzcd}[ampersand replacement=\&]
        {\sfBdj(n)} \& {\sfBdj(1)} \\
        {\sfEnd(n)} \& {\sfEnd(1)}
        \arrow[from=1-1, to=1-2]
        \arrow[from=1-1, to=2-1]
        \arrow[from=1-2, to=2-2]
        \arrow[from=2-1, to=2-2]
    \end{tikzcd}\]
    This is clear by unravelling the definitions.
\end{proof}

We now study $\sfAdj(n)$. Observe first that there is a natural identification,
\[\sfAdj(n)\simeq \sfAdj(1)\times_{\sfAdj(0)}\cdots\times_{\sfAdj(0)}\sfAdj(1),\]
where $\sfAdj(1)\to \sfAdj(0)$ is given by the functor $(\Phi,\Psi,S,R)\mapsto \Psi$. Thus $[n]\mapsto \sfAdj(n)$ organize into a presheaf of $\two$-categories over $\Delta_{\inert}$.

We will construct an equivalence, $\sfBdj(-)\resto{\Delta_{\inert}}\simeq\sfAdj(-)$, thereby upgrading $\sfAdj(-)$ to a 1-Segal simplicial $\two$-category.

For each fixed $n$, consider the map of objects $\sfBdj(n)\to\sfAdj(n)$ given by
\[(\frakS;\Phi,\Psi,S,R)\mapsto (\Phi_i,\Psi,SI_i,I_i^RR),\]
where $I_i:\Phi_I\inclto \Phi$ are the terms of the SOD. To construct this as a functor, we pass to diagram categories as follows.

Let $\frakn(n)$ be the ordinary 1-category with objects $1,\dots,n,\infty$ and arrows $i\to \infty$. Let $\frakh(n)$ be the ordinary 1-category with objects $1,\dots,n,0,\infty$ and arrows $i\to \infty$, $\infty\to 0$. There is a natural functor $\frakn(n)\to\frakh(n)$, depicted below.
\[\begin{tikzcd}[ampersand replacement=\&]
	\& \infty \&\&\& \infty \& \\
	\&\& {} \& {} \& 0 \\
	1 \& \cdots \& n \& 1 \& \cdots \& n
	\arrow[from=2-3, to=2-4]
	\arrow[from=2-5, to=1-5]
	\arrow[from=3-1, to=1-2]
	\arrow[from=3-2, to=1-2]
	\arrow[from=3-3, to=1-2]
	\arrow[from=3-4, to=2-5]
	\arrow[from=3-5, to=2-5]
	\arrow[from=3-6, to=2-5]
\end{tikzcd}\]

By construction, $\sfAdj(n)\to \sfFun(\frakn(n),\sfSt_k)$ and likewise, $\sfBdj(n)\to \sfFun(\frakh(n),\sfSt_k)$ are locally full sub-$\two$-categories. Pullback along $\frakn(n)\to \frakh(n)$ induces the claimed functors $\sfBdj(n)\to\sfAdj(n)$.

Observe further that $\sfFun(\frakn(-),\sfSt_k)$, $\sfFun(\frakh(-),\sfSt_k)$ organize into presheaves of $\two$-categories over $\Delta_{\inert}$ because $\frakn(-)$ and $\frakh(-)$ organize into precosheaves. This induces the aforementioned presheaf $\sfAdj(-)$. It also induces the presheaf $n\mapsto \sfBdj(n)$, which agrees with $\sfBdj\resto{\Delta_{\inert}}$, and thus we have the desired map,
\[\opname{C}(-):\sfBdj(-)\resto{\Delta_{\inert}}\to\sfAdj(-).\]

\begin{prop}\label{prop:bdjisadj}
    The map $\opname{C}(n):\sfBdj(n)\to\sfAdj(n)$ is an equivalence for each $n$. Thus $\sfAdj(-)$ extends to a 1-Segal simplicial $\two$-category.
\end{prop}

\begin{proof}
    The case $n=0$ and $n=1$ are trivial. We prove the result by induction on $n$, reducing to the base case $n=2$, which for the moment we assume.
    
    Let $n\geq3$ and assume the result for all smaller cases. Consider the commuting diagram in $\Delta$ of the form
    \begin{equation}\label{equation:deltadiagram}\begin{tikzcd}[ampersand replacement=\&]
        {[n]}\pullback \& {[n-1]}\pullback \& {[n-2]} \\
        {[2]} \& {[1]} \& {[0]}
        \arrow[from=1-2, to=1-1]
        \arrow["j", from=1-3, to=1-2]
        \arrow[from=2-1, to=1-1]
        \arrow["i", from=2-2, to=1-2]
        \arrow["a", from=2-2, to=2-1]
        \arrow[from=2-3, to=1-3]
        \arrow[from=2-3, to=2-2]
    \end{tikzcd}
    \end{equation}
    where both squares are pushouts, $a$ is active, $i$ is the inert map with image $\{0,1\}$, and $j$ is the inert map with image $\{1,\dots,n-1\}$. Note that the outer square is also a pushout square involving only inert maps.

    By applying $\sfAdj(-)$ to the inert arrows, and $\sfBdj(-)$ to everything, this induces the following commutative diagram.
    \[\begin{tikzcd}[ampersand replacement=\&]
        {\sfBdj(n)}\pullback \& {\sfBdj(n-1)} \& {\sfBdj(n-2)} \& \\
        {\sfBdj(2)} \& {\sfBdj(1)} \& {\sfBdj(0)} \\
        \& {\sfAdj(n)} \& {\sfAdj(n-1)}\pullback \& {\sfAdj(n-2)} \\
        \& {\sfAdj(2)} \& {\sfAdj(1)} \& {\sfAdj(0)}
        \arrow[from=1-1, to=1-2]
        \arrow[from=1-1, to=2-1]
        \arrow[from=1-1, to=3-2]
        \arrow[from=1-2, to=1-3]
        \arrow[from=1-2, to=2-2]
        \arrow[from=1-2, to=3-3]
        \arrow[from=1-3, to=2-3]
        \arrow[from=1-3, to=3-4]
        \arrow[from=2-1, to=2-2]
        \arrow[from=2-1, to=4-2]
        \arrow[from=2-2, to=2-3]
        \arrow[from=2-2, to=4-3]
        \arrow[from=2-3, to=4-4]
        \arrow["{\opname{Y}}"', dashed, from=3-2, to=3-3]
        \arrow[from=3-2, to=4-2]
        \arrow[from=3-3, to=3-4]
        \arrow[from=3-3, to=4-3]
        \arrow[from=3-4, to=4-4]
        \arrow["{\opname{X}}"', dashed, from=4-2, to=4-3]
        \arrow[from=4-3, to=4-4]
    \end{tikzcd}\]
    
    To navigate this diagram, we will interpret the diagonal arrows as coming out of the page. The back-left square is a pullback by the 2-Segal property of $\sfBdj(-)$, and front-right square is a pullback by the 1-Segal property of $\sfAdj(-)$. The diagonal maps come from $\opname{C}(-)$ and all but $\opname{C}(n)$ is an equivalence by inductive hypothesis. These equivalences induce the dotted map denoted $\opname{X}$. Finally, the dotted map $\opname{Y}$ can be completed using the pullback square on the front-right.

    The front-outer square is a pullback by the 1-Segal property of $\sfAdj(-)$ applied to the outer square of (\ref{equation:deltadiagram}). Thus the front-left square is a pullback, and we deduce that $\opname{C}(n)$ is an equivalence as desired.

    It remains to prove the base case $n=2$. We embed our two $\two$-categories into slightly larger $\two$-categories as follows. Let $\frakn'$ and $\frakh'$ be the small $\two$-categories as shown, together with a map $f:\frakn'\to\frakh'$.
    \[\begin{tikzcd}[ampersand replacement=\&]
        \& \infty \&\&\& \infty \& \\
        \&\& {} \& {} \& 0 \\
        1 \&\& 2 \& 1 \&\& 2
        \arrow["f", from=2-3, to=2-4]
        \arrow[from=2-5, to=1-5]
        \arrow[from=3-1, to=1-2]
        \arrow[""{name=0, anchor=center, inner sep=0}, from=3-3, to=1-2]
        \arrow[from=3-3, to=3-1]
        \arrow[from=3-4, to=2-5]
        \arrow[""{name=1, anchor=center, inner sep=0}, from=3-6, to=2-5]
        \arrow[from=3-6, to=3-4]
        \arrow[between={0}{0.8}, Rightarrow, from=3-1, to=0]
        \arrow[between={0}{0.8}, Rightarrow, from=3-4, to=1]
    \end{tikzcd}\]
    We may embed $\opname{J}:\sfAdj(2)\to \sfFun(\frakn',\sfSt_k)$ as a locally full sub-$\two$-category.
    \[\begin{tikzcd}[ampersand replacement=\&]
        \& \Psi \&\&\& \Psi \& \\
        \&\& {} \& {} \\
        {\Phi_1} \&\& {\Phi_2} \& {\Phi_1} \&\& {\Phi_2}
        \arrow["{R_1}", shift left, from=1-2, to=3-1]
        \arrow["{R_2}", shift left, from=1-2, to=3-3]
        \arrow[maps to, from=2-3, to=2-4]
        \arrow["{S_1}", shift left, from=3-1, to=1-2]
        \arrow["{S_2}", shift left, from=3-3, to=1-2]
        \arrow["{S_1}", from=3-4, to=1-5]
        \arrow[""{name=0, anchor=center, inner sep=0}, "{S_2}"', from=3-6, to=1-5]
        \arrow["{R_1S_2}", from=3-6, to=3-4]
        \arrow[between={0}{0.8}, Rightarrow, from=3-4, to=0]
    \end{tikzcd}\]
    Indeed, it is the sub-$\two$-category on objects $i\mapsto \calA_i$ such that the arrows $F_i:\calA_i\to \calA_\infty$ are left adjoints, and such that the lax-commuting triangle becomes a commuting triangle once we pass to the right adjoint of $F_1$. Moreover, given two objects $i\mapsto \calA_i,\calB_i$, we must restrict to those morphisms which are adjointable with over the arrows $i\to \infty$ (see \S\ref{subsection:mndadj}). Likewise, we may embed $\opname{K}:\sfBdj(2)\to \sfFun(\frakh',\sfSt_k)$ as a locally full sub-$\two$-category.

    Observe now that the pullback functor $f^*:\sfFun(\frakh',\sfSt_k)\to \sfFun(\frakn',\sfSt_k)$ restricts to these locally full sub-$\two$-categories. Indeed, given an object $\calF' := (\Phi_1',\Phi_2';\Phi',\Psi',S',R')$ with embeddings $I_i':\Phi_i'\inclto \Phi$, the image under $f^*\circ\opname{K}$ is as follows.
    
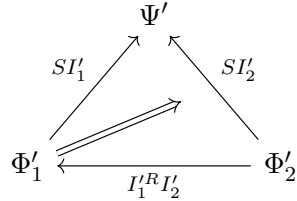
\[\begin{tikzcd}[ampersand replacement=\&]
        \& {\Psi'} \& \\
        \\
        {\Phi_1'} \&\& {\Phi_2'}
        \arrow["{SI_1'}", from=3-1, to=1-2]
        \arrow[""{name=0, anchor=center, inner sep=0}, "{SI_2'}"', from=3-3, to=1-2]
        \arrow["{I_1'^RI_2'}", from=3-3, to=3-1]
        \arrow[between={0}{0.8}, Rightarrow, from=3-1, to=0]
    \end{tikzcd}\]
    For this to factor through $\opname{J}$ we require that $I_1'^RRSI_2'\to I_1'^RI_2'$ is an equivalence, which follows as its cocone is $I_1'^R\coCone(1\to R'S')I_2'=0$. The verification that morphisms of objects also factor through $\sfAdj(2)$ is immediate.
    
    The resulting functor is identified with $\opname{C}(2)$ (by considering the functors $\frakn\to\frakn',\frakh\to\frakh'$), thus we have the following diagram of $\two$-categories.
    \[\begin{tikzcd}[ampersand replacement=\&]
        {\sfBdj(2)} \& {\sfAdj(2)} \\
        {\sfFun(\frakh',\sfSt_k)} \& {\sfFun(\frakn',\sfSt_k)}
        \arrow["\opname{C}(2)", from=1-1, to=1-2]
        \arrow["\opname{K}", from=1-1, to=2-1]
        \arrow["\opname{J}", from=1-2, to=2-2]
        \arrow["f^*", from=2-1, to=2-2]
    \end{tikzcd}\]
    Importantly, the vertical maps are locally fully faithful - i.e. 
    By Theorem~\ref{thm:laxfibrations}, $f^*$ admits a left adjoint $f_!$, which moreover admits an explicit description. Given $\calF\in\sfAdj(2)$, we can compute $f_! \circ \opname{J}(\calF)$ and see that the object agrees with $\opname{K}(\opname{Ind}(\calF))$, where $\opname{Ind}(\calF)$ was constructed on the level of objects in \S\ref{subsection:enhancedFS}. Similarly, we see that given a map $P:\calF_1\to\calF_2$, the map $f_! \circ \opname{J}(P)$ also lifts along $\opname{K}$.

    Thus $\opname{Ind}$ canonically extends to a functor such that the following commutes.
    \[\begin{tikzcd}[ampersand replacement=\&]
        {\sfBdj(2)} \& {\sfAdj(2)} \\
        {\sfFun(\frakh',\sfSt_k)} \& {\sfFun(\frakn',\sfSt_k)}
        \arrow["{\opname{K}}"', from=1-1, to=2-1]
        \arrow["{\opname{Ind}}"', from=1-2, to=1-1]
        \arrow["{\opname{J}}", from=1-2, to=2-2]
        \arrow["{f_!}", from=2-2, to=2-1]
    \end{tikzcd}\]

    For notational convenience, we denote the four $\two$-categories in the diagram above by $\sfA,\sfB,\sfN',\sfH'$. Let $\calF\in\sfA=\sfAdj(2)$, $\calF'\in\sfB=\sfBdj(2)$. Then, we have a diagram of mapping categories as follows.
    \[\begin{tikzcd}[ampersand replacement=\&]
        {\sfB(\calF',\opname{Ind}(\calF))} \& {\sfA(\opname{C}(2)(\calF'),\calF)} \\
        {\sfH'(\opname{K}(\calF'),f_!\opname{J}(\calF))} \& {\sfN'(f^*\opname{K}(\calF'),\opname{J}(\calF))}
        \arrow["{j'}"', hook, from=1-1, to=2-1]
        \arrow["j", hook, from=1-2, to=2-2]
        \arrow["\sim", tail reversed, from=2-2, to=2-1]
    \end{tikzcd}\]
    The vertical maps are full subcategories because $\opname{J},\opname{K}$ are locally fully faithful. Verifying that the subcategories agree under the adjunction $f_!\adjto f^*$, we obtain an adjunction $\opname{Ind}\adjto\opname{C}(2)$. Moreover, $\opname{Ind}$ is 2-fully faithful because $f_!$ is.

    It remains to prove that $\opname{Ind}$ is essentially surjective. This is the calculation outlined in Lemma~\ref{lem:twocatsequivalence}.
\end{proof}

\subsection{\texorpdfstring{$\sfSph(n)$}{Sph(n)} as a simplicial \texorpdfstring{$\two$}{2}-category}\label{subsection:simplicialsph}

We now reinstate the sphericity conditions and construct the simplicial $\two$-category $\sfSph(-)$. We then construct an action of the braid group on $\sfSph(n)$ and show that this agrees with the braid group action defined at the level of objects in \cite{kapranovPerverseSchobers2015}. We use this to define the $\two$-category $\sftwoP(\bbC,R)$ of perverse schobers on $(\bbC,R)$.

\begin{prop}\label{prop:segalsph}
    The sub-$\two$-categories $\sfSph(n)\inclto \sfAdj(n)$ define a sub-simplicial $\two$-category which is 1-Segal.
\end{prop}

\begin{proof}
    It follows from the calculation of \cite{barbacoviCompositionTwoSpherical2021} or \cite{kapranovPerverseSchobersAlgebra2020}*{Proposition 3.5.6} that the ``monoidal product''
    \[d_{2,1}:\sfAdj(2)\to\sfAdj\]
    sends abstract perverse schobers to abstract perverse schobers. The same is clearly true for the ``monoidal unit''
    \[d_{0,0}:\sfAdj(0)\to\sfAdj.\]
    Moreover, there is an identification,
    \[\sfSph(n)\simeq \sfSph(1)\times_{\sfSph}\cdots\times_{\sfSph}\sfSph(1)\]
    compatibly with the analogous identification for $\sfAdj(n)$. The result now follows.
\end{proof}

\begin{prop}\label{prop:activecartesian2}
    There is a natural map of simplicial $\two$-categories $\pi(-):\sfSph(-)\to\sfAut(-)$ which is Cartesian over $\Delta_{\act}$. In particular, for each $n$, there is a natural pullback square as follows.
    \[\begin{tikzcd}
        {\sfSph(n)} \pullback & \sfSph \\
        {\sfAut(n)} & \sfAut
        \arrow["{\FS_{\sfSph}}", from=1-1, to=1-2]
        \arrow[from=1-1, to=2-1]
        \arrow[from=1-2, to=2-2]
        \arrow["{\FS_{\sfAut}}", from=2-1, to=2-2]
    \end{tikzcd}\]
\end{prop}

\begin{proof}
    We first show that the composition $\sfSph(n)\inclto \sfAdj(n)\to \sfEnd(n)$ factors through $\sfAut(n)$. This is trivial for $n=0,1$. For $n=2$, this follows from one direction of Lemma~\ref{lem:twocatsspherical}. For $n>2$, this follows by induction using that $\sfAut(-)$ is 2-Segal. So the map $\pi:\sfAdj(-)\simeq\sfBdj(-)\to\sfEnd(-)$ of Proposition~\ref{prop:simplicialbdjtoend} induces the desired map.

    By Proposition~\ref{prop:activecartesian4}, $\pi(-):\sfAdj(-)\to\sfEnd(-)$ is Cartesian over $\Delta_{\act}$. We deduce that over $a:[1]\to[2]$, the natural map
    \[\sfSph(2)\to\sfAut(2)\times_{\sfAut}\sfSph\]
    is 2-fully faithful. The other direction of Lemma~\ref{lem:twocatsspherical} proves that this map is essentially surjective, proving the desired result for $a:[1]\to[2]$. Let $b:[n-1]\to [n]$ be any active arrow. The 2-Segal property gives us a map of pullback squares of $\two$-categories as follows, where the left vertical arrows come from $b$ and the right from $a$.
    \[\begin{tikzcd}[ampersand replacement=\&]
        {\sfSph(n)}\pullback \& {\sfSph(2)} \\
        {\sfSph(n-1)} \& {\sfSph(1)}
        \arrow[from=1-1, to=1-2]
        \arrow[from=1-1, to=2-1]
        \arrow[from=1-2, to=2-2]
        \arrow[from=2-1, to=2-2]
    \end{tikzcd}
    \to
    \begin{tikzcd}[ampersand replacement=\&]
        {\sfAut(n)}\pullback \& {\sfAut(2)} \\
        {\sfAut(n-1)} \& {\sfAut(1)}
        \arrow[from=1-1, to=1-2]
        \arrow[from=1-1, to=2-1]
        \arrow[from=1-2, to=2-2]
        \arrow[from=2-1, to=2-2]
    \end{tikzcd}\]
    Moreover, this map is Cartesian over the right vertical arrow, so by a diagram chase, it is Cartesian over the left arrow.

    This proves that $\sfSph(-)\to\sfAut(-)$ is Cartesian over $\Delta_{\act}$.
\end{proof}

\subsection{Braid group action on abstract schobers}\label{subsection:mutationsph}
We now construct a $\two$-categorical braid group action on $\sfSph(n)$. We note that a $\two$-categorical braid group action has already been announced to be in the upcoming work of Abell\'an--Christ--Jasso \cite{acj}. Their approach considers all Lagrangian skeleta of $(\bbC,R)$ which includes all systems of cuts, as well as more general skeleta where multiple arcs can be incident to one singularity. In comparison, we will only consider those skeleta in $\Cuts(\bbC,R)$ (see Definition~\ref{defn:cuts}), but the novelty is that this action will be evidently compatible with abstract localized schobers.

\begin{cor}\label{cor:action2}
    There is a natural action of the braid group $\Br_n$ on $\sfSph(n)$, which respects the pullback square of Proposition~\ref{prop:activecartesian2}. Moreover, this agrees at the level of objects with the action defined in \cite{kapranovPerverseSchobers2015}*{(2.8)}.
\end{cor}

\begin{proof}
    We can pullback the action of $\Br_n$ on $\sfAut(n)$ via the pullback square of Proposition~\ref{prop:activecartesian2}, where we take $\Br_n$ to act trivially on $\sfAut$ and $\sfSph$.
    
    We compute the resulting braid group action on $\sfSph(2)$. Let $(\Phi_i,\Psi,S_i,R_i)\in\sfSph(2)$. The action of the negative generator $\sigma^{-1}\in\Br_2$ on the corresponding element of $\sfAut(2)\times_{\sfAut}\sfSph$ is as follows.
    \[\begin{tikzcd}[ampersand replacement=\&]
        \& \Psi \&\&\& \Psi \\
        \& \Phi \& {} \& {} \& \Phi \\
        {\Phi_1} \& {\Phi_1} \&\& {\Phi_2} \& {\Phi_1}
        \arrow["S", from=2-2, to=1-2]
        \arrow["{\sigma^{-1}\cdot-}", maps to, from=2-3, to=2-4]
        \arrow["S", from=2-5, to=1-5]
        \arrow["{I_1}", from=3-1, to=2-2]
        \arrow["{I_2}"', from=3-2, to=2-2]
        \arrow["{TI_2}", from=3-4, to=2-5]
        \arrow["{I_1}"', from=3-5, to=2-5]
    \end{tikzcd}\]
    The only non-trivial part is expressing the spherical functor $STI_2$ in terms of $S_i$ and $R_i$. Writing $U:=\Cone(SR\to\id)$, we have the identification $ST\simeq US$. We have the factorization $U\simeq U_1U_2$, where $U_i:=\Cone(S_iR_i\to\id)$, which also came up in the proof of Lemma~\ref{lem:twocatsspherical}. So
    \[STI_2\simeq USI_2 \simeq U_1U_2S_2 \simeq U_1S_2T_2.\]
    This agrees with \cite{kapranovPerverseSchobers2015}*{(2.8)} as claimed. Note that they write $T_i$ for the autoequivalence on $\Psi$ that we call $U_i$.

    The proof for general $n$ follows by the same inductive scheme as the one used to define the braid group action on $\sfAut(n)$ in Theorem~\ref{thm:inftyadmissible}.
\end{proof}

By using the same construction as Definition~\ref{defn:nonabstract}, we may define the $\two$-category $\sftwoP(\bbC,R)$ of perverse schobers.

\subsection{Fourier transforms and Stokes structures}\label{subsection:fourier}
In \cite{kapranovPerverseSchobersAlgebra2020}*{\S 4.1}, a Fourier transform for perverse schobers on $(\bbC_z,R)$ is developed (assuming that $R$ satisfies a genericity property). This is given by a map on the level of objects,
\begin{align*}
    \opname{FT}:\sftwoP(\bbC_z,R) &\isomto \sftwoP^{\opname{Stokes}}(\bbC_w,0)\\
    \calF&\mapsto\check{\calF}.
\end{align*}
The target is the collection of perverse schobers equipped with $R$-Stokes data \cite{kapranovPerverseSchobersAlgebra2020}*{\S 4.1.D}, a categorification of Stokes data which appears in the irregular Riemann-Hilbert correspondence. This Stokes data is roughly a $\zeta=\opname{Arg}(w)\in S^1$-family of SODs of $\check{\calF}\resto{\bbC^\times_w}$.

We recast this Fourier transform as an equivalence of $\two$-categories in our setting. When $R=\{0\}$, there are no Stokes structures to worry about and the Fourier transform roughly interchanges $\Phi$ and $\Psi$ (See for example \cite{gammagePerverseSchobers3d2023}*{Definition 2.3}). This identifies the projection functor $\sftwoP(\bbC_z,0)\to\ol{\sftwoP}(\bbC_z,0)$ with the restriction to the complement of $0\in\bbC_w$, as expressed by the following commuting square.
\[\begin{tikzcd}
    {\sftwoP(\bbC_z,0)}\pullback & {\sftwoP(\bbC_w,0)} \\
    {\ol{\sftwoP}(\bbC_z,0)} & {\sftwoL(\bbC_w^\times)}
    \arrow["{\opname{FT}}", from=1-1, to=1-2]
    \arrow["\pi"', from=1-1, to=2-1]
    \arrow["-|_{\bbC_w^\times}"', from=1-2, to=2-2]
    \arrow["\opname{FT}"', from=2-1, to=2-2]
\end{tikzcd}\]
Thus under Fourier transform, localized schobers on $(\bbC_z,0)$ correspond to local systems on $\bbC_w^\times$.

For general $R\subset \bbC$, we may combine this with the square from Theorem~\ref{thm:merge}.
\[\begin{tikzcd}
    {\sftwoP(\bbC_z,R)}\pullback & {\sftwoP(\bbC_z,0)}\pullback & {\sftwoP(\bbC_w,0)} \\
    {\ol{\sftwoP}(\bbC_z,R)} & {\ol{\sftwoP}(\bbC_z,0)} & {\sftwoL(\bbC_w^\times)}
    \arrow["{\FS_{\sftwoP}}", from=1-1, to=1-2]
    \arrow["\pi"', from=1-1, to=2-1]
    \arrow["{\opname{FT}}", from=1-2, to=1-3]
    \arrow["\pi"', from=1-2, to=2-2]
    \arrow["-|_{\bbC_w^\times}"', from=1-3, to=2-3]
    \arrow["{\FS_{\ol{\sftwoP}}}"', from=2-1, to=2-2]
    \arrow["\opname{FT}"', from=2-2, to=2-3]
\end{tikzcd}\]
The outer pullback square gives us an equivalence,
\[\sftwoP(\bbC_z,R)\isomto \sftwoP(\bbC_w,0)\times_{\sftwoL(\bbC_w^\times)}\ol{\sftwoP}(\bbC_z,R).\]
This fiber product can be interpreted as equipping the nearby cycles part of an object of $\sftwoP(\bbC_w,0)$ with an $R$-Stokes structure. Thus under Fourier transform, we may understand a localized schobers on $(\bbC,R)$ as none other than a local system on $\bbC_w-0$ equipped with an $R$-Stokes structure.

\section{Prelocalized schobers on a disc}\label{section:prelocalized}
In this section, we study the naive categorification of $\Perv(\bbC,R)/\Loc(\bbC)$, which we denote by
\[\olsftwoP^{\pre}(\bbC,R) := \sftwoP(\bbC,R)/\sfLoc(\bbC).\]
As usual, we will make a choice of cuts and work at the level of abstract schobers.

In \S\ref{subsection:prelocalized}, we will define this quotient $\two$-category to be the right-orthogonal to the sub-$\two$-category of abstract local systems, $\sfSph^{\const}(n)\subseteq \sfSph(n)$.

In \S\ref{subsection:quotientismonad}, we show that there is an equivalence,
\[\sfSph/\sfSph^{\const}\simeq \sfSphMnd,\]
where $\sfSphMnd$ is the $\two$-category of pairs $(\Phi,M)$ consisting of a stable category $\Phi$ and a spherical monad $M$.

\begin{defn}\label{defn:sphmonad}
    A monad $M$ on category $\Phi$ is said to be a spherical monad if the monadic adjunction $S:\Phi\tofrom\Mod(M):R$ is a spherical adjunction.
\end{defn}

Spherical monads are studied in detail in \cite{christSphericalMonadicAdjunctions2023}. Observe that if $M$ is spherical, then $\Cone(\id\xto{u} M)$ is invertible. We will prove in Appendix~\ref{appendix:sphmonad}, that the converse is true, and so we have the following strengthening of \cite{christSphericalMonadicAdjunctions2023}*{Theorem 4.1}.

\begin{thm}\label{thm:christstrengthening1}
    A monad $M$ on $\Phi$ is spherical if and only if $\Cone(\id\xto{u} M)$ is invertible.
\end{thm}

For general $n$, we show that there is an equivalence
\[\sfSph(n)/\sfSph^{\const}(n)\simeq \sfSphMnd(n),\]
where $\sfSphMnd$ is the $\two$-category of pairs $(\frakS;\Phi,M)$ consisting of a stable category $\Phi$ and a spherical monad $M$ such that $\coCone(\id\xto{u} M)$ is conjugate-compatible with $\frakS$.

In \S\ref{subsection:twocatsphmnd}, we will upgrade the assignment
\[[n]\mapsto\sfSphMnd(n)\]
to a 2-Segal simplicial $\two$-category $\sfSphMnd(-)$, produce maps of simplicial $\two$-categories,
\[\sfSph(-)\to \sfSphMnd(-)\to \sfAut(-).\]
We then define a braid group action $\Br_n$ on $\sfSphMnd(n)$, compatibly with the braid group actions of $\sfSph(n)$ and $\sfAut(n)$.

\subsection{Orthogonal to local systems}\label{subsection:prelocalized}
We study the right-orthogonal to the sub-$\two$-category of local systems.  We begin by recalling which abstract perverse schobers are local systems.

\begin{defn}\label{defn:constant}
    An object $\calF\in\sfSph(n)$ is said to be an abstract local system if $\Phi_i=0$ for all $i$. For each $n$, we let $\sfSph^{\const}(n)\subseteq \sfSph(n)$ denote the full sub-$\two$-category consisting of abstract local systems.
\end{defn}

We study the quotient $\sfSph(n)/\sfSph^{\const}(n)$, which can be thought of as a naive categorification of abstract localized perverse sheaves. Since quotients of $\two$-categories are difficult, we will simply define the quotient to be the right orthogonal of $\sfSph^{\const}(n)$. We will later construct a projection functor as an adjoint to the embedding.

\begin{defn}
    An object $\calF\in\sfSph(n)$ is said to be an abstract prelocalized perverse schober with $n$ singularities if it is right orthogonal to $\sfSph^{\const}(n)$. We let $\sfSph^{\const\perp}(n)$ denote the full sub-$\two$-category of abstract prelocalized perverse schobers.
\end{defn}

\begin{prop}\label{prop:conservative}
    A perverse schober $(\Phi_i,\Psi,S_i,R_i)$ is an abstract prelocalized perverse schober if and only if the right adjoint
    \[\oplus_i\Phi_i\from\Psi:\oplus_iR_i\]
    is conservative.
\end{prop}

\begin{proof}
    We may identify objectwise, $\sfSt_k\isomto\sfSph^{\const}(n)$ via
    \[\opname{Const}:\calD\mapsto (\Phi_i=0,\Psi=\calD,S_i:0\to \Psi).\]
    Let $\calF:=(\Phi_i,\Psi,S_i:\Phi_i\to\Psi)$ be an abstract perverse schober. Observe that
    \[\Hom(\opname{Const}(\calD),\calF)\]
    is equivalent to the category of functors $G:\calD\to\Psi$ such that for each $i$, $R_i\circ G=0$.
    
    Suppose that $\calF$ is prelocalized. Let $\psi$ object of $\Psi$ such that $(\oplus_iR_i)(\psi)=0$. Then this defines an object of $\Hom(\opname{Const}(\Mod_k),\calF)$. We deduce that $\psi=0$, and hence that $(\oplus_iR_i)$ is conservative.

    Conversely, suppose that $(\oplus_iR_i)$ is conservative. Let $\calD$ be a category and let $f:\opname{Const}(\calD)\to\calF$ be a morphism. Given any object $d\in\calD$, $f(d)$ is an object of $\Psi$ satisfying
    \[(\oplus_iR_i)(f(d))=0.\]
    So by conservativity, $f(d)=0$. Hence $f=0$ and we deduce that $\calF$ is prelocalized.
\end{proof}

\begin{rmk}
    We compare this calculation to an analog with perverse sheaves. Let us take $k=\bbC$, $n=1$ for simplicity, and recall that $\Perv(\bbC,0)$ is equivalent to the category of diagrams of $k$-vector spaces,
    \[s:V\tofrom W:r\]
    such that $1-rs$ and $sr-1$ are invertible. We recall \cite{geiglePerpendicularCategoriesApplications1991} that for abelian categories, the correct notion of orthogonal subcategory asks for higher ext groups to vanish. With this in mind, an object $s:V\tofrom W:r$ is orthogonal to local systems if and only if $r$ is an isomorphism (and not just injective). This subcategory is easily seen to be equivalent to the category $\Loc(S^1)$.

    This is an instance in which the lack of a categorification of ``derived structure'' is a good thing; in the perverse schobers setting, asking for $R$ to be an equivalence (and not just conservative) would imply that $\Phi=\Psi=0$ and thus is not an interesting condition.
\end{rmk}

\begin{lem}\label{lem:activecartesian3}
    Let $\calF:=(\Phi_i,\Psi,S_i,R_i)$ be an object in $\sfSph(n)$. This is an abstract prelocalized perverse schober if and only if $\FS_{\sfSph}(\calF)$ is an abstract prelocalized perverse schober. Equivalently, there is a pullback square,
    \[\begin{tikzcd}
        {\sfSph^{\const\perp}(n)} \arrow[dr, phantom, "\scalebox{1.5}{$\lrcorner$}" , very near start, color=black] & {\sfSph^{\const\perp}} \\
        {\sfSph(n)} & {\sfSph.}
        \arrow[from=1-1, to=1-2]
        \arrow[hook, from=1-1, to=2-1]
        \arrow[hook, from=1-2, to=2-2]
        \arrow["{\FS_{\sfSph}}", from=2-1, to=2-2]
    \end{tikzcd}\]
\end{lem}

\begin{proof}
    Denote by $R$ the right adjoint appearing in $\FS_{\sfSph}(\calF)$. By Proposition~\ref{prop:conservative}, $\calF$ is a prelocalized schober if and only if $\oplus_iR_i$ is conservative, and $\FS_{\sfSph}(\calF)$ prelocalized if and only if $R$ is conservative. Observe that the following diagram commutes.
    \[\begin{tikzcd}
        {\oplus_i\Phi_i} & {\langle\Phi_1,\dots,\Phi_n\rangle} & \Psi
        \arrow["\oplus_i I_i^R", from=1-2, to=1-1]
        \arrow["{\oplus_iR_i}"', curve={height=12pt}, from=1-3, to=1-1]
        \arrow["R", from=1-3, to=1-2]
    \end{tikzcd}\]
    Here, $\oplus_i I_i^R$ is conservative by definition of SODs. It follows that $\oplus_iR_i$ is conservative if and only if $R$ is conservative, completing the proof.
\end{proof}

\subsection{Spherical monads}\label{subsection:quotientismonad}
We demonstrate a different description of $\sfSph^{\const\perp}$ in terms of spherical monads, and thereby exhibit an adjoint $\sfSph\to\sfSph^{\const\perp}$ to the inclusion of $\sfSph^{\const\perp}$.

We first ignore invertibility conditions and take $n=1$. The full sub-$\two$-category $\sfAdj^{\const\perp}\subset \sfAdj$ consists of those adjunctions where the right adjoint is conservative. These are equivalent to monadic adjunctions, so the embedding is identified with the Eilenberg--Moore construction,
\[\opname{P}^R:\sfMnd\inclto \sfAdj.\]
By definition, this is the right adjoint to the functor $P$, which outputs the monad associated to an adjunction. Note that in \S\ref{subsection:mndadj} we denoted these by $i^*\adjto i_*$.

Under this identification, $\sfSph^{\const\perp}$ is the sub-$\two$-category of those monads for which the image under $\opname{P}^R$ is spherical. Such monads are said to be spherical, as defined in Definition~\ref{defn:sphmonad}. We write $\sfSphMnd\subset \sfMnd$ for this sub-$\two$-category.

We deduce from Theorem~\ref{thm:christstrengthening1} that this restricts to an adjunction
\[\opname{P}:\sfSph\tofrom \sfSphMnd:\opname{P}^R.\]
There is also a functor $\sfSphMnd\to\sfAut$ given by $(\Phi,M)\mapsto (\Phi,\coCone(\id\to M))$.

For general $n$, we define
\[\sfSphMnd(n) := \sfAut(n)\times_{\sfAut}\sfSphMnd.\]
We denote its objects by $(\frakS;\Phi,M)$ consisting of a category $\Phi$, monad $M$, such that $\Cone(\id\to M)$ is conjugate-compatible with SOD $\frakS$, together with functors
\[\sfSph(n)\to\sfSphMnd(n)\to\sfAut(n).\]
We will show in Proposition~\ref{prop:monadisortho} that there is an identification,
\[\sfSph^{\const\perp}(n)\simeq \sfSphMnd(n).\]

\subsection{\texorpdfstring{$\sfSphMnd(n)$}{SphMnd(n)} as a simplicial \texorpdfstring{$\two$}{2}-category}\label{subsection:twocatsphmnd}
We construct the simplicial $\two$-category $\sfSphMnd(-)$. We will also construct a braid group action, and define the $\two$-category of non-abstract prelocalized schobers $\olsftwoP^{\pre}(\bbC,R)$. Since the techniques and ideas are the same as in previous sections, we will omit most details.

\begin{prop}\label{prop:simplicialsphmnd}
    The assignment $[n]\mapsto \sfSphMnd(n)$ upgrades to a simplicial $\two$-category which is 2-Segal. Moreover, the natural functors
    \[\sfSph(n)\xto{\opname{P}_n} \sfSphMnd(n)\to \sfAut(n)\]
    upgrade to maps of simplicial $\two$-categories which are Cartesian over $\Delta_{\act}$.
\end{prop}

\begin{proof}
    These follow in the same way as \S\ref{section:adjaut}. We outline the steps and omit the details.
    \begin{itemize}
        \item First consider the simplicial $\two$-category $\sfMnd'(-) := \sfEnd'(-)\times_{\sfEnd}\sfMnd$ and see that it is 2-Segal.
        \item Verify that the inclusion $\sfMnd\inclto \sfMnd'(1)$ admits a right adjoint via $(\Phi_1;\Phi,M)\mapsto (\Phi_1,I_1^RMI_1)$, by a proof similar to Lemma~\ref{lem:keyrightadjoint}, Lemma~\ref{lem:rightadjointend} and Lemma~\ref{lem:rightadjointbdj}.
        \item Verify that the inclusion $\sfMnd(n):= \sfEnd(n)\times_{\sfEnd}\sfMnd\inclto \sfMnd'(n)$ admits a right adjoint, by reducing to the case $n=1$.
        \item Use Lemma~\ref{lem:descendingsimplicialobjects} to upgrade $\sfMnd(-)$ to a simplicial $\two$-category. Verify that this is 2-Segal.
        \item See that the projection $\opname{P}:\sfAdj\to \sfMnd$ and the natural map $\sfMnd\to \sfEnd$ induce a map of simplicial $\two$-categories, $\sfBdj'(-)\to\sfMnd'(-)\to \sfEnd'(-)$.
        \item Use Lemma~\ref{lem:descendingsimplicialmaps} to obtain maps of simplicial objects $\sfAdj(-)\simeq\sfBdj(-)\to\sfMnd(-)$ and $\sfMnd(-)\to\sfEnd(-)$. Then verify that these are Cartesian over each active arrow $[1]\to [n]$ and hence over $\Delta_{\act}$.
        \item Finally, verify that these properties and structures all restrict to respective subcategories as given in the statement.
    \end{itemize}
\end{proof}

\begin{prop}\label{prop:monadisortho}
    For each $n$, the functor $\opname{P}_n$ admits a right adjoint $\opname{P}_n^R$ which is 2-fully faithful. Moreover, the resulting full sub-$\two$-category is $\sfSph^{\const\perp}(n)$.
\end{prop}

\begin{proof}
    Let $n=1$. We recall from the discussion at the beginning of this subsection that $\opname{P}:\sfSph\to\sfSphMnd$ admits a 2-fully faithful right adjoint $\opname{P}^R$ given by the inclusion of $\sfSph^{\const\perp}$.

    By applying $\sfAut(n)\times_{\sfAut}-$, we learn that $\opname{P}_n$ admits a 2-fully faithful right adjoint $\opname{P}_n^R$. Since $\sfSph(-)\to \sfSphMnd(-)$ is Cartesian over the active arrow $[1]\to [n]$, we have a pullback square as follows.
    \[\begin{tikzcd}[ampersand replacement=\&]
        {\sfSphMnd(n)}\pullback \& \sfSphMnd \\
        {\sfSph(n)} \& \sfSph
        \arrow[from=1-1, to=1-2]
        \arrow["{\opname{P}_n^R}"', hook, from=1-1, to=2-1]
        \arrow["{\opname{P}^R}"', hook, from=1-2, to=2-2]
        \arrow[from=2-1, to=2-2]
    \end{tikzcd}\]
    So we are done by comparing this to the diagram of Lemma~\ref{lem:activecartesian3}.
\end{proof}

From here on, we write $\sfSphMnd(n)$ instead of $\sfSph^{\const\perp}(n)$. As such, we refer to objects in $\sfSphMnd(n)$ as abstract prelocalized schobers and view $\opname{P}_n$ as the quotient functor. In Appendix~\ref{appendix:sphmonad}, we study the objects of $\sfSphMnd(n)$ in more detail.

\begin{cor}\label{cor:action3}
    There is a natural action of the braid group $\Br_n$ on $\sfSphMnd(n)$. Moreover, the two functors in Proposition~\ref{prop:simplicialsphmnd} are naturally $\Br_n$-equivariant.
\end{cor}

\begin{proof}
    Observe that the sub-$\two$-category $\sfSph^{\const}(n)$ is stable under the action of $\Br_n$ on $\sfSph(n)$. So we obtain an action of $\Br_n$ on $\sfSphMnd(n)$ which respects $\opname{P}_n^R$ and hence $\opname{P}_n$. It is easy to verify that this action also naturally respects the functor to $\sfAut(n)$. For example, we can factor this functor as $\sfSphMnd(n)\xto{\opname{P}_n^R}\sfSph(n)\to \sfAut(n)$.
\end{proof}

By the same construction as Definition~\ref{defn:nonabstract}, we obtain the $\two$-categories $\olsftwoP^{\pre}(\bbC,R)$ of prelocalized schobers, together with functors
\[\sftwoP(\bbC,R)\to \olsftwoP^{\pre}(\bbC,R)\to \olsftwoP(\bbC,R).\]

\section{From prelocalized schobers to localized schobers}\label{section:pretolocalized}
In this section, we describe how we may recover localized schobers from prelocalized ones by stabilizing an operation on $\olsftwoP^{\pre}(\bbC,R)$ called symplectic suspension. Together with our study of prelocalized schobers in \S\ref{section:prelocalized}, this gives a precise sense in which we may define localized schobers from perverse schobers, as depicted below.
\begin{align*}
    \text{Schobers on }(\bbC,R)&\overset{\text{quotient}}{\rightsquigarrow}\text{Prelocalized schobers on }(\bbC,R)\\
    &\overset{\text{stabilization}}{\rightsquigarrow}\text{Localized schobers on }(\bbC,R)
\end{align*}
The need to stabilize the symplectic suspension is motivated by the following issue which arises in symplectic geometry.

\begin{obs}\label{obs:stable}
    Let $f:X\to\bbC$ be a Lefschetz (or symplectic) fibration with critical values $R$. Let $\calF(X,f)\in\sfP(\bbC,R)$ be the associated perverse schober defined in \cite{kapranovPerverseSchobersAlgebra2020}*{\S 3.6} and let $\ol{\calF(X,f)}^{\opname{pre}}\in\olsftwoP^{\opname{pre}}(\bbC,R)$ be the resulting prelocalized perverse schober.

    The operation $\overline{\calF(X,f)}^{\opname{pre}}\mapsto \overline{\calF(X\times \bbC_x,f+x^2)}^{\opname{pre}}$ of prelocalized schobers is not ``invertible''.
\end{obs}

\begin{eg}
    Consider for example $X=\bbC\xto{z^2}\bbC$. Then, $\calF(\bbC,z^2)$ is equivalent to the spherical adjunction
    \[\Mod_k\tofrom\Loc(S^0)\simeq\Mod_k\oplus\Mod_k.\]
    Thus the prelocalized schober is given by $\Mod_k$ together with the monad/algebra $C^\bullet(S^0;k)$ of cochains on $S^0$. The algebra structure is that of the cup product.

    After symplectic suspension, $\calF(\bbC^2,x^2+z^2)$ is equivalent to the spherical adjunction
    \[\Mod_k\tofrom\Loc(S^1).\]
    Thus the prelocalized schober is $\Mod_k$ with monad/algebra $C^\bullet(S^1;k)$.

    The algebra $C^\bullet(S^1;k)$ contains ``less information'' since this is isomorphic to a square-zero extension of $k$ while $C^\bullet(S^0;k)$ is not.
\end{eg}

In \S\ref{subsection:suspension}, we formulate the symplectic suspension operation algebraically, working at the level of individual abstract perverse schobers. We expect that under this operation, $\calF(X,f)\mapsto \calF(X\times \bbC_x,f+x^2)$, although we don't verify this outside examples such as the one above. We then define the symplectic suspension similarly for abstract prelocalized and localized perverse schobers.

We then upgrade these to functors of $\two$-categories and exhibit compatibility with braid group actions so that we obtain symplectic suspension functors on (prelocalized) perverse schobers,
\begin{align*}
    \calSus_{\sftwoP}:&\sftwoP(\bbC,R)\to \sftwoP(\bbC,R),\\
    \calSus_{\olsftwoP}:&\olsftwoP^{\pre}(\bbC,R)\to \olsftwoP^{\pre}(\bbC,R).
\end{align*}

We prove in \S\ref{subsection:stabilization} the main theorem of this section, which confirms that localized schobers can indeed be recovered from prelocalized schobers by stabilizing symplectic suspension.

\begin{thm}\label{thm:limit0}
    There is an equivalence of $\two$-categories,
    \[\olsftwoP(\bbC,R)\isomto \lim(\cdots\to\olsftwoP^{\pre}(\bbC,R)\xto{\calSus}\olsftwoP^{\pre}(\bbC,R)\xto{\calSus}\olsftwoP^{\pre}(\bbC,R)).\]
\end{thm}

\begin{proof}
    This follows from Theorem~\ref{thm:limit1}.
\end{proof}

\subsection{Symplectic suspension}\label{subsection:suspension}
We will begin by defining the symplectic suspension operation on $\sfAdj$ and show that it induces operations on $\sfSph$, $\sfMnd$ and $\sfSphMnd$. We will later upgrade these to functors of $\two$-categories.

On $\sfAdj$, consider the following operation defined on objects.
\begin{align*}
    \calSus_{\sfAdj}: \sfAdj&\to \sfAdj\\
    [S:\Phi\tofrom\Psi:R]&\mapsto [S':\Phi\tofrom\Psi':=\lim(\Phi\oplus\Phi\to\Psi):R'].
\end{align*}
Here, the limit is understood to be strict. So an object is given by a triple $(\phi_1,\phi_2,\alpha)$, where $\alpha:S\phi_1\isomto S\phi_2$ is an isomorphism. In this notation, the adjoint functors are given by
\begin{align*}
    S':\phi&\mapsto (\phi,\phi,1),\\
    R':(\phi_1,\phi_2,\alpha) &\mapsto \coCone(\phi_1\oplus\phi_2\to RS\phi_1).
\end{align*}
We calculate in Lemma~\ref{lem:sussph} that there is an identification $\coCone(\id\to R'S')\simeq \coCone(\id\to RS)[-1]$.

\begin{eg}\label{eg:localsystem}
    Let $X$ be a space and consider the adjunction
    \[\calA(X):=[\pi^*:\Mod_k\tofrom\Loc(X):\pi_*],\]
    where $\Loc(X)$ denotes the category of $\infty$-local systems. Then,
    \[\calSus_{\sfAdj}(\calA(X))\simeq\calA(\Sigma X).\]
    In this way, suspension of spaces is compatible with suspension of adjoint pairs. 
\end{eg}

\begin{lem}\label{lem:sussph}
    Suppose that an adjunction $\calA$ is spherical. Then its suspension $\calSus_{\sfAdj}(\calA)$ is also spherical. Thus the operation descends to $\calSus_{\sfSph}:\sfSph\to\sfSph$.
\end{lem}

\begin{proof}
    We may use the formulae above to compute the twist functors explicitly. We have
    \[T'_{\Phi}:\phi\mapsto \Tot(\phi\xto{(1,-1)}\phi\oplus\phi\to RS\phi)\simeq T_{\Phi}[-1]\phi,\]
    and
    \[T'_{\Psi'}:(\phi_1,\phi_2,\alpha)\mapsto (T_{\Phi}[1]\phi_2,T_{\Phi}[1]\phi_1,T_{\Psi}\alpha^{-1}).\]
    These are clearly autoequivalences.
\end{proof}

On $\sfMnd$, consider the following operation defined on objects.
\begin{align*}
    \calSus_{\sfMnd}: \sfMnd&\to \sfMnd\\
    (\Phi,M)&\mapsto (\Phi,M':=\lim(\id\oplus \id\to M))
\end{align*}
Here, $\id$ is the trivial monad on $\Phi$, $\id\to M$ is the canonical map of monads and the limit is taken in the category of monads on $\Phi$.

\begin{lem}\label{lem:sussphmnd}
    Suppose that a monad $M$ on $\Phi$ is spherical. Then its suspension $\calSus_{\sfMnd}(M)$ is also spherical. Thus the operation descends to $\calSus_{\sfSphMnd}:\sfSphMnd\to \sfSphMnd$.
\end{lem}

\begin{proof}
    This follows by observing that $\coCone(\id\to M')\simeq \coCone(\id\to M)[-1]$, and applying Theorem~\ref{thm:christstrengthening1}.
\end{proof}

Finally, on $\sfEnd$ and $\sfAut$, we have the operations $\calSus_{\sfEnd},\calSus_{\sfAut}$, each given by the formula $(\Phi,T)\mapsto (\Phi,T[-1])$.

For general $n$, we define
\[\calSus_{\sfAut}:\sfAut(n)\to \sfAut(n)\]
by $(\frakS;\Phi,T)\mapsto (\frakS;\Phi,T[-1])$. We can also define
\begin{align*}
    \calSus_{\sfSph}:&\sfSph(n)\to \sfSph(n),\\
    \calSus_{\sfSphMnd}:&\sfSphMnd(n)\to \sfSphMnd(n)
\end{align*}
by applying $\sfAut(n)\times_{\sfAut}-$ to $\calSus_{\sfSph},\calSus_{\sfSphMnd}$ respectively.

\begin{rmk}[Thom--Sebastiani for schobers]
    We may generalize this construction to a binary operation,
    \begin{align*}
        -\ast-: \sfSph\times \sfSph&\to \sfSph\\
        [S_1:\Phi_1\tofrom\Psi_1:R_1],[S_2:\Phi_2\tofrom\Psi_2:R_2]&\mapsto [S_{12}:\Phi_{12}\tofrom\Psi_{12}:S_{12}^R],
    \end{align*}
    so that $\calSus_{\sfSph} = \calF(\bbC_z,z^2)\ast-$. In future work, we study this binary operation. We expect that on Lefschetz schobers (with critical value $R=\{0\}$), there is an identification,
    \[\calF(X,f)\ast\calF(Y,g)\simeq \calF(X\times Y,f\boxplus g),\]
    which can be thought of as a categorification of the Thom--Sebastiani formula. Moreover, we expect that under the 3d mirror symmetry of Gammage--Hilburn--Mazel-Gee \cite{gammagePerverseSchobers3d2023}, this operation corresponds to convolution monoidal product $\opname{2IndCoh}_{\mathbb{L}}(\bbA^1/\bbG_m)$ coming from the product monoidal structure on $\bbA^1/\bbG_m$.
\end{rmk}

We now upgrade the symplectic suspension to endofunctors of $\two$-categories. To do this, it is natural to work in certain enlarged $\two$-categories, where we allow for more 1-morphisms given by oplax natural transformations of diagrams. This is already apparent from our definitions at the level of objects. For example, to define the symplectic suspension on monads, we wrote down a diagram $\id\to M\from \id$ of monads on $\Phi$. The arrow $(\Phi,\id)\to (\Phi,M)$ doesn't appear in $\sfMnd$, as this is given by an oplax natural transformation of presheaves on $\frakmnd$.

We recall from \cite{haugsengLaxTransformationsAdjunctions2021} that $\sfAdj$ is equivalent to the $\two$-category $\sfFun(\frakadj,\sfSt_k)$ of diagrams on the walking adjunction $\frakadj$. Consider the functors of $\two$-categories,
\[\frakadj\from \frakmnd\from \frakend\from \mathfrak{pt},\]
where $\mathfrak{pt}$ denotes the trivial $\two$-category. We denote the maps from the trivial $\two$-category by $\mathfrak{pt}\xto{a} \frakadj,\mathfrak{pt}\xto{m}\frakmnd,\mathfrak{pt}\xto{e}\frakend$. It follows by Theorem~\ref{thm:laxfibrations} that there are left adjoints to $a^*,m^*,e^*$, denoted $a_!,m_!,e_!$. These send $\Phi\in\sfSt_k$ to the objects
\[[\Phi\tofrom \Phi]\in\sfAdj,\; (\Phi,\id)\in\sfMnd,\; (\Phi,\id)\in\sfEnd\]
respectively. It follows that the units of these adjunctions, $\Id\to a^*a_!, m^*m_!, e^*e_!$ are each equivalences, and moreover that the resulting lift of $a_!\Phi$ to $\sfAdj_{\oplax}\resto{\Phi}$ is an initial object. Similar statements hold for $\sfMnd_{\oplax}$ and $\sfEnd_{\oplax}$.

We define functors
\begin{align*}
    \calSus_{\sfAdj_{\oplax}}&:=a_!a^*\times_{\id}a_!a^*:\sfAdj_{\oplax}\to\sfAdj_{\oplax}\\
    \calSus_{\sfMnd_{\oplax}}&:=m_!m^*\times_{\id}m_!m^*:\sfMnd_{\oplax}\to\sfMnd_{\oplax}\\
    \calSus_{\sfEnd_{\oplax}}&:=e_!e^*\times_{\id}e_!e^*:\sfEnd_{\oplax}\to\sfEnd_{\oplax}.
\end{align*}

These are compatible with the pullback functors
\[\sfAdj_{\oplax}\to \sfMnd_{\oplax}\to \sfEnd_{\oplax},\]
because pullback, being a right adjoint, commutes with limits.

These functors restrict to the locally full subcategories $\sfAdj\subset\sfAdj_{\oplax}$, $\sfMnd$ and $\sfEnd$ respectively. By construction, these agree on the level of objects with the mappings described in the previous subsection. It follows that symplectic suspension restricts to endofunctors of $\sfSph,\sfSphMnd,\sfAut$ respectively.

For general $n$, define $\calSus_{\sfAut}$ on $\sfAut(n)$ by recalling that it is a full sub-$\two$-category,
\[\sfAut(n)\subset\sfSOD(n)\times_{\sfSt_k}\sfAut,\]
and seeing that $\Id\times\calSus_{\sfAut}$ preserves this sub-$\two$-category. We can define the symplectic suspension functor on $\sfSph(n)$ by recalling from Proposition~\ref{prop:activecartesian2} that
\[\sfSph(n)\simeq \sfAut(n)\times_{\sfAut}\sfSph,\]
and applying symplectic suspension to each of these three terms separately. We can do similarly for $\sfSphMnd(n)$. It easily follows that symplectic suspension naturally commutes with the braid group action $\Br_n$, and so we obtain symplectic suspension operations
\[\calSus_{\sftwoP},\calSus_{\olsftwoP^{\pre}},\calSus_{\olsftwoP}\]
on the respective (non-abstract) perverse schober categories.

\subsection{Stabilization}\label{subsection:stabilization}
We may now translate Observation~\ref{obs:stable} into algebraic terms on the level of objects.

\begin{obs}\label{obs:key}
    When $n\geq1$, the functor $\calSus_{\sfSphMnd}:\sfSphMnd(n)\to \sfSphMnd(n)$ is not an equivalence.
\end{obs}

For geometric reasons explained in \S\ref{subsection:radonissuspension}, this is problematic. The following example demonstrates that it is not an equivalence on $\sfSphMnd$.

\begin{eg}\label{eg:suspendtwice}
    Assume that $k$ is a discrete commutative ring. Monads on $\Mod_k$ are given by $k$-algebras. We claim that if $A$ is a $k$-algebra and $B=k\times_Ak$ is its suspension, then the unit map $k\to B$, viewed as a morphism in $\Mod_k$, admits a splitting $B\to k$. Indeed, we can simply choose one of the two projections to $k$.

    So to see that $\calSus$ is not an equivalence on $\sfSphMnd$, it suffices to exhibit any $k$-algebra $B$ which doesn't admit such a splitting (and gives a spherical monad). We can simply take $B=0$.
\end{eg}

We fix this by inverting $\calSus$ on $\sfSphMnd$ and more generally on $\sfSphMnd(n)$ in a universal manner. The following theorem says that this procedure recovers the $\two$-category of abstract \textit{localized} perverse schobers.

\begin{thm}\label{thm:limit1}
    There is an equivalence of $\two$-categories,
    \[\sfAut(n)\isomto \lim(\cdots\to\sfSphMnd(n)\xto{\calSus}\sfSphMnd(n)\xto{\calSus}\sfSphMnd(n)).\]
\end{thm}

\begin{rmk}\label{rmk:splitting}
    We saw in Example~\ref{eg:suspendtwice} that any $k$-algebra $B$ which doesn't split as a $k$-module doesn't have a preimage along $\calSus$. We may interpret Theorem~\ref{thm:limit1} as saying that the $k$-algebras which admit all repeated preimages along $\calSus$ are those that split as a square-zero extension, $B=k\oplus T$, where $T$ is a $k$-module for which $T\otimes_k-$ is invertible.

    The construction appearing in \cite{segalAllAutoequivalencesAre2020} uses that from any autoequivalence $T$, we may produce a square-zero monad $\id\oplus T[-1]$. We expect that this is exactly what $\sfAut\to \sfSphMnd$ does.
\end{rmk}

\begin{prop}\label{prop:limit2}
    Consider the following diagram in which each row is a limit diagram, and where we have omitted the subscript of $\calSus$ for readability.
    \[\begin{tikzcd}[ampersand replacement=\&]
        {\sfMnd[\calSus^{-1}]} \& \cdots \& \sfMnd \& \sfMnd \& \sfMnd \\
        \sfEnd \& \cdots \& \sfEnd \& \sfEnd \& \sfEnd
        \arrow[from=1-1, to=1-2]
        \arrow["\opname{C}"', from=1-1, to=2-1]
        \arrow[from=1-2, to=1-3]
        \arrow["\calSus", from=1-3, to=1-4]
        \arrow[from=1-3, to=2-3]
        \arrow["\calSus", from=1-4, to=1-5]
        \arrow[from=1-4, to=2-4]
        \arrow[from=1-5, to=2-5]
        \arrow[from=2-1, to=2-2]
        \arrow[from=2-2, to=2-3]
        \arrow["\calSus"', from=2-3, to=2-4]
        \arrow["\calSus"', from=2-4, to=2-5]
    \end{tikzcd}\]
    The resulting functor $\opname{C}:\sfMnd[\calSus^{-1}]\to \sfEnd$ is an equivalence of $\two$-categories. 
\end{prop}

\begin{proof}
    By Yoneda lemma, it suffices to show that for any test object $\frakt\in\wCat_{\two}$, the following map of mapping spaces is an equivalence,
    \begin{equation}\label{equation:wantequiv}
        \Map_{\wCat_{\two}}(\frakt,\sfMnd[\calSus^{-1}])\to \Map_{\wCat_{\two}}(\frakt,\sfEnd).
    \end{equation}
    In fact, this lives over the $\zero$-category $\Map(\frakt,\sfSt_k)$. It suffices to check that for each point $\opname{X}\in \Map(\frakt,\sfSt_k)$, the map (\ref{equation:wantequiv}) restricts to an equivalence.

    We apply $\Map_{\wCat_{\two}}(\frakt,-)$ and then the restriction to the fiber over $\opname{X}$ to the whole diagram, noting that these operations preserve limits. By Lemma~\ref{lem:haugsengvariant}, we have a functorial identification,
    \[\Map(\frakt,\sfMnd)\resto{\opname{X}}\simeq \Map(\frakmnd,\sfFun(\frakt,\sfSt_k))\resto{\opname{X}}\simeq \Alg(\sfT(\opname{X},\opname{X}))^{\simeq},\]
    where $\sfT := \sfFun(\frakt,\sfSt_k)$. Similarly,
    \[\Map(\frakt,\sfEnd)\resto{\opname{X}}\simeq \Map(\frakmnd,\sfFun(\frakt,\sfSt_k))\resto{\opname{X}}\simeq \sfT(\opname{X},\opname{X})^{\simeq}.\]
    This results in the following diagram.
    \[\begin{tikzcd}[ampersand replacement=\&]
        {\lim(\Alg(\sfT(\opname{X},\opname{X}))^{\simeq})} \& \cdots \& {\Alg(\sfT(\opname{X},\opname{X}))^{\simeq}} \& {\Alg(\sfT(\opname{X},\opname{X}))^{\simeq}} \\
        {\sfT(\opname{X},\opname{X})^{\simeq}} \& \cdots \& {\sfT(\opname{X},\opname{X})^{\simeq}} \& {\sfT(\opname{X},\opname{X})^{\simeq}}
        \arrow[from=1-1, to=1-2]
        \arrow[from=1-1, to=2-1]
        \arrow[from=1-2, to=1-3]
        \arrow["\calSus", from=1-3, to=1-4]
        \arrow[from=1-3, to=2-3]
        \arrow[from=1-4, to=2-4]
        \arrow[from=2-1, to=2-2]
        \arrow[from=2-2, to=2-3]
        \arrow["{[-1]}"', from=2-3, to=2-4]
    \end{tikzcd}\]
    The left vertical arrow is an equivalence by applying the limit-preserving functor $(-)^{\simeq}$ to Lemma~\ref{lem:lurievariant} and identifying the map with the equivalence
    \[\Sp(\Alg(\sfT(\opname{X},\opname{X})))\simeq \sfT(\opname{X},\opname{X}),\]
    of \cite{lurieHigherAlgebra2017}*{7.3.4.7}.
\end{proof}

\begin{lem}\label{lem:haugsengvariant}
    Let $\sfT\in\wCat_{\two}$ and $\opname{X}\in\sfT$. Then, the fiber over $\opname{X}$ of the map of $\zero$-categories, $\Map(\frakmnd,\sfT)\to \sfT^{\simeq}$, is the $\zero$-category $\Alg(\sfT(\opname{X},\opname{X}))^{\simeq}$.  An analogous statement holds with $\frakmnd$ replaced by $\frakend$.
\end{lem}

\begin{proof}
    It follows from \cite{haugsengLaxTransformationsAdjunctions2021}*{Theorem 1.4} that the fiber of
    \[\sfFun(\frakmnd,\sfT)_{\oplax}^{\one}\to \sfT^{\one}\]
    over $\opname{X}$ is identified with $\Alg(\sfT(\opname{X},\opname{X}))$. The inclusion
    \[\sfFun(\frakmnd,\sfT)^{\one}\to \sfFun(\frakmnd,\sfT)_{\oplax}^{\one}\]
    induces an equivalence of $\zero$-categories after applying $(-)^{\simeq}$. This is because any oplax natural transformation of diagrams that yields an invertible 1-morphism in $\sfFun(\frakmnd,\sfT)_{\oplax}^{\one}$ is in fact an honest natural transformation of diagrams. This completes the proof. The statement for $\frakend$ is similar.
\end{proof}

\begin{lem}\label{lem:lurievariant}
    Let $\calC$ be a large, cocomplete $\one$-category with an initial object denoted $k\in\calC$. Let $\Omega_{\calC}$ be the operation, $c\mapsto k\times_c k$. Then the following limit exists, and there is an equivalence of categories,
    \[\opname{Sp}(\calC_{/k})\isomto\lim(\cdots\to \calC\xto{\Omega_{\calC}}\calC\xto{\Omega_{\calC}}\calC).\]
    Here, $\opname{Sp}(\calC_{/k})$ is the $\one$-category of spectrum objects of the pointed category $\calC_{/k}$ as defined in \cite{lurieHigherAlgebra2017}.
\end{lem}

\begin{rmk}
    This notation is confusing as our symplectic \textit{suspension} corresponds to a ``\textit{loop space} functor'' here. We also see this in the negative shift $[-1]$ in $\calSus_{\sfAut}$. This comes down to the covariance of the functor
    \[C^*(-;k):\opname{Spaces}\to \Alg_k.\]
\end{rmk}

\begin{proof}[Proof of Lemma~\ref{lem:lurievariant}]
    We observe that $\Omega_{\calC}$ factors through the slice category $\calC_{/k}$ once we choose one of the two projections, $\Omega_{\calC}c\to k$.
    \[\begin{tikzcd}
        \calC & \calC \\
        & {\calC_{/k}}
        \arrow["{\Omega_{\calC}}", from=1-1, to=1-2]
        \arrow[from=1-1, to=2-2]
        \arrow["{\opname{forget}}"', from=2-2, to=1-2]
    \end{tikzcd}\]
    We may extend this to the following naturally commuting diagram, where $\Omega_{\calC_{/k}}$ denotes the operation given by the same formula on the pointed category $\calC_{/k}$.
    \[\begin{tikzcd}[ampersand replacement=\&]
        \cdots \& \calC \& \calC \& \calC \& \calC \\
        \& \cdots \& {\calC_{/k}} \& {\calC_{/k}} \& {\calC_{/k}}
        \arrow[from=1-1, to=1-2]
        \arrow["{\Omega_{\calC}}", from=1-2, to=1-3]
        \arrow[from=1-2, to=2-3]
        \arrow["{\Omega_{\calC}}", from=1-3, to=1-4]
        \arrow[from=1-3, to=2-4]
        \arrow["{\Omega_{\calC}}", from=1-4, to=1-5]
        \arrow[from=1-4, to=2-5]
        \arrow[from=2-2, to=2-3]
        \arrow[from=2-3, to=1-3]
        \arrow["{\Omega_{\calC_{/k}}}"', from=2-3, to=2-4]
        \arrow[from=2-4, to=1-4]
        \arrow["{\Omega_{\calC_{/k}}}"', from=2-4, to=2-5]
        \arrow[from=2-5, to=1-5]
    \end{tikzcd}\]
    The top and bottom rows are coinitial in each other. Thus the existence of the limit, and the value of the limit of these rows are identified. The limit of the bottom row exists, and is the definition of the stabilization $\opname{Sp}(\calC_{/k})$.
\end{proof}

\begin{proof}[Proof of Theorem~\ref{thm:limit1}]
    We first prove the statement when $n=1$. We can pullback the diagram of Proposition~\ref{prop:limit2} along $\sfAut\inclto \sfEnd$. Using Theorem~\ref{thm:christstrengthening1}, the top row of the resulting diagram, which is still a limit, is identified with
    \[\sfAut\to\cdots\to\sfSphMnd\xto{\calSus}\sfSphMnd\xto{\calSus}\sfSphMnd,\]
    exhibiting the desired limit. The result for general $n$ follows by pulling back along $\sfAut(n)\to\sfAut$.
\end{proof}

\section{Schobers on a non-characteristic family of discs}\label{section:families}
In this section, we will define schobers and localized schobers on non-characteristic families of discs over a base $B$. We will work complex analytically.

Let $B\subseteq \bbC^{N-1}$ an open subset with respect to the Euclidean topology, we would like to consider perverse schobers on $\bbC\times B$. It is important that $B$ is a subset of $\bbC^{N-1}$ for the following reason. Recall that given a locally constant perverse sheaf on $B$, the natural operation at $b\in B$ which outputs a vector space is not the stalk, but a microstalk, which requires an additional choice of a quadratic form $q_b$ on $T_bB$. Different choices of $q_b$ are related by a Maslov class. For perverse sheaves, we can avoid this choice by taking the ordinary stalk and shifting by $\dim(B)$.

This shifted stalk operation is not available to us in the categorified setting, so we are forced to choose quadratic forms on tangent spaces of $B$. We believe that the work of Nadler--Shende \cite{nadlerSheafQuantizationWeinstein2022}*{\S 10.4} provides a suitable categorification of the Maslov class, organizing how different choices of $q_b$ are related.

We will take the following simpler but less canonical approach. We fix once and for all, compatible choices of quadratic form at each point $b\in B$, which is possible since $\bbC^{N-1}$ admits such a choice. This lets us ignore subtleties involving the Maslov class for the rest of the paper.

\begin{defn}\label{defn:nonchar}
    A non-characteristic family of discs over an open subset $B\subseteq \bbC^{N-1}$ is a pair $(\bbC\times B,H)$ where $H\subset \bbC\times B$ is a hypersurface such that
    \begin{enumerate}
        \item $H\inclto \bbC\times B\xto{\pi} B$ is finite,
        \item there exists a complement of divisors $B^{\opname{sm}}\subseteq B$ such that $H^{\sm}:=H\cap\pi^{-1}(B^{\opname{sm}})\to B^{\opname{sm}}$ is \'etale, and
        \item the closure $\Lambda_H$ of $T^*_{H^{\sm}}$ is non-characteristic for the projection $\pi:\bbC\times B\to B$.
    \end{enumerate}
    Moreover, if we can take $B^{\sm}=B$, then we say that $(\bbC\times B,H)$ is a locally constant family of discs.
\end{defn}

\begin{eg}
    Let $B=\bbC_x$ and denote by $y$ the coordinate of the fibers so that a family of discs is described by some curve $C\subset \bbC^2_{x,y}$. Then, the non-charactericity condition is equivalent to asking that no tangent lines to $C$, or limits thereof, are vertical. For example, $C=\{y^2=x^3\}$ defines a non-characteristic family of discs. A non-example is $\{y^2=x\}$ which has a vertical tangent at $(0,0)$.
\end{eg}

\begin{eg}
    For any $B$, we have the locally constant family of discs, $(\bbC\times B,0\times B)$.
\end{eg}

We informally state the main definitions of this section.

\begin{defn}\label{defn:family1}
    Let $(\bbC\times B,H)$ be a locally constant family of discs. A perverse schober $\calF$ on $(\bbC\times B,H)$ is the data of a locally constant choice of object $\calF\resto{b}$ for each $b\in B$. Denote the collection of such objects by $\sftwoP(\bbC\times B,H)$.
\end{defn}

By applying $\FS_{\sftwoP}$ on each fiber, we obtain a mapping,
\[\FS_{\sftwoP}:\sftwoP(\bbC\times B,H)\to \sftwoP(\bbC\times B,0\times B).\]

\begin{defn}\label{defn:family2}
    Let $(\bbC\times B,H)$ be a non-characteristic family of discs and choose $B^{\opname{sm}}\subseteq B$ as in Definition~\ref{defn:nonchar}. A perverse schober on $(\bbC\times B,H)$ consists of the following data.
    \begin{enumerate}
        \item A perverse schober $\calF\resto{B^{\opname{sm}}}\in \sftwoP(\bbC\times B^{\opname{sm}},H^{\sm})$.
        \item An extension of
        \[\FS_{\sftwoP}(\calF\resto{B^{\opname{sm}}})\in \sftwoP(\bbC\times B^{\opname{sm}},0\times B^{\opname{sm}})\]
        to an object
        \[\FS_{\sftwoP}(\calF)\in \sftwoP(\bbC\times B,0\times B).\]
    \end{enumerate}
\end{defn}

\begin{rmk}\label{rmk:comparison2}
    Definition~\ref{defn:family2} is motivated by, and can be considered a categorification of \cite{gelfandMicrolocalPerverseSheaves2005}*{Proposition 3.3}. This statement is a consequence of an existing definition of perverse sheaves, while for us it is the definition.
\end{rmk}

In \S\ref{subsection:universal} and \S\ref{subsection:schobersonfamily}, we will upgrade $\sftwoP(\bbC\times B,H)$ to a $\two$-category, where $(\bbC\times B,H)$ is a locally constant or non-characteristic family of discs. A key input to our theory is the following consequence of Theorem~\ref{thm:merge}.

\begin{prop}\label{prop:FSisfaithful}
    The functor $\FS_{\sftwoP}:\sftwoP(\bbC,R)\to \sftwoP(\bbC,0)$ is 2-faithful.
\end{prop}

Using this, we will show that Definition~\ref{defn:family2} is independent of the choice of $B^{\sm}$, and in particular agrees with Definition~\ref{defn:family1} for locally constant families. We will also show that whenever $B$ is contractible and $b\in B^{\sm}$, then the restriction functor
\[\sftwoP(\bbC\times B,H)\to \sftwoP(\bbC,H_b)\]
is 2-fully faithful, and identify the image in terms of certain periodicity properties. This is closely related to the periodic SODs as studied in \cite{dyckerhoffNsphericalFunctorsCategorification2023}. We use this to give explicit descriptions of schobers on locally constant families of discs in several examples.

The statements and definitions of \S\ref{subsection:schobersonfamily} admit analogues for both localized and prelocalized schobers.

\begin{rmk}\label{rmk:comparison}
    It is instructive to consider the real analog, where Definition~\ref{defn:family2} fails to define the correct category. Namely, let $L$ be a singular support condition on $\bbR^2_{x,y}$ which is non-characteristic over $y$, and locally constant on $\bbR_x^{\opname{sm}}$. There is a functor,
    \[\sh(\bbR^2_{x,y},L)\to \sh(\bbR_{y}\times \bbR_x^{\opname{sm}},L\resto{\bbR_x^{\opname{sm}}})\times_{\sh(\bbR_{y}\times \bbR_x^{\opname{sm}},0\times \bbR_x^{\opname{sm}})}\sh(\bbR_{y}\times \bbR_x,0\times \bbR_x).\]
    This fails to be an equivalence. For example, compare the following two singular supports. Let $L_1$ be the union of the conormal of the lines $y=x$ and $y=-x$. Let $L_2$ be the union of the conormals to $y=0$ and $y=x^2$. These singular supports define non-equivalent categories, but are indistinguishable on the right hand side of the above expression.

    This comes down to the failure of Proposition~\ref{prop:FSisfaithful} in the real setting. Concretely, the functor which maps a filtered complex of vector spaces to its total space is not faithful.

    Said more conceptually, in the real setting, Morse groups in different Morse indices can admit interesting connecting maps, while in the complex setting, all ``Morse categories'' live in middle index and don't admit such maps between them.
\end{rmk}

In \S\ref{subsection:y2x3}, we study in detail schobers supported on the curve $y^2=x^3$. We show that our definition decategorifies correctly to perverse sheaves as studied by MacPherson--Vilonen \cite{macphersonPerverseSheavesSingularities1988}. We also show that our definition agrees with a microlocalization of $\bbA_2$-schobers as defined by Dyckerhoff--Wedrich \cite{dyckerhoffPerverseSchobersCoxeter2025}, and deduce a relation to $A_2$-configurations of spherical objects of Seidel--Thomas \cite{seidelBraidGroupActions2000}.

\subsection{Schobers on the universal family of discs}\label{subsection:universal}
In this section, we will define a local system of $\two$-categories over the configuration space of $n$ distinct points on $\bbC$, which will be important in what follows.

Let $\Conf(n)$ denote the configuration space parametrizing $n$ distinct points on $\bbC$. Then, the assignment
\[R\in\Conf(n)\mapsto\sftwoP(\bbC,R)\]
upgrades to a local system of $\two$-categories on $\Conf(n)$, which we denote $\ul{\sftwoP}(\bbC,-)$. Moreover, $\FS_{\sftwoP}:\sftwoP(\bbC,R)\to\sftwoP(\bbC,0)$ upgrades to a map of local systems,
\[\ul{\sftwoP}(\bbC,-)\to \ul{\sftwoP}(\bbC,0),\]
where $\ul{\sftwoP}(\bbC,0)$ denotes the constant local system with fibers $\sftwoP(\bbC,0)$.

Fixing $R\in\Conf(n)$, it is well known that $\pi_1(\Conf(n),R)\simeq \Br_R$ is a braid group. So by transport, we obtain a group action $\Br_R\actson \sftwoP(\bbC,R)$.
\begin{prop}\label{prop:intertwines}
    The functor $\FS_{\sftwoP}:\sftwoP(\bbC,R)\to \sftwoP(\bbC,0)$ is naturally $\Br_R$-equivariant, where $\sftwoP(\bbC,0)$ is given the trivial action.
\end{prop}

\begin{cor}\label{cor:discrete}
    The fibers of $\FS_{\sftwoP}:\sftwoP(\bbC,R)\to \sftwoP(\bbC,0)$ are discrete.
\end{cor}

\begin{proof}
    This is immediate from Proposition~\ref{prop:FSisfaithful} together with the fact that $\FS_{\sftwoP}$ is also conservative.
\end{proof}

By combining Proposition~\ref{prop:intertwines} and Corollary~\ref{cor:discrete}, the following definition makes sense.

\begin{defn}\label{defn:periodic2}
    Let $\calG\in\sftwoP(\bbC,R)$. Given $\sigma\in\Br_R$, say that $\calG$ is $\sigma$-periodic if $\calG$ is a fixed point for the $\Br_R$-action acting on the fiber of $\FS_{\sftwoP}$ over $\FS_{\sftwoP}(\calG)$. For a subgroup $H\subseteq \Br_R$, say that $\calG$ is periodic for $H$ if for all $\sigma\in H$, $\calG$ is $\sigma$-periodic.
\end{defn}

We relate the group actions $\Br_R\actson \sftwoP(\bbC,R)$ and $\Br_n\actson \sfSph(n)$, and thus relate Definition~\ref{defn:periodic2} to the notion of periodic SODs given in Definition~\ref{defn:periodic}. We first study the geometry. We claim that $\Cuts(\bbC,R)$ admits in addition to the right torsor action of $\Br_n$, a commuting action of $\Br_R$ by the transport along $\Conf(n)$. A choice $K\in\Cuts(\bbC,R)$ determines
\begin{itemize}
    \item an identification $\sftwoP(\bbC,R)\simeq \sfSph(n)$,
    \item and a group isomorphism $\Br_R\xto{u} \Br_n$ defined by $\sigma\cdot K\simeq K\cdot u(\sigma)$.
\end{itemize}
The actions $\Br_R\actson \sftwoP(\bbC,R)$ and $\Br_n\actson \sfSph(n)$ are identified via these two identifications.

\begin{rmk}
    This is formally similar to the following elementary scenario. Consider a set $S$ of cardinality $n$, and let $\opname{Bij}$ be the set of bijections $f:\{1,\dots,n\}\to S$. Then, $\opname{Bij}$ has a right action by $\Sym_n$ and a left action by $\Sym_S$, by pre- and post-composition respectively. A choice $f\in \opname{Bij}$ determines an isomorphism of groups, $\Sym_S\xto{u} \Sym_n$ defined by the property that for all $\sigma\in\Sym_S$, $\sigma\circ f = f\circ u(\sigma)$.
\end{rmk}

\subsection{The \texorpdfstring{$\two$}{2}-category of schobers on a family of discs}\label{subsection:schobersonfamily}
We upgrade Definition~\ref{defn:family1} and Definition~\ref{defn:family2} to $\two$-categories.

Suppose that $(\bbC\times B,H)$ is a locally constant family of discs, with $B$ connected. Let $f:B\to \Conf(n)$ be the map $b\mapsto H_b$. We define
\[\sftwoP(\bbC\times B,H) := \Gamma(B,f^*\ul{\sftwoP}(\bbC,-)).\]
We can take the enhanced Fukaya--Seidel functor in families, to obtain a functor,
\[\FS_{\sftwoP}:\sftwoP(\bbC\times B,H)\to \sftwoP(\bbC\times B,0\times B).\]

Suppose now that $(\bbC\times B,H)$ is a non-characteristic family of discs. We define the $\two$-category of perverse schobers on $(\bbC\times B,H)$ as the following fibered product of $\two$-categories.
\[\begin{tikzcd}
    {\sftwoP(\bbC\times B,H)}\pullback & {\sftwoP(\bbC\times B,0\times B)} \\
    {\sftwoP(\bbC\times B^{\opname{sm}},H^{\sm})} & {\sftwoP(\bbC\times B^{\opname{sm}},0\times B^{\opname{sm}})}
    \arrow[from=1-1, to=1-2]
    \arrow[from=1-1, to=2-1]
    \arrow["{-\mid_{B^{\opname{sm}}}}", from=1-2, to=2-2]
    \arrow["{\FS_{\sftwoP}}"', from=2-1, to=2-2]
\end{tikzcd}\]

The following statements both use Proposition~\ref{prop:FSisfaithful}.

\begin{prop}\label{prop:restrict}
    Suppose $B$ is contractible and let $(\bbC\times B,H)$ be a non-characteristic family of discs. Fix $b\in B^{\opname{sm}}$. Then, the restriction functor
    \[-\resto{b}:\sftwoP(\bbC\times B,H)\to \sftwoP(\bbC,H_b)\]
    is 2-fully faithful. Moreover, the image is the full sub-$\two$-category
    \[\sftwoP(\bbC,H_b)^{\Image(\sigma)\opname{-periodic}}\]
    consisting of perverse schobers $\calG\in \sftwoP(\bbC,H_b)$ which are periodic for the image of the natural group homomorphism $\sigma:\pi_1(B^{\opname{sm}},b)\to\Br_{H_b}$, as defined in Definition~\ref{defn:periodic2}.
\end{prop}

\begin{proof}
    Consider the following naturally commuting diagram of $\two$-categories where each square is a pullback square.
    \[\begin{tikzcd}
        {\sftwoP(\bbC\times B,H)}\pullback & \sfQ\pullback & {\sftwoP(\bbC\times B,0\times B)} \\
        {\sftwoP(\bbC\times B^{\opname{sm}},H^{\sm})} & {\sfQ^{\opname{sm}}}\pullback & {\sftwoP(\bbC\times B^{\opname{sm}},0\times B^{\opname{sm}})} \\
        & {\sftwoP(\bbC,H_b)} & {\sftwoP(\bbC,0)}
        \arrow["{\opname{I}'}", from=1-1, to=1-2]
        \arrow[from=1-1, to=2-1]
        \arrow[from=1-2, to=1-3]
        \arrow["{\opname{K}}", from=1-2, to=2-2]
        \arrow["{-\mid_{B^{\opname{sm}}}}", from=1-3, to=2-3]
        \arrow["{\opname{J}}", curve={height=-80pt}, from=1-3, to=3-3]
        \arrow["{\opname{I}}", from=2-1, to=2-2]
        \arrow["{-\mid_b}"', curve={height=12pt}, from=2-1, to=3-2]
        \arrow[from=2-2, to=2-3]
        \arrow["{\opname{L}}", from=2-2, to=3-2]
        \arrow["{-\mid_b}", from=2-3, to=3-3]
        \arrow["{\FS_{\sftwoP}}", from=3-2, to=3-3]
    \end{tikzcd}\]
    Here, $\sfQ$ and $\sfQ^{\opname{sm}}$ are defined by the pullback squares. 
    
    
    
    From Proposition~\ref{prop:FSisfaithful}, Proposition~\ref{prop:faithful} and Lemma~\ref{lem:faithfulpullback}, we learn that $\opname{I}$ is 2-fully faithful. 
    By pullback square, $\opname{I}':\sftwoP(\bbC\times B,H)\to \sfQ$ is also fully faithful. Since $B$ is contractible, the functor $\opname{J}$ is an equivalence. By base change, the composition $\opname{L}\circ\opname{K}:\sfQ\to \sftwoP(\bbC,H_b)$ is also an equivalence. Composing, we learn that $\opname{L}\circ\opname{K}\circ \opname{I}'$ is fully faithful. Moreover this is identified with
    \[-\resto{b}:\sftwoP(\bbC\times B,H)\to \sftwoP(\bbC,H_b).\]
    
    
    
    By Proposition~\ref{prop:faithful}, the essential image of $\opname{I}$ consists of pairs $(\calC,\calG_b)$, where $\calC\in\sftwoP(\bbC\times B^{\opname{sm}},0\times B^{\opname{sm}})$, $\calG\in \sftwoP(\bbC,H_b)$ is a fixed point under the action of $\pi_1(B^{\opname{sm}},b)$ on the space of lifts of $\calC\resto{b}$. The action factors through
    \[\rho:\pi_1(B^{\opname{sm}},b)\to \Br_{H_b},\]
    with $\Br_{H_b}$ acting as described in Definition~\ref{defn:periodic2}. So being a fixed point is equivalent to asking that $\calG_b$ is periodic for the image of $\rho$.
    
    The essential image of functor $\opname{I}'$ consists of objects $\calG\in \sfQ\simeq \sftwoP(\bbC,H_b)$ such that $\opname{K}(\calG)$ lies in the essential image of $\opname{I}$. But
    \[\opname{K}(\calG) = (\calC,\calG),\]
    (where $\calC$ is the trivial local system with fibers $\FS_{\sftwoP}(\calG)$) so we are done.
\end{proof}

\begin{prop}\label{prop:independent}
    Definition~\ref{defn:family2} is independent of the choice of $B^{\sm}$.
\end{prop}

\begin{proof}
    Suppose that $B_1^{\sm}, B_2^{\sm}\subset B$ are two different open subsets which both realize $(\bbC\times B,H)$ as a non-characteristic family of discs. Without loss of generality, we may assume $B_1^{\sm}\subseteq B_2^{\sm}$ and fix a common basepoint $b\in B_1^{\sm}\subseteq B_2^{\sm}$.

    See that the natural map
    \[\pi_1(B_1^{\opname{sm}},b)\to \pi_1(B_2^{\opname{sm}},b)\]
    is surjective. A similar proof to that of Proposition~\ref{prop:restrict} proves the desired result, using Corollary~\ref{cor:surjectivepi1} in the place of Proposition~\ref{prop:faithful}.
\end{proof}

Everything carries over to localized and prelocalized perverse schobers.

\begin{eg}\label{eg:ykxn}
    Let $0<m<n$ be integers. Consider the pair $(\bbC^2_{x,y},C_{m,n})$ where $C_{m,n}$ is the curve $y^m=x^n$. This defines a non-characteristic family of discs over $B=\bbC_x$. We apply Proposition~\ref{prop:restrict} and make a choice of cuts $K$ on the fiber of a fixed basepoint $x_0\in \bbC_x$ to learn that $\sftwoP(\bbC^2_{x,y},C_{m,n})$ is equivalent to the sub-$\two$-category
    \[\sfSph(m)^{n\opname{-periodic}}\subset \sfSph(m),\]
    consisting of objects such that the SOD (on the vanishing cycle of $\FS_{\sfSph}(-)$) is $\tau^n$-periodic where $\tau = \sigma_{m-1}\sigma_{m-2}\cdots \sigma_{1}\in\Br_m$. Here, we recall from the end of \S\ref{subsection:universal} that the choice of cuts $K$ identifies $\Br_{C_{m,n}|_{x_0}}$ with $\Br_m$.
\end{eg}

\begin{eg}\label{eg:waldhausen3}
    We continue Example~\ref{eg:waldhausen1} and Example~\ref{eg:waldhausen2}, carrying over notation. We recall that from a spherical adjunction $F$, we constructed an object of $\sfSph(n)$ via the spherical adjunction $S_{n-1}(F)$.
    
    We then gave a non-abstract interpretation of this and proposed that $S_{n-1}(F)$ comes from a pushforward along $f_0(z)=z^n$, and that lifts to $\sfSph(n)$ come from a factorization
    \[\sftwoP(\bbC_z,0)\xto{(f_x)_*} \sftwoP(\bbC_y,R_x)\xto{\FS_{\sftwoP}} \sftwoP(\bbC_y,0).\]
    We expect that these organize into an object of $\sftwoP(\bbC^2_{x,y},y(y^{n-1}-x^n))$. Applying Proposition~\ref{prop:restrict}, this expectation can concretely be interpreted as a specific periodicity property on the SOD of the vanishing cycle $\calA'$ of $S_{n-1}(F)$. When $n=2$, this is the 4-periodicity of the SOD $\langle\calA,\calB\rangle$ which is well known, see for example \cite{dyckerhoffSphericalAdjunctionsStable2021}*{Proposition 2.5.12}. For general $n$, this is periodicity for $\sigma^n$, where for example if $n=5$, then $\sigma$ is the braid illustrated below.
    
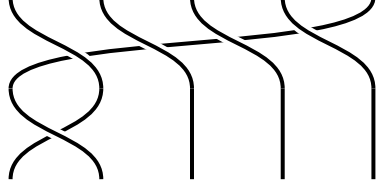
\[\begin{tikzpicture}[scale=1.2]
        \tikzset{
            strand/.style={black, line width=1.5pt},
            mask/.style={white, line width=4.5pt}
        }

        \draw[strand] (5, 2) .. controls (5, 1.35) and (1, 1.65) .. (1, 1);

        \draw[mask] (4, 2) .. controls (4, 1.5) and (5, 1.5) .. (5, 1);
        \draw[strand] (4, 2) .. controls (4, 1.5) and (5, 1.5) .. (5, 1);

        \draw[mask] (3, 2) .. controls (3, 1.5) and (4, 1.5) .. (4, 1);
        \draw[strand] (3, 2) .. controls (3, 1.5) and (4, 1.5) .. (4, 1);

        \draw[mask] (2, 2) .. controls (2, 1.5) and (3, 1.5) .. (3, 1);
        \draw[strand] (2, 2) .. controls (2, 1.5) and (3, 1.5) .. (3, 1);

        \draw[mask] (1, 2) .. controls (1, 1.5) and (2, 1.5) .. (2, 1);
        \draw[strand] (1, 2) .. controls (1, 1.5) and (2, 1.5) .. (2, 1);

        \draw[strand] (3, 1) -- (3, 0);
        \draw[strand] (4, 1) -- (4, 0);
        \draw[strand] (5, 1) -- (5, 0);

        \draw[strand] (2, 1) .. controls (2, 0.5) and (1, 0.5) .. (1, 0);

        \draw[mask] (1, 1) .. controls (1, 0.5) and (2, 0.5) .. (2, 0);
        \draw[strand] (1, 1) .. controls (1, 0.5) and (2, 0.5) .. (2, 0);

    \end{tikzpicture}\]
    This is more fundamentally described as the braid which rotates the $n$ roots of $y(y^{n-1}-1)=0$ by $\frac{2\pi i}{n-1}$.

    A proof of this expectation can be deduced from \cite{dyckerhoffNsphericalFunctorsCategorification2023}*{Theorem 5.4.2}. One first shows that $\sigma\cdot\frakS=\tau_{n-1}(\frakS)$, where $\tau_{n-1}$ is the autoequivalence coming from the paracyclic structure constructed in \cite{dyckerhoffSphericalAdjunctionsStable2021}. Now, $\tau_{n-1}^n=\Sigma^2$ (\cite{dyckerhoffSphericalAdjunctionsStable2021}*{Discussion following Proposition 3.1.3}) is the double suspension, so that $\sigma^n\cdot\frakS = \frakS$.

\end{eg}

\begin{eg}\label{eg:y2x2localized}
    Consider the pair $(\bbC_{x,y}^2,C_{2,2})$ where $C_{2,2}$ is defined in Example~\ref{eg:ykxn}. This is simply the union of two lines, $y=\pm x$. We compute $\olsftwoP(\bbC^2_{x,y},C_{2,2})$. By Proposition~\ref{prop:restrict} (with a fixed pointed choice of cuts), $\olsftwoP(\bbC_{x,y}^2,C_{2,2})$ is equivalent to the sub-$\two$-category of $\sfAut(2)^{\opname{2-periodic}}\subset \sfAut(2)$, given by objects such that the SOD $\Phi=\langle\Phi_1,\Phi_2\rangle$ is $2$-periodic.

    Being $2$-periodic is equivalent to asking that $\Phi_1$ is both left and right orthogonal to $\Phi_2$. In other words, $\Phi=\Phi_1\oplus\Phi_2$. The autoequivalence $T$ also splits $T=T_1\oplus T_2$, where $T_i$ is an autoequivalence of $\Phi_i$.

    Thus there is an equivalence,
    \[\sfAut(2)^{\opname{2-periodic}}\simeq \sfAut \oplus \sfAut.\]
    We note that in the setting of \cite{dyckerhoffNsphericalFunctorsCategorification2023}, an SOD being 2-periodic is equivalent to the gluing functor being zero, which yields the same conclusion.

    Geometrically, we interpret this as the observation that the Legendrian $\Lambda_{C_{2,2}}^\infty$ is a disjoint union of two smooth Legendrians.
\end{eg}

\begin{eg}\label{eg:y2x3localized}
    Consider the pair $(\bbC_{x,y}^2,C_{2,3})$ where $C_{2,3}$ is defined in Example~\ref{eg:ykxn}. We compute $\olsftwoP(\bbC^2_{x,y},C_{2,3})$. We apply Proposition~\ref{prop:restrict} to learn that $\olsftwoP(\bbC_{x,y}^2,C_{2,3})$ is equivalent to the sub-$\two$-category of $\sfAut(2)^{\opname{3-periodic}}\subset \sfAut(2)$, given by objects such that the SOD $\Phi=\langle\Phi_1,\Phi_2\rangle$ is $3$-periodic.

    By Example~\ref{eg:aut2}, an object in the image is given by the data of a 3-periodic functor $F:\Phi_1\to\Phi_2$ admitting all adjoints, autoequivalences $T_1,T_2$ of $\Phi_1$, $\Phi_2$ respectively, together with an equivalence, $F\circ T_1\simeq T_2\circ F^{RR}$.

    It is shown in \cite{dyckerhoffNsphericalFunctorsCategorification2023} that a functor $F$ is 3-periodic if and only if it is an equivalence. In particular, $F^{RR}\simeq F$, and so $T_1$ and $T_2$ are identified via $F$. So the data is equivalent to just giving a single pair $(\Phi_1,T_1)$, defining an equivalence,
    \[\sfAut(2)^{\opname{3-periodic}}\simeq \sfAut.\]

    Geometrically, we interpret this as the observation that $\Lambda_{C_{2,3}}^\infty$ is a smooth Legendrian. We will further interpret this as an instance of Radon transform in Example~\ref{eg:yx3}.
\end{eg}

\begin{rmk}\label{rmk:cerf}
    There is a fruitful analogy between perverse schobers on $(\bbC_y,R)$ and Picard-Lefschetz theory, with sheaves on $(\bbR,L)$ and Morse theory, where $L$ is a singular support condition containing the zero section and containing only covectors in non-negative codirections.

    We extend this analogy to schobers on a family of discs, $(\bbC^2_{x,y},C)$, non-characteristic over $B=\bbC_x$. This is analogous to Cerf theory, which concerns how generic families of Morse functions may degenerate. In the classical theory, there are two types of Cerf singularities, called exchange singularities and birth-death singularities.

    The families of (localized) perverse schobers appearing in Example~\ref{eg:y2x2localized} and Example~\ref{eg:y2x3localized} can be thought of as analogues of these two types of singularities.
\end{rmk}

\subsection{Schobers supported on a cuspidal cubic}\label{subsection:y2x3}
We make a detailed study of the category $\sftwoP(\bbC^2_{x,y},C_{2,3})$, where the $C_{2,3}$ is the curve $y^2=x^3$.

The category of perverse sheaves microsupported along the associated Lagrangian $\Lambda_{C_{2,3}}$ is described by MacPherson--Vilonen \cite{macphersonPerverseSheavesSingularities1988}*{Proposition 3.1}, which we recall below. We aim to show that our definition categorifies this.

\begin{prop}\label{prop:mv}
    Let $A=\bbC\langle a_1,a_2\rangle/(a_1a_2a_1 -a_1^2 +a_1, a_2a_1a_2 -a_2^2 +a_2)$. Then, the category $\Perv(\bbC^2_{x,y},C_{2,3})$ is equivalent to the abelian category of $A$-modules such that $1-a_1$ and $1-a_2$ are invertible.
\end{prop}

We show that our definition of $\sftwoP(\bbC^2_{x,y},C_{2,3})$ decategorifies to Proposition~\ref{prop:mv}. Note however that we must be in the setting of large categories in order to use Proposition~\ref{prop:KLisEM}, although it is not clear if this is necessary.

\begin{prop}\label{prop:y2x3description}
    The $\two$-category $\sfSph(2)^{\opname{3-periodic}}$ is equivalent to the $\two$-category whose objects are $(\Psi,F_1,F_2)$, where
    \begin{itemize}
        \item $\Psi$ is a stable category and
        \item $F_1,F_2$ are spherical monads on $\Psi$ with units $u_i:\id\to F_i$,
    \end{itemize}
    such that the following squares are Cartesian.
    \[\begin{tikzcd}
        {F_1F_1} & {F_1F_2F_1} \\
        {F_1} & 0
        \arrow["{1\circ u_2\circ 1}"', from=1-1, to=1-2]
        \arrow["{m_1}", from=1-1, to=2-1]
        \arrow[from=1-2, to=2-2]
        \arrow[from=2-1, to=2-2]
    \end{tikzcd}
    \begin{tikzcd}
        {F_2F_2} & {F_2F_1F_2} \\
        {F_2} & 0
        \arrow["{1\circ u_1\circ 1}"', from=1-1, to=1-2]
        \arrow["{m_2}", from=1-1, to=2-1]
        \arrow[from=1-2, to=2-2]
        \arrow[from=2-1, to=2-2]
    \end{tikzcd}
    \]
    In particular, taking $K$-theory yields a perverse sheaf microsupported on $C_{2,3}$ via Proposition~\ref{prop:mv}.
\end{prop}

We begin with a straightforward description of $\sftwoP(\bbC^2_{x,y},C_{2,3})$.

\begin{lem}\label{lem:y2x3easy}
    The $\two$-category $\sftwoP(\bbC^2_{x,y},C_{2,3})$ is equivalent to the $\two$-category whose objects are $(\Phi_1,\Phi_2,\Psi,S_1,S_2)$, consisting of
    \begin{itemize}
        \item stable categories $\Phi_1,\Phi_2,\Psi$, and
        \item spherical functors $S_i:\Phi_i\to \Psi$ for $i=1,2$,
    \end{itemize}
    such that $S_1^RS_2$ is an equivalence.
\end{lem}

\begin{proof}
    By Example~\ref{eg:ykxn}, we have an equivalence,
    \[\sftwoP(\bbC^2_{x,y},C_{2,3})\simeq \sfSph(2)^{\opname{3-periodic}}.\]
    We recall from \cite{dyckerhoffNsphericalFunctorsCategorification2023} that a 2-term SOD is 3-periodic if and only if its gluing functor is an equivalence. The (Cartesian) gluing functor here is $S_1^RS_2$.
\end{proof}

\begin{proof}[Proof of Proposition~\ref{prop:y2x3description}]
    We begin with the description of Lemma~\ref{lem:y2x3easy}. We deduce from $S_1^RS_2$ being invertible that $S_1$ is conservative so we can describe $\Phi_1$ monadically in terms of $\Psi$, as discussed in \S\ref{subsection:prelocalized} (although the role of $\Phi,\Psi$ are reversed). We have likewise for $\Phi_2$. Thus, an object of $\sftwoP(\bbC^2_{x,y},C_{2,3})$ is equivalent to the data of
    \begin{itemize}
        \item a stable category $\Psi$,
        \item spherical monads $F_1,F_2$ on $\Psi$,
        \item such that $S_1^RS_2$ is an equivalence,
    \end{itemize}
    where $S_i:\Mod_{\Psi}(F_{i})\to\Psi$ is the forgetful functor, so that $F_{i} = S_iS_i^L$.

    We use the notation
    \[S_i^LS_i\to 1\to T_i\xto{+1}\]
    for relevant twist functors, so that $T_iS_i^R[-1]\isomto S_i^L$.

    To ask for $S_1^RS_2$ to be an equivalence is equivalent to asking that
    \begin{enumerate}
        \item\label{item:unit} the unit $1\to S_1^RS_2S_2^LS_1$ is an equivalence, and
        \item\label{item:counit} the counit $S_2^LS_1S_1^RS_2\to 1$ is an equivalence.
    \end{enumerate}
    We begin with (\ref{item:unit}). We post-compose with $T_i[-1]$ and ask instead that
    \[T_1[-1]\to S_1^LS_2S_2^LS_1\]
    is an equivalence.

    We use a strategy similar to the proof of Proposition~\ref{prop:sphmonad} in order to rewrite the last condition in terms of the monads. By Proposition~\ref{prop:KLisEM}, $S_i^L$ is (stably) essentially surjective. Also, $S_1$ is conservative, so we may post-compose and pre-compose by these respectively. So we ask that
    \[S_1T_1S_1^L[-1]\to S_1S_1^LS_2S_2^LS_1S_1^L\]
    is an equivalence. Rewriting this in terms of monads, condition (\ref{item:unit}) is equivalent to asking that the first of the squares in the statement is Cartesian. Similarly, condition (\ref{item:counit}) is equivalent to asking that the second square is Cartesian.
\end{proof}

We next relate our description to the framed $\bbA_2$-schobers defined by Dyckerhoff--Wedrich \cite{dyckerhoffPerverseSchobersCoxeter2025}. We recall the definition here.

\begin{defn}\label{defn:dw}
    A framed $\bbA_2$-schober is given by a commuting square of stable categories,
    \[\begin{tikzcd}
        {\calA_{3}} & {\calA_{1,2}} \\
        {\calA_{2,1}} & {\calA_{1,1,1}}
        \arrow["I", from=1-1, to=1-2]
        \arrow["H"', from=1-1, to=2-1]
        \arrow["F"', from=1-2, to=2-2]
        \arrow["G", from=2-1, to=2-2]
    \end{tikzcd}\]
    such that
    \begin{enumerate}
        \item all functors admit right adjoints,
        \item $F,G$ are spherical,
        \item the cones of the Beck--Chevalley maps $HI^R \to G^RF, IH^R\to F^RG$ are equivalences,
        \item certain Beck--Chevalley squares (\cite{dyckerhoffPerverseSchobersCoxeter2025}*{(2.3.3),(2.3.4)}) are Cartesian, and
        \item the cotwist $T_3:\calA_3\to\calA_3$ (\cite{dyckerhoffPerverseSchobersCoxeter2025}*{Example 3.26}) is an equivalence.
    \end{enumerate}
    We denote by $\sftwoP^{\opname{DW}}(\bbA_2)$ the $\two$-category of $\bbA_2$-schobers.
\end{defn}

\begin{prop}\label{prop:agreeswithdw}
    The $\two$-category $\sftwoP(\bbC^2_{x,y},C_{2,3})$ is equivalent to the full sub-$\two$-category of $\sftwoP^{\opname{DW}}(\bbA_2)$ consisting of $\bbA_2$-schobers such that $\calA_{3}=0$.
\end{prop}

\begin{proof}
    An $\bbA_2$-schober with $\calA_{3}=0$ is given by a diagram of stable categories,
    \[\begin{tikzcd}
        {} & {\calA_{1,2}} \\
        {\calA_{2,1}} & {\calA_{1,1,1}}
        \arrow["F"', from=1-2, to=2-2]
        \arrow["G", from=2-1, to=2-2]
    \end{tikzcd}\]
    such that
    \begin{enumerate}
        \item all functors admit all repeated adjoints,
        \item $F,G$ are spherical,
        \item the compositions $G^RF$ and $F^RG$ are equivalences,
        \item certain Beck--Chevalley squares (\cite{dyckerhoffPerverseSchobersCoxeter2025}*{(2.3.3),(2.3.4)}) are Cartesian, and
        \item the cotwist $T_3:\calA_3\to\calA_3$ is an equivalence.
    \end{enumerate}
    One checks that the condition on Beck--Chevalley squares follows from the remaining ones. The cotwist invertibility is trivially satisfied. So we recover the description given in Lemma~\ref{lem:y2x3easy}.
\end{proof}

We can interpret Proposition~\ref{prop:agreeswithdw} as follows. We think of $\bbA_2$-schobers as a candidate for perverse schobers on $(\bbC^2_{x,y},C_{2,3}\cup T^*_{0}\bbC^2)$. Their singular support condition differs from $\Lambda_{C_{2,3}}$ by the conormal at zero. The ``microstalk'' at a point here is $\calA_{3}$, so we may ``microlocalize'' to $\Lambda_{C_{2,3}}$ by setting this to zero, providing a candidate for perverse schobers on $(\bbC^2,C_{2,3})$. Proposition~\ref{prop:agreeswithdw} verifies that this candidate agrees with our definition. In fact, this sub-$\two$-category is enough to encode the $A_2$-configurations as discussed in \cite{dyckerhoffPerverseSchobersCoxeter2025}*{\S 3.6}. 

\begin{eg}\label{eg:y2x3geometry}
    Let $k=\bbC$. Consider for each $x\neq0$, the perverse schober coming from a symplectic fibration,
    \[\calF_x := \calF(\bbC_u,\frac12(u^3-3xu))\in\sftwoP(\bbC_y,R_x).\]
    Here, $R_x=\Crit(\frac12(u^3-3xu))=\{\pm x^{3/2}\}$. This defines an object
    \[\calF\resto{\bbC_x-0}\in\sftwoP((\bbC_x-0)\times\bbC_y,y^2=x^3).\]
    In order to extend over $x=0$, we must exhibit a trivialization of $\FS_{\sftwoP}(\calF_x)$.

    The category underlying $\FS_{\sftwoP}(\calF_x)$ is exactly the Fukaya--Seidel category (in the original sense) of the pair $(\bbC_u,\frac12(u^3-3xu))$, where we have chosen a stop in a fixed direction $\theta$. This is independent of $x$ since the leading term $u^3$ is independent of $x$. By varying $\theta$, this exhibits the desired trivialization.

    We can also study this perverse schober by restriction to the fiber $x=1$ via Proposition~\ref{prop:restrict}. After making a pointed choice of cuts from the critical values to a far away point $y_\infty$, we obtain the object of $\sfSph(2)$ given by
    \[S_1,S_2:\Vect\to \Fuk(u^3-3u=y_\infty),\]
    where $S_1,S_2$ are two vanishing cycles. These vanishing cycles famously intersect once, so that $S_2^RS_1:\Vect\to \Vect$ is an equivalence. By Lemma~\ref{lem:y2x3easy}, this defines a perverse schober on $(\bbC^2,C_{2,3})$.
\end{eg}

We expect that a similar story continues to hold for $\bbA_N$-schobers, by observing that the discriminant locus $\Delta\subset \bbC^{N-1}\sslash \Sym_N$ defines a non-characteristic family of discs by treating the $N$\ts{th} elementary symmetric polynomial as the fiber coordinate. More generally, we expect that from a perverse schober on $(\bbC^{N+1},H)$, we may produce a (perverse) local system of categories on $\bbC^{N+1}-H$, so that we obtain an action of $\pi_1(\bbC^{N+1}-H,p_0)$ on the nearby cycles category at $p_0\in\bbC^{N+1}-H$.

\section{Radon transforms}\label{section:radon}
In this section we formulate a conjectural Radon transform for localized perverse schobers between dual pairs of non-characteristic families of discs.

We begin by recalling the Radon transform for perverse sheaves on dual projective spaces $\bbP^n$ and $\check{\bbP}^n$ so that points on one parametrize hyperplanes in the other. The geometry underlying the Radon transform is the isomorphism of cosphere bundles,
\[T^\infty\bbP^N\simeq T^\infty\check{\bbP}^N.\]
Given a Legendrian $\Lambda^\infty\subseteq T^\infty\bbP^N$, let $\Lambda\subseteq T^*\bbP^N$ and $\check{
\Lambda}\subseteq T^*\check{\bbP}^N$ be the respective conic Lagrangians (with the zero section added). Radon transform of localized perverse sheaves is an equivalence of categories,
\[\Perv_\Lambda(\bbP^N)/\Loc(\bbP^N)\simeq\Perv_{\check{
\Lambda}}(\check{\bbP}^N)/\Loc(\check{\bbP}^N),\]
see for example \cite{kiehlLefschetzTheoryBrylinskiRadon2001}*{Chapter IV}.

More relevant to us is the following. Consider $\bbC^N_{\ul{x},y}$ with coordinates $x_1,\dots,x_{N-1},y$ and dually, $\bbC^N_{\ul{a},b}$ with coordinates $a_1,\dots,a_{N-1},b$. The point $(\ul{a},b)$ parametrizes the line $y=\ul{a}\cdot\ul{x}-b$. There is an isomorphism of contact manifolds,
\begin{equation}\label{equation:cosphere}
T^\infty_{\opname{nc}}\bbC^N_{\ul{x},y}\simeq T^\infty_{\opname{nc}}\bbC^N_{\ul{a},b},
\end{equation}
where the left-hand side is the open subset of $T^\infty\bbC^N_{\ul{x},y}$ consisting of those covectors non-characteristic for the projection to $\bbC^{N-1}_{\ul{x}}$, and the right-hand side is defined similarly. Whenever $\Lambda^\infty$ is a closed Legendrian inside these open subsets, there is an equivalence of categories,
\[\Perv_\Lambda(\bbC^N_{\ul{x},y})/\Loc(\bbC^N_{\ul{x},y})\simeq\Perv_{\check{
\Lambda}}(\bbC^N_{\ul{a},b})/\Loc(\bbC^N_{\ul{a},b}).\]

\begin{rmk}
    Under Radon transform for projective spaces, passage to an affine open subset $\bbC^N\subset \bbP^N$ by removing a hyperplane $h\subset \bbP^N$ on one side corresponds to restricting to those covectors which are non-characteristic for the projection from $\check{h}\in\check{\bbP}^N$. So to obtain a Radon transform where both sides are affine opens, we must also have non-charactericity on both sides as we have seen.
\end{rmk}

\begin{rmk}
    The contact manifold $T^\infty_{\opname{nc}}\bbC^N_{\ul{x},y}$ can also be described as the first jet bundle, $J^1\bbC^N_{\ul{x}}$. As such, (\ref{equation:cosphere}) says that the first jet bundles of dual vector spaces are canonically identified.
\end{rmk}

Suppose that $(\bbC^{N-1}_{\ul{x}}\times \bbC_y,H)$ is a non-characteristic family of discs. Then the Legendrian $\Lambda_H^\infty$ lives inside $T^\infty_{\opname{nc}}\bbC^N_{\ul{x},y}$. Assuming $H$ contains no hyperplane components, we have under the identification above,
\[\Lambda_H^\infty \simeq \Lambda_{\check{H}}^\infty,\]
where $\check{H}\subset \bbC^N_{\ul{a},b}$ parametrizes tangent hyperplanes of $H$, or limits thereof. We may now state the main conjecture of this paper.

\begin{conj}\label{conj:radon}
    Suppose that $(\bbC_y\times\bbC^{N-1}_{\ul{x}},H)$ and $(\bbC_b\times \bbC^{N-1}_{\ul{a}},\check{H})$ are non-characteristic families of discs which are dual in the manner described above. Then, there is an equivalence of $\two$-categories,
    \[\olsftwoP(\bbC^N_{\ul{x},y},H)\simeq \olsftwoP(\bbC^N_{\ul{a},b},\check{H}).\]
\end{conj}

We will refer to this conjectural equivalence as a Radon transform for perverse schobers.

\begin{eg}\label{eg:yx3}
    Let $C=\{y=x^3\}$. This defines a locally constant family of discs, and so
    \[\olsftwoP(\bbC^2_{x,y},\{y=x^3\})\simeq\sfAut.\]
    The dual curve is $\check{C}_{1,3} = \{b^2=\frac{4}{27}a^3\}$. Ignoring multiplicative constants, we computed in Example~\ref{eg:y2x3localized} that there is an equivalence,
    \[\olsftwoP(\bbC^2_{a,b},\check{C}_{1,3})\simeq\sfAut.\]
    These $\two$-categories are evidently equivalent.
\end{eg}

Generalizing this example, we have the following.

\begin{thm}\label{thm:ykxnschober}
    Let $0<m<n$ be integers and let $C_{m,n} = \{y^m = x^n\}$. Then, there is a bijection between equivalence classes of objects of the $\two$-categories,
    \[\olsftwoP(\bbC^2_{x,y},C_{m,n})\overset{\opname{bijection}}{\simeq} \olsftwoP(\bbC^2_{a,b},\check{C}_{m,n}).\]
\end{thm}

We note that in Example~\ref{eg:yx3} and Theorem~\ref{thm:ykxnschober}, the construction of the equivalence depends on the choice of quadratic forms along the base as discussed at the start of \S\ref{section:families}, the choice of a point on the base to apply Proposition~\ref{prop:restrict}, and the choice of cuts along the resulting fiber. Ultimately, the Radon transform should be independent of such choices. In practice, we will not attempt to show that our constructions are canonical, postponing this to future work. 

In \S\ref{subsection:radonissuspension}, we explain what goes wrong with this Radon transform if we had worked with prelocalized schobers instead of localized schobers, and a posteriori explain the significance of Observation~\ref{obs:stable}.

In \S\ref{subsection:ykxn}, we prove Theorem~\ref{thm:ykxnschober} by explicitly exhibiting an equivalence in a combinatorial manner, generalizing Example~\ref{eg:yx3}.

\subsection{Radon transforms and suspension}\label{subsection:radonissuspension}
We consider the easiest interesting instance of Radon transform given by
\[C=\{y=x^2\}\subset\bbC^2_{x,y},\]
so that our input schober is supported along a parabola. The dual curve $\check{C} = \{b=\frac14 a^2\}$ is also a parabola. We informally attempt to illustrate what goes wrong if we had worked with the prelocalized perverse schobers, and thus motivate the stabilization construction in \S\ref{subsection:stabilization}.

It turns out that no issues arise with prelocalized schobers during \S\ref{section:families}, so that we may define
\[\olsftwoP^{\pre}(\bbC^2_{x,y},\{y=x^2\}).\]
In this case, our family is locally constant, and we have an equivalence,
\[-\resto{x=0}:\olsftwoP^{\pre}(\bbC^2_{x,y},\{y=x^2\})\to\olsftwoP^{\pre}(\bbC_{y},\{0\})\simeq\sfSphMnd.\]
On the other side, we have similarly,
\[-\resto{a=0}:\olsftwoP^{\pre}(\bbC^2_{a,b},\{b=\frac14 a^2\})\to\olsftwoP^{\pre}(\bbC_{b},\{0\})\simeq\sfSphMnd.\]

While the two sides are obviously equivalent, this equivalence should not be called the Radon transform. Geometrically, the composition,
\[\olsftwoP^{\pre}(\bbC^2_{x,y},\{y=x^2\})\xto{\opname{Radon}}\olsftwoP^{\pre}(\bbC^2_{a,b},\{b=\frac14 a^2\})\xto{-\mid_{a=a_0}}\olsftwoP^{\pre}(\bbC_{b},\{\frac14 a_0^2\})\]
is given by ``pushforward along the map $-y+a_0x:\bbC^2_{x,y}\to \bbC_b$'', as illustrated in Figure~\ref{figure:radonparabola}. Setting $a_0=0$, we obtain pushforward along $y$.

\begin{figure}[ht]
    \caption{Projection on $\bbC^2_{x,y}$ is dual to restriction on $\bbC^2_{a,b}$}
    \label{figure:radonparabola}

    \begin{tikzpicture}[scale=3]
    \clip (-1, -0.5) rectangle (1.2, 0.7);

    \draw[->, gray, line width=0.6pt] (-1, -0.4) -- (1.2, -0.4) node[above left, gray] {$x$};
    \draw[->, gray, line width=0.6pt] (-0.9, -0.5) -- (-0.9, 0.7) node[below right, gray] {$y$};


    \draw[black, thick, domain=-1:1, samples=100] plot (\x, {0.5*\x*\x}) node[below, black] {$C$};

    \draw[<-, red, thick, domain=-0.1:0.8, samples=100] plot (\x, {0.5*\x+0.5*0.5*0.5});
    \draw[<-, red, thick, domain=-0.05:0.85, samples=100] plot (\x, {0.5*\x});
    \draw[<-, blue, thick, domain=0:0.9, samples=100] plot (\x, {0.5*\x-0.5*0.5*0.5});
    \draw[<-, red, thick, domain=0.05:0.95, samples=100] plot (\x, {0.5*\x-0.5*0.5});
    \draw[<-, red, thick, domain=0.1:1, samples=100] plot (\x, {0.5*\x-1.5*0.5*0.5});

    \end{tikzpicture}
    \begin{tikzpicture}[scale=3]
    \clip (-1, -0.5) rectangle (1.2, 0.7);

    \draw[->, gray, line width=0.6pt] (-1, -0.4) -- (1.2, -0.4) node[above left, gray] {$a$};
    \draw[->, gray, line width=0.6pt] (-0.9, -0.5) -- (-0.9, 0.7) node[below right, gray] {$b$};


    \draw[black, thick, domain=-1:1, samples=100] plot (\x, {0.25*\x*\x}) node[below, black] {$C^{\vee}$};

    \draw[red, thick] (0.5, -0.5) -- (0.5, 0.5) node[right, red] {$a=a_0$};
    \fill[blue] (0.5,0.25*0.5*0.5) circle (0.5pt);
    \end{tikzpicture}
\end{figure}

So the functor out of $\olsftwoP^{\pre}(\bbC_{y},\{0\})\simeq\sfSphMnd$ takes a prelocalized perverse schober on $(\bbC_y,0)$, spreads it out over $(\bbC_{x,y},y=x^2)$, and then pushes this forward along $y$. Geometrically, this amounts to the operation discussed in Observation~\ref{obs:stable} and abstracted to the operation $\calSus_{\sfSphMnd}:\sfSphMnd\to\sfSphMnd$.

We can organize this discussion in the following diagram.
\[\begin{tikzcd}
    {\olsftwoP^{\pre}(\bbC^2_{x,y},\{y=x^2\})} & {\olsftwoP^{\pre}(\bbC^2_{a,b},\{b=\frac14 a^2\})} \\
    {\olsftwoP^{\pre}(\bbC_{y},\{0\})} & {\olsftwoP^{\pre}(\bbC_{b},\{0\})} \\
    \sfSphMnd & \sfSphMnd
    \arrow["{\opname{Radon}}", from=1-1, to=1-2]
    \arrow["{-\mid_{x=0}}"', from=1-1, to=2-1]
    \arrow["{\text{``}y_*\text{''}}"{description}, from=1-1, to=2-2]
    \arrow["{-\mid_{a=0}}"', from=1-2, to=2-2]
    \arrow["\text{Observation}~\ref{obs:stable}", from=2-1, to=2-2]
    \arrow["\simeq"', from=2-1, to=3-1]
    \arrow["\simeq"', from=2-2, to=3-2]
    \arrow["{\calSus_{\sfSphMnd}}", from=3-1, to=3-2]
\end{tikzcd}\]

We may summarize the conclusion as follows.
\begin{obs}\label{obs:radonissuspension}
    Radon transform for prelocalized schobers on the curve $\{y=x^2\}$ should be given by suspension $\calSus_{\sfSphMnd}$.
\end{obs}

As pointed out in Observation~\ref{obs:key}, $\calSus_{\sfSphMnd}$ is not an equivalence, so the prelocalized analog of Conjecture~\ref{conj:radon} doesn't work, in turn motivating the necessity to stabilize $\calSus_{\sfSphMnd}$. A similar line of reasoning motivates stabilization on $\olsftwoP^{\pre}(\bbC_{y},R)$ for general $R$.

\subsection{Radon transform for torus singularities}\label{subsection:ykxn}
In this subsection, we will prove Theorem~\ref{thm:ykxnschober}. Recall that $0<m<n$ and $C_{m,n}$ is the curve
\[y^m=x^n.\]
We recall from Example~\ref{eg:ykxn} that there is an equivalence,
\[\olsftwoP(\bbC^2_{x,y},C_{m,n})\simeq \sfAut(m)^{n\opname{-periodic}},\]
where $\tau = \sigma_m\sigma_{m-1}\cdots \sigma_{1}\in\Br_m$ and where $\sfAut(m)^{n\opname{-periodic}}$ is the full sub-$\two$-category of $\sfAut(m)$ for which the SOD is $\tau^n$-periodic.
Similarly, since $\check{C}_{m,n}$ is $C_{n-m,n}$ up to some constants, there is an equivalence,
\[\olsftwoP(\bbC^2_{a,b},\check{C}_{m,n})\simeq \sfAut(n-m)^{n\opname{-periodic}}.\]
So Theorem~\ref{thm:ykxnschober} admits the following abstract restatement.

\begin{thm}\label{thm:ykxn}
    Let $0<m<n$ be integers. There is a bijection between equivalence classes of objects of $\two$-categories,
    \[\sfAut(m)^{n\opname{-periodic}}\overset{\opname{bijection}}{\simeq} \sfAut(n-m)^{n\opname{-periodic}}.\]
\end{thm}

Let $\Phi$ be a stable category, $q\geq1$, and suppose $S_i:\Phi_i\to \Phi$, $i=1,\dots,q$ are right-admissible subcategories. The slightly unusual notation is because we view this as an object $\calF\in \sfAdj(q)$. We give an explicit description of $\FS_{\sfAdj}(\calF)$ in terms of complexes in $\Phi$, generalizing the discussion in \S\ref{subsection:enhancedFS} in the special case that each $S_i$ is fully faithful.

\begin{rmk}
    We note that inclusions of non-zero subcategories are never spherical, and so we will really need to work outside of this setting.
\end{rmk}

\begin{lem}\label{lem:bicomplex}
    Let $S_i:\Phi_i\to \Phi$, $i=1,\dots,q$ be right-admissible subcategories, with right adjoints $R_i$. Let $\calF\in \sfAdj(q)$ be the resulting object. Define $\Phi_{1,\dots,q}$ to be the full subcategory of bicomplexes in $\Phi$ of the following form,
    \[\begin{tikzcd}[ampersand replacement=\&]
        {\phi^q_q} \& 0 \& 0 \& 0 \\
        \vdots \& \ddots \& 0 \& 0 \\
        {\phi^2_q} \& \dots \& {\phi^2_2} \& 0 \\
        {\phi^1_q} \& \cdots \& {\phi^1_2} \& {\phi^1_1}
        \arrow[from=1-1, to=1-2]
        \arrow[from=1-2, to=1-3]
        \arrow[from=1-3, to=1-4]
        \arrow[from=2-1, to=1-1]
        \arrow[from=2-1, to=2-2]
        \arrow[from=2-2, to=1-2]
        \arrow[from=2-2, to=2-3]
        \arrow[from=2-3, to=1-3]
        \arrow[from=2-3, to=2-4]
        \arrow[from=2-4, to=1-4]
        \arrow[from=3-1, to=2-1]
        \arrow[from=3-1, to=3-2]
        \arrow[from=3-2, to=2-2]
        \arrow[from=3-2, to=3-3]
        \arrow[from=3-3, to=2-3]
        \arrow[from=3-3, to=3-4]
        \arrow[from=3-4, to=2-4]
        \arrow[from=4-1, to=3-1]
        \arrow[from=4-1, to=4-2]
        \arrow[from=4-2, to=3-2]
        \arrow[from=4-2, to=4-3]
        \arrow[from=4-3, to=3-3]
        \arrow[from=4-3, to=4-4]
        \arrow[from=4-4, to=3-4]
    \end{tikzcd}\]
    such that for each $1\leq i\leq j\leq q$, 
    \[\phi^i_j\in \Phi_i,\]
    and such that for each $1\leq i< j\leq q$, the complex
    \[\phi^i_j\to R_i\phi^{i+1}_j\to R_i\phi^{i+2}_j\to\cdots\to R_i\phi^j_j\]
    is acyclic. In other words, its totalization is zero.
    
    Define functors
    \begin{align*}
        I_l^R: \Phi_{1,\dots,q}&\to \Phi_l\\
        (\phi^i_j)&\mapsto \phi^l_l.
    \end{align*}
    Then, these admit fully faithful left adjoints, $I_i:\Phi_i\inclto \Phi_{1,\dots,q}$, and these induce a $q$-term right-admissible SOD
    \[\Phi_{1,\dots,q} = \langle\Phi_1,\dots,\Phi_q\rangle.\]
    Moreover, there is an equivalence $\FS_{\sfAdj}(\calF)\simeq \Phi_{1,\dots,q}$ and the natural SOD on $\FS_{\sfAdj}(\calF)$ agrees with this one.

    Furthermore, let $\Complex(\Phi_1,\dots,\Phi_q)$ denote the full subcategory of complexes in $\Phi$ of the form,
    \[\phi^1\to\phi^2\to\cdots\to \phi^q,\]
    such that for each $1\leq i\leq q$, $\phi_i\in\Phi_i$. Then, there is an equivalence of categories
    \[\Phi_{1,\dots,q}\isomto \Complex(\Phi_1,\dots,\Phi_q)\]
    given by totalizing in the horizontal direction.
\end{lem}

\begin{proof}
    The case $n=2$ follows from the discussion of \S\ref{subsection:enhancedFS}, by noting that any morphism $S_ia\to S_ib$ admits a unique lifting along the fully faithful functor $S_i$. The result follows from this one by induction on $n$ and by using the 2-Segal structure of $\sfAdj(-)$ proven in Proposition~\ref{prop:bdjisadj}.
\end{proof}

In addition, we show that $\Phi_{1,\dots,q}$ admits an opposite \textit{left} admissible SOD, defined analogously to Definition~\ref{defn:sod}.

\begin{lem}\label{lem:complexSOD}
    In the situation of Lemma~\ref{lem:bicomplex}, let $P_l:\Complex(\Phi_1,\dots\Phi_q)\to \Phi_l$ be the map
    \[(\phi^i)\mapsto \phi^l.\]
    Then, $P_l$ admits a fully faithful right adjoint $P_l^R$, and the resulting subcategories form a \textit{left} admissible SOD,
    \[\Phi = \langle\Phi_q,\dots,\Phi_1\rangle.\]

    If moreover the right-admissible SOD $\frakS$ of $\Complex(\Phi_1,\dots\Phi_q)$ is in fact $\infty$-admissible, then this left-admissible SOD is also $\infty$-admissible and is equal to the mutation $\delta^{-1}\cdot\frakS$.
\end{lem}

\begin{proof}
    The case $n=2$ can be proven directly, see for example \cite{dyckerhoffSphericalAdjunctionsStable2021}*{Proposition 2.3.2}. For general $n$, the result follows by induction. Explicitly, $P_l^R\phi^l$ is given by the bicomplex $(\phi^i_j)$ where
    \[\phi^i_j=0\]
    unless $j=l$ and $\phi^l_l = \phi^l$. The remaining terms are uniquely determined.

    To see the mutation, we observe that $P_1^R \simeq I_1$ and proceed by induction.

    \[\begin{tikzpicture}[scale=1.2]
        \tikzset{
            strand/.style={black, line width=1.5pt},
            mask/.style={white, line width=3.5pt}
        }

        
        \draw[mask] (5, 0) .. controls (5, 0.8) and (1,0.8) .. (1, 1);
        \draw[strand] (5, 0) .. controls (5, 0.8) and (1,0.8) .. (1, 1) node[at start, below] {$\Phi_5\rangle$} node[above, black] {$\langle P_5^R\Phi_5,$};
        
        \draw[mask] (4, 0) .. controls (4, 0.6) and (2,0.6) .. (2, 1);
        \draw[strand] (4, 0) .. controls (4, 0.6) and (2,0.6) .. (2, 1) node[at start, below] {$\Phi_4,$} node[above, black] {$P_4^R\Phi_4,$};

        \draw[mask] (3, 0) -- (3, 1);
        \draw[strand] (3, 0) -- (3, 1) node[at start, below] {$\Phi_3,$} node[above, black] {$P_3^R\Phi_3,$};

        \draw[mask] (2, 0) .. controls (2, 0.4) and (4,0.4) .. (4, 1);
        \draw[strand] (2, 0) .. controls (2, 0.4) and (4,0.4) .. (4, 1) node[at start, below] {$\Phi_2,$} node[above, black] {$P_2^R\Phi_2,$};
        
        \draw[mask] (1, 0) .. controls (1, 0.2) and (5,0.2) .. (5, 1);
        \draw[strand] (1, 0) .. controls (1, 0.2) and (5,0.2) .. (5, 1) node[at start, below] {$\langle\Phi_1,$} node[above, black] {$P_1^R\Phi_1\rangle$};
    \end{tikzpicture}\]
\end{proof}

Given a stable category $\Phi$, we define
\[\opname{\Cycle}_n(\Phi) := \Complex(\Phi,\Phi,\dots,\Phi)\]
to be the category of length $(n-1)$ complexes in $\Phi$. We will think of the objects as length $n$ acyclic cycles. There are functors
\[P_1,\dots,P_{n-1}:\opname{\Cycle}_n(\Phi)\to \Phi\]
where $P_i$ gives the $i$\ts{th} term of the complex. Completing the cycle, we also have
\[P_n:\opname{\Cycle}_n(\Phi)\to \Phi\]
given by assigning to each complex, its totalization. We should be careful about specifying the cohomological shift in this totalization but we will omit this detail. There is an action of $\bbZ$ on $\opname{\Cycle}_n(\Phi)$ by rotation of the acyclic cycle. As such, we will denote objects of $\opname{\Cycle}_n(\Phi)$ by $(\phi^i)_{i\in\bbZ/n}$ and for $i\in\bbZ/n$, we refer to $\phi^i$ as the $i$\ts{th} term of the acyclic cycle.

\begin{eg}
    A length 3 acyclic cycle is nothing more than an exact triangle. The $\bbZ$-action rotates the exact triangle.
\end{eg}

\begin{rmk}
    We may equivalently view $\opname{\Cycle}_n(\Phi)$ as the $(n-1)$\ts{th} term in the Waldhausen construction. The $\bbZ$-action is the additional paracyclic symmetry (see for example \cite{lurierotationinvariance}), and $n\in\bbZ$ acts by the double suspension $\Sigma^2$.
\end{rmk}

\begin{proof}[Proof of Theorem~\ref{thm:ykxn}]
    We take as our input, the datum $(\frakS;\Phi,T)$, where $\frakS = (\Phi_1,\dots,\Phi_{k})$ is our $\tau^n$-periodic SOD. Define inductively, the subcategory $\Phi_{i+m}$ for $i\geq1$ via $\tau$-mutation so that
    \begin{align*}
        \langle\Phi_{i},\Phi_{i+1},\dots,\Phi_{i+m-1}\rangle = \langle\Phi_{i+1},\Phi_{i+2},\dots,\Phi_{i+m}\rangle.
    \end{align*}
    Then, the periodicity asserts that $\Phi_{i} = \Phi_{i+n}$ for all $i\geq1$. So from here on, we will take this subscript $i\in\bbZ/n$.

    We will now define the output category $\calA$, and then construct on this a $(n-m)$-term, $\tau^n$-periodic SOD. We define
    \[\calA\subset \opname{Cycle}_n(\Phi)\]
    to be the full subcategory of those acyclic cycles $(\phi^i)_{i\in\bbZ/n}$ such that $\phi^i\in\Phi_i$.

    For each $i$, there is a natural functor by restriction,
    \begin{align*}
        W_i: \calA&\to \Complex(\Phi_{i},\Phi_{i+1},\dots,\Phi_{i+n-m-1})\\
        (\phi^i)_{i\in\bbZ/n}&\mapsto [\phi^i\to\dots\to \phi^{i+n-m-1}].
    \end{align*}
    This functor is an equivalence because the remaining categories $\Phi_{i+n-m},\dots,\Phi_{i+n-1}$ form an SOD of $\Phi$.
    
    By Lemma~\ref{lem:complexSOD}, we have an $\infty$-admissible SOD
    \[\Complex(\Phi_{i},\Phi_{i+1},\dots,\Phi_{i+n-m-1}) = \langle\Phi_{i+n-m-1},\dots,\Phi_{i+1},\Phi_{i}\rangle\]
    where the inclusions are given by the right adjoints to the projection functors
    \[Q_{j,i}: \Complex(\Phi_{i},\Phi_{i+1},\dots,\Phi_{i+n-m-1})\to \Phi_{i+j},\]
    for $j=0,\dots,n-m-1$. So via the equivalence $W_i$, we have an SOD of $\calA$ by the right adjoints to projection functors $Q_{j,i}\circ W_i$ where $j=0,\dots,n-m-1$.
    
    Observe now that
    \[Q_{j,i}\circ W_i \simeq P_{i+j},\]
    where $P_j:\calA\to \Phi_j$ outputs the $j$\ts{th} term of the acyclic cycle. So writing $\calA_j$ for the subcategory defined by $P_j^R$, we have for each $i$,
    \[\calA = \langle\calA_{i+n-m-1},\dots,\calA_{i+1},\calA_i\rangle.\]
    Hence $\calA$ admits a $\tau^n$-periodic SOD $(\calA_{n-m},\dots,\calA_1)$ which we denote $\mathfrak{T}$.

    We now exhibit on $\calA$ an autoequivalence which is conjugate-compatible with $\mathfrak{T}$. By the conjugate-compatibility of the autoequivalence $T$ on $\Phi$, we know that $T$ induces equivalences,
    \[T:\Phi_{i+m} \isomto \Phi_{i}.\]
    Indeed, $T$ is conjugate-compatible with $(\Phi_{i+1},\dots,\Phi_{i+m})$ so that $T(\Phi_{i+m})$ lies inside the joint right-orthogonal of $\Phi_{i+1},\dots,\Phi_{i+m-1}$, which is $\Phi_i$. Since $T$ is an autoequivalence, this is also an autoequivalence. See also the braid diagram appearing in Proof of Theorem~\ref{thm:alternate}.

    Thus the following is a well defined autoequivalence of $\calA$.
    \begin{align*}
        U:\calA&\to \calA\\
        (\phi^i)_{i\in\bbZ/n}&\mapsto (T\phi^{i+m})_{i\in\bbZ/n}
    \end{align*}
    This satisfies $P_{i}\circ U\simeq T\circ P_{i+m}$, so
    \[P_{i}^R\circ T\simeq U\circ P_{i+m}^R.\]
    Thus we have equivalences
    \[U:\calA_{i+m}\isomto \calA_i.\]
    We learn that $U^{-1}$ is conjugate-compatible with $\mathfrak{T}$, and so
    \[(\mathfrak{T};\calA,U^{-1})\in\sfAut(n-m)^{n\opname{-periodic}}.\]

    We now compute the composition
    \[\sfAut(m)^{n\opname{-periodic}}\to \sfAut(n-m)^{n\opname{-periodic}}\to \sfAut(m)^{n\opname{-periodic}}\]
    by applying the procedure twice.

    Let $\calB\subset \opname{Cycle}_n(\calA)$ be the subcategory of acyclic cycles $(a_j)_{j\in\bbZ/n}$ so that $a_j\in\calA_j$ (where the differential under the $j$-indexing is backwards). We write
    \[a_j = (\phi^i_j)_{i\in\bbZ/n}\]
    and think of $\phi^i_j$ as forming an ``acyclic bi-cycle''. As before, there are equivalences via restriction,
    \[\calB\isomto  \Complex(\calA_m,\dots,\calA_1).\]
    Combining this with the restriction
    \[\calA\isomto \Complex(\Phi_1,\dots,\Phi_{n-m}),\]
    we learn that we may restrict our ``acyclic bi-cycle'' $(\phi^i_j)$ to a bicomplex inside the box $1\leq i\leq n-m, 1\leq j \leq m$. Unravelling the condition to be in $\calA_i$, the resulting bicomplexes are of the following form, with conditions equivalent to those of Lemma~\ref{lem:bicomplex}.
    \[\begin{tikzcd}[ampersand replacement=\&]
        \vdots \& \vdots \& \vdots \& \vdots \\
        0 \& 0 \& 0 \& 0 \\
        {\phi^m_m} \& 0 \& 0 \& 0 \\
        \vdots \& \ddots \& 0 \& 0 \\
        {\phi^2_m} \& \dots \& {\phi^2_2} \& 0 \\
        {\phi^1_m} \& \cdots \& {\phi^1_2} \& {\phi^1_1}
        \arrow[from=2-1, to=1-1]
        \arrow[from=2-1, to=2-2]
        \arrow[from=2-2, to=1-2]
        \arrow[from=2-2, to=2-3]
        \arrow[from=2-3, to=1-3]
        \arrow[from=2-3, to=2-4]
        \arrow[from=2-4, to=1-4]
        \arrow[from=3-1, to=2-1]
        \arrow[from=3-1, to=3-2]
        \arrow[from=3-2, to=2-2]
        \arrow[from=3-2, to=3-3]
        \arrow[from=3-3, to=2-3]
        \arrow[from=3-3, to=3-4]
        \arrow[from=3-4, to=2-4]
        \arrow[from=4-1, to=3-1]
        \arrow[from=4-1, to=4-2]
        \arrow[from=4-2, to=3-2]
        \arrow[from=4-2, to=4-3]
        \arrow[from=4-3, to=3-3]
        \arrow[from=4-3, to=4-4]
        \arrow[from=4-4, to=3-4]
        \arrow[from=5-1, to=4-1]
        \arrow[from=5-1, to=5-2]
        \arrow[from=5-2, to=4-2]
        \arrow[from=5-2, to=5-3]
        \arrow[from=5-3, to=4-3]
        \arrow[from=5-3, to=5-4]
        \arrow[from=5-4, to=4-4]
        \arrow[from=6-1, to=5-1]
        \arrow[from=6-1, to=6-2]
        \arrow[from=6-2, to=5-2]
        \arrow[from=6-2, to=6-3]
        \arrow[from=6-3, to=5-3]
        \arrow[from=6-3, to=6-4]
        \arrow[from=6-4, to=5-4]
    \end{tikzcd}\]
    Thus we learn that $\calB\simeq \Phi$. Under this identification, the SOD on $\calB$ is given by $\delta^{-2}\cdot\frakS$, and the autoequivalence by $T$. So the composition is
    \[(\frakS;\Phi,T)\mapsto (\mathfrak{T};\calA,U^{-1})\mapsto (\delta^{-2}\cdot\frakS;\Phi,T).\]
    This proves the desired bijection at the level of objects.
\end{proof}

\appendix

\section{Spherical monads}\label{appendix:sphmonad}

\subsection{Strengthening Christ's theorem}\label{subsection:christ}
We recall the following theorem of Christ.

\begin{thm}[Christ, \cite{christSphericalMonadicAdjunctions2023}*{Theorem 4.1}]\label{thm:christ}
    Let $\Phi$ be a stable $\infty$-category and $M$ a monad with unit $u:\id\to M$ and let $T:=\Cone(u)$. The following are equivalent.
    \begin{enumerate}
        \item \label{item:1} The endofunctor $T$ is an equivalence, and the unit $u$ satisfies $T u\simeq uT$.
        \item \label{item:2} The monadic adjunction $S:\Phi\tofrom \Mod_\Phi(M):R$ is spherical.
    \end{enumerate}
\end{thm}

Let us split statement (\ref{item:1}) into two parts.
\begin{enumerate}[label=(\roman*)]
    \item \label{item:1a} The endofunctor $T$ is an equivalence.
    \item \label{item:1b} The unit $u$ satisfies $T u\simeq uT$.
\end{enumerate}

Theorem~\ref{thm:christstrengthening1} strengthens this by showing that statement \ref{item:1b} in fact always holds. We begin by recalling the meaning of $Tu\simeq uT$. As written here, this is ambiguous as these have different targets, $TM$ and $MT$ respectively. To make sense of this, we recall a version of \cite{christSphericalMonadicAdjunctions2023}*{Lemma 2.2}, stated in terms of monads instead of adjunctions.

\begin{lem}\label{lem:commutes}
    Let $M$ be a monad with unit $u$ and let $T:=\Cone(u:\id\to M)$. Then, there is a canonical identification $TM\simeq MT$.
\end{lem}

\begin{proof}
    Consider the naturally commuting square
    \[\begin{tikzcd}
        M & MM \\
        M & M.
        \arrow["uM", from=1-1, to=1-2]
        \arrow["{1_M}", from=1-1, to=2-1]
        \arrow["m", from=1-2, to=2-2]
        \arrow["{1_M}", from=2-1, to=2-2]
    \end{tikzcd}\]
    We may compute the total complex in two ways. Taking the fiber vertically first and then cone horizontally yields $\coCone(m)$. Taking cone horizontally and then fiber vertically yields $TM$. So we have the identification,
    \[\coCone(m)\simeq TM.\]
    By symmetry, we also have
    \[\coCone(m)\simeq MT.\]
    Composing these two provides the desired identification, $TM\simeq MT$.
\end{proof}

We may now state the key assertion.

\begin{prop}\label{prop:main}
    Let $\Phi$ be a stable $\infty$-category and $M$ a monad with unit $u:\id\to M$ and let $T:=\Cone(u)$. Then, the unit $u$ satisfies $Tu\simeq uT$ under the equivalence $TM\simeq MT$ of Lemma~\ref{lem:commutes}.
\end{prop}

\begin{proof}
    The idea of the proof is the same as, and will be clearly compatible with the proof of Lemma~\ref{lem:commutes}. Consider the following commutative cube which we call $X$.
    \[\begin{tikzcd}
        1 && M & \\
        & M && MM \\
        M && M & {} \\
        & M && M
        \arrow["u", from=1-1, to=1-3]
        \arrow["u", from=1-1, to=2-2]
        \arrow["u", from=1-1, to=3-1]
        \arrow["Mu", from=1-3, to=2-4]
        \arrow["{{ }}"{description}, from=1-3, to=3-3]
        \arrow["uM", from=2-2, to=2-4]
        \arrow[from=2-2, to=4-2]
        \arrow["m", from=2-4, to=4-4]
        \arrow[from=3-1, to=3-3]
        \arrow[from=3-1, to=4-2]
        \arrow[from=3-3, to=4-4]
        \arrow[from=4-2, to=4-4]
    \end{tikzcd}\]
    For each direction $*\in\{\hor,\ver,\diag\}$, we denote by $\Cone_*(-)$ the result of applying $\Cone$ in the direction indicated by $*$. We write $\coCone_*(-)$ similarly.

    We have comparisons,
    \begin{equation}\label{eqn:comparison1}
        \coCone_{\ver}(\Cone_{\hor}(X))\simeq\Cone_{\hor}(\coCone_{\ver}(X)),
    \end{equation}
    and
    \begin{equation}\label{eqn:comparison2}
        \coCone_{\ver}(\Cone_{\dia}(X))\simeq\Cone_{\dia}(\coCone_{\ver}(X)).
    \end{equation}
    It is easy to calculate that
    \[\coCone_{\ver}(\Cone_{\hor}(X))\simeq\begin{tikzcd}
        T & \\
        & TM.
        \arrow["Tu", from=1-1, to=2-2]
    \end{tikzcd}\]
    and similarly that
    \[\coCone_{\ver}(\Cone_{\dia}(X))\simeq
    \begin{tikzcd}
        T && MT.
        \arrow["uT", from=1-1, to=1-3]
    \end{tikzcd}\]
    We also calculate that
    \[\coCone_{\ver}(X)\simeq\begin{tikzcd}
        {T[-1]} && 0 & \\
        & 0 && {\coCone(m)}.
        \arrow[from=1-1, to=1-3]
        \arrow[from=1-1, to=2-2]
        \arrow[from=1-3, to=2-4]
        \arrow[from=2-2, to=2-4]
    \end{tikzcd}\]
    So we have the comparison,
    \begin{equation}\label{eqn:comparison3}
        \Cone_{\hor}(\coCone_{\ver}(X))\simeq\Cone_{\dia}(\coCone_{\ver}(X)).
    \end{equation}
    Composing comparisons (\ref{eqn:comparison1}),(\ref{eqn:comparison2}) and (\ref{eqn:comparison3}), and taking care of signs, we deduce the desired result.
\end{proof}

\begin{proof}[Proof of Theorem~\ref{thm:christstrengthening1}]
    This follows from the original version, Theorem~\ref{thm:christ}, together with Proposition~\ref{prop:main}.
\end{proof}

\subsection{A description of objects of \texorpdfstring{$\sfSphMnd(n)$}{SphMnd(n)}}\label{subsection:sphmonad}
We recall that currently, $\sfSphMnd(n)$ is defined as
\[\sfSphMnd(n) := \sfAut(n)\times_{\sfAut}\sfSphMnd.\]
In this subsection, we provide a direct characterization of abstract prelocalized perverse schobers, analogously to and compatibly with the equivalence $\sfAdj(n)\simeq\sfBdj(n)$ of Proposition~\ref{prop:bdjisadj}.

The following corollary of Proposition~\ref{prop:conservative} says that we are allowed to forget about $\Psi$.

\begin{cor}\label{cor:sphmonad}
    A prelocalized schober is equivalent to the data of categories $\Phi_i$ together with a monad $M$ on $\oplus_i\Phi_i$ such that for each $j$, the composition
    \[\Phi_j \to\oplus_i\Phi_i \to \Mod_M(\oplus_i\Phi_i)\]
    is spherical.
\end{cor}

\begin{proof}
    We apply Barr-Beck to the adjunction,
    \[\oplus_i\Phi_i \tofrom \Psi.\]
    The right adjoint is conservative by Proposition~\ref{prop:conservative} and spherical functors automatically preserve colimits. Moreover, all categories are assumed cocomplete. We deduce that there is an identification
    \[\Psi\isomto\Mod_M(\oplus_i\Phi_i)\]
    thus we can reconstruct the perverse schober from just the data of the monad.
\end{proof}

\begin{defn}
    A monad $M=\oplus_{i,j}M_{ij}$ acting on $\oplus_i\Phi_i$ is said to be a spherical matrix monad if the condition in Corollary~\ref{cor:sphmonad} is satisfied.
\end{defn}

It is interesting to characterize which monads on $\oplus_i\Phi_i$ are spherical matrix monads without reference to the Eilenberg--Moore adjunction, as we studied for $n=1$ in the previous section. The following is a generalization of Theorem~\ref{thm:christstrengthening1}. We may also understand this as a categorification of the quiver description in \cite{gelfandPerverseSheavesQuivers1996}.

\begin{prop}\label{prop:sphmonad}
    Consider categories $\Phi_i$, together with a monad $M=\oplus_{i,j}M_{ij}$ acting on $\oplus_i\Phi_i$. This monad is a spherical matrix monad if and only if the following two conditions are satisfied.
    \begin{enumerate}
        \item \label{enum:sphmonad1} For each $i$, $T_i:=\coCone(\id_{\Phi_i}\to M_{ii})$ is an equivalence of categories.
        \item \label{enum:sphmonad2} For each $i,j$, the map $M_{ji}\to M_{ij}^RT_i[1]$ is an equivalence of functors, where this map is obtained via adjunction from the composition $M_{ij}M_{ji}\to M_{ii}\to T_i[1]$.
    \end{enumerate}
    In fact, we may replace Condition~(\ref{enum:sphmonad2}) with the following weaker condition.
    \begin{enumerate}[resume]
        \item \label{enum:sphmonad3} For each $i\neq j$, the map $M_{ji}\to M_{ij}^RT_i[1]$ is an equivalence of functors, where this map is obtained via adjunction from the composition $M_{ij}M_{ji}\to M_{ii}\to T_i[1]$.
    \end{enumerate}
\end{prop}

\begin{proof}
    The stronger statement follows from the weaker statement. Indeed, by Theorem~\ref{thm:christstrengthening1}, Condition~(\ref{enum:sphmonad1}) implies that for each $i$, $M_{ii}$ is a spherical monad. This then implies Condition~(\ref{enum:sphmonad2}) whenever $j=i$.
    
    We prove the weaker statement. We begin by noting that the adjunctions
    \[S_i:\Phi_i\tofrom\opname{Mod}_{\oplus_i\Phi}(M_{ij}):R_i\] 
    recover the monad via $R_iS_j\isomto M_{ij}$. In particular, $T_i \simeq \coCone(\id\to R_iS_i)$.

    Suppose that $M$ is a spherical matrix monad, or equivalently that each $S_i$ is spherical. Then by the above, Condition~(\ref{enum:sphmonad1}) is clear. We prove Condition~(\ref{enum:sphmonad2}). Recall that $R_i$ admits a right adjoint $R_i^R$ and the natural map
    \[S_i\to R^R_iT_i[1]\]
    is an equivalence. Composing with $R_j$ on the left, we see that $M_{ji}\to M_{ij}^RT_i[1]$ is an equivalence.
    
    We must be careful here and verify that this is identified with the natural transformation coming from the composition $M_{ij}M_{ji}\to M_{ii}\to T_i[1]$ by adjunction. We observe that $M_{ii}\to T_i[1]$ can be post-composed after the fact, so it suffices to compare the natural transformations
    \begin{itemize}
        \item $R_j(S_i\to (R_i^RR_i)S_i)$, and
        \item $M_{ji}\to M_{ij}^RM_{ii}$.
    \end{itemize}
    This can be seen by comparing the following diagrams representing these natural transformations, where the functors are composed horizontally as in writing order, and natural transformations go downwards.

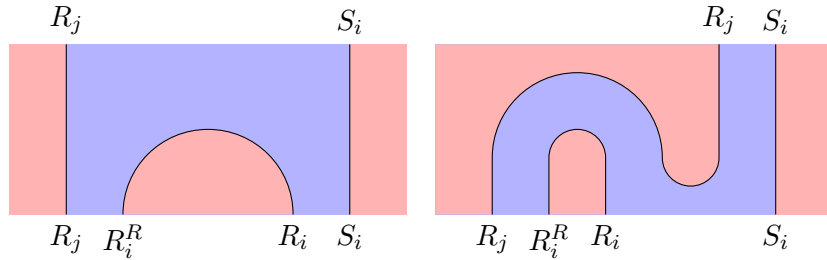
\tikzstyle{catc}=[red!30]
\tikzstyle{catd}=[blue!30]

\[
\begin{tikzpicture}[scale=0.75]
\path coordinate[] (epsilon)
+(-1.5,-1.5) coordinate[label=below:$R_i^R$] (tl)
+(1.5,-1.5) coordinate[label=below:$R_i$] (tr)
+(-2.5,1.5) coordinate[label=above:$R_j$] (ul)
+(2.5,1.5) coordinate[label=above:$S_i$] (ur)
+(-2.5,-1.5) coordinate[label=below:$R_j$] (dl)
+(2.5,-1.5) coordinate[label=below:$S_i$] (dr);
\draw (tl) to[out=90, in=180] (epsilon.center) to[out=0, in=90] (tr);
\draw (ul) to (dl);
\draw (ur) to (dr);
\begin{scope}[on background layer]
\fill[catd] +(-3.5,-1.5) rectangle +(3.5,1.5);
\fill[catc] (tl) to[out=90, in=180] (epsilon.center) to[out=0, in=90] (tr) --
cycle;
\fill[catc] +(-3.5,-1.5) rectangle +(-2.5,1.5);
\fill[catc] +(2.5,-1.5) rectangle +(3.5,1.5);
\end{scope}
\end{tikzpicture}
\quad
\begin{tikzpicture}[scale=0.75]
\path coordinate[] (tr)
++(0,-3) coordinate[] (br)
++(-1,0) coordinate[label=below:$S_i$] (brl)
++(0,3) coordinate[label=above:$S_i$] (trl)
++(-1,0) coordinate[label=above:$R_j$] (a)
++(0,-2) coordinate[] (b)
++(-0.5,-0.5) coordinate[] (c)
++(-0.5,0.5) coordinate[] (d)
++(-1.5,1.5) coordinate[] (e)
++(-1.5,-1.5) coordinate[] (f)
++(0,-1) coordinate[label=below:$R_j$] (g)
++(-1,0) coordinate[] (bl)
++(0,3) coordinate[] (tl)
++(2,-3) coordinate[label=below:$R_i^R$] (h)
++(0,1) coordinate[] (i)
++(0.5,0.5) coordinate[] (j)
++(0.5,-0.5) coordinate[] (k)
++(0,-1) coordinate[label=below:$R_i$] (l);
\draw (trl) -- (brl);
\draw (a) -- (b) to [out=-90, in=0] (c.center) to [out=180, in=-90] (d) to [out=90, in=0] (e.center) to [out=180, in=90] (f) -- (g);
\draw (h) -- (i) to [out=90, in=180] (j.center) to [out=0, in=90] (k) -- (l);
\begin{scope}[on background layer]
\fill[catd] (bl) rectangle (tr);
\fill[catc] (h) -- (i) to [out=90, in=180] (j.center) to[out=0, in=90] (k) -- (l) -- cycle;
\fill[catc] (a) -- (b) to [out=-90, in=0] (c.center) to [out=180, in=-90] (d) to [out=90, in=0] (e.center) to [out=180, in=90] (f) -- (g) -- (bl) -- (tl) -- cycle;
\fill[catc] (brl) rectangle (tr);
\end{scope}
\end{tikzpicture}
\]
    Using this diagram as a mnemonic, we may produce the desired identification.

    Let us now prove that if $M$ satisfies Condition~(\ref{enum:sphmonad1}) and Condition~(\ref{enum:sphmonad2}), then for each $i$, $S_i$ is spherical. We will apply the 2-out-of-4 property, \cite{christSphericalMonadicAdjunctions2023}*{Theorem 2.15}. We must first show that $R_i$ admits a right adjoint. It is clear that $M_{ii}$ admits a right adjoint $M_{ii}^R$, since $M$ is an extension of two autoequivalences which each automatically admit right adjoints. Observing that $\opname{Mod}_{\Phi}(M_{ii})\isomto\opname{coMod}_{\Phi}(M_{ii}^R)$ we get the right adjoint from the comonadic adjunction,
    \[R_i: \opname{coMod}_{\Phi}(M_{ii}^R)\tofrom \Phi: R_i^R.\]
    
    We can now apply the 2-out-of-4 property. It is immediate that the cotwist functor of $S_i\adjto R_i$ is invertible, and so it suffices to verify that the natural map $S_i\to R_i^RT_i[1]$ is an equivalence. This becomes an equivalence after we post-compose with $\oplus_jR_j$, by Condition~(\ref{enum:sphmonad2}) and the comparison made above. Since $\oplus_jR_j$ is conservative, it follows that $S_i\to R_i^RT_i[1]$ is indeed an equivalence.
\end{proof}

We may also construct the projection functor directly.

\begin{cor}
    Given an object $(\Phi_i,\Psi,S_i,R_i)$ of $\sfSph(n)$, the monad $M=\oplus_{i,j}R_iS_j$ is a spherical matrix monad. In particular, there is a projection functor 
    \[\sfSph(n)\to \sfSphMnd(n).\]
    This agrees with the projection functor constructed in \S\ref{subsection:quotientismonad}.
\end{cor}

\begin{proof}
    The two conditions of Proposition~\ref{prop:sphmonad} follow by the same argument as the proof of the ``only if'' direction, noting that this direction didn't use the conservativity of $\oplus_i R_i$.
\end{proof}

\begin{cor}\label{prop:decategorification2}
    Recall notation from Proposition~\ref{prop:gmvquiver}. Given an abstract prelocalized perverse schober given as a spherical matrix monad $\oplus_{i,j} M_{ij}$ acting on $\oplus_i\Phi_i$, we may form an object of $\opname{Q}(n)$ by taking $M_i := K_0(\Phi_i)$ and $m_{ij} = K_0(M_{ij})$.
\end{cor}

\bibliography{microlocal}

@article{acj,
  title = {On the $(\infty,2)$-category of perverse schobers on a surface},
  author = {Abell{\'a}n, Fernando and Christ, Merlin and Jasso, Gustavo},
  note = {In preparation}
}

@article{abellanFreeFibrationsLax2026,
  title = {Free Fibrations, Lax Colimits and {{Kan}} Extensions for $(\infty,2)$-Categories},
  author = {Abell{\'a}n, Fernando and Haugseng, Rune and Martini, Louis},
  year = 2026,
  eprint = {2602.07604},
  primaryclass = {math.CT},
  publisher = {arXiv},
  doi = {10.48550/arXiv.2602.07604},
  url = {http://arxiv.org/abs/2602.07604},
  urldate = {2026-07-31},
  archiveprefix = {arXiv}
}

@article{addingtonNewDerivedSymmetries2016,
  title = {New Derived Symmetries of Some Hyperk\"ahler Varieties},
  author = {Addington, Nicolas},
  year = 2016,
  journal = {Algebraic Geometry},
  primaryclass = {math.AG},
  pages = {223--260},
  issn = {22142584},
  doi = {10.14231/AG-2016-011},
  urldate = {2026-09-11},
  archiveprefix = {arXiv}
}

@article{andronikofMicrolocalVersionRiemannHilbert1994,
  title = {A Microlocal Version of the {{Riemann-Hilbert}} Correspondence},
  author = {Andronikof, Emmanuel},
  year = 1994,
  journal = {Topological Methods in Nonlinear Analysis},
  volume = {4},
  pages = {417--425},
  publisher = {Nicolaus Copernicus University in Toru\'n, Juliusz Schauder Center for Nonlinear Studies},
  issn = {1230-3429},
  url = {https://projecteuclid.org/journals/topological-methods-in-nonlinear-analysis/volume-4/issue-2/A-microlocal-version-of-the-Riemann-Hilbert-correspondence/tmna/1479287054.full},
  urldate = {2026-08-06},
  langid = {english}
}

@article{barbacoviCompositionTwoSpherical2021,
  title = {On the Composition of Two Spherical Twists},
  author = {Barbacovi, Federico},
  year = 2021,
  eprint = {2006.06016},
  primaryclass = {math.AG},
  publisher = {arXiv},
  doi = {10.48550/arXiv.2006.06016},
  url = {http://arxiv.org/abs/2006.06016},
  urldate = {2026-09-08},
  archiveprefix = {arXiv}
}

@article{beilinson2016constructible,
  title={Constructible sheaves are holonomic},
  author={Beilinson, Alexander},
  journal={Selecta Mathematica},
  volume={22},
  number={4},
  pages={1797--1819},
  year={2016},
  publisher={Springer}
}

@book{bernstein2006jacobi,
  title={The Jacobi-Perron algorithm: its theory and application},
  author={Bernstein, Leon},
  year={2006},
  publisher={Springer}
}

@article{bondalPerverseSchobersBirational2018,
  title={Perverse schobers and birational geometry},
  author={Bondal, Alexey and Kapranov, Mikhail and Schechtman, Vadim},
  journal={Selecta Mathematica},
  volume={24},
  number={1},
  pages={85--143},
  year={2018},
  publisher={Springer}
}

@article{christGinzburgAlgebrasTriangulated2022,
  title = {Ginzburg Algebras of Triangulated Surfaces and Perverse Schobers},
  author = {Christ, Merlin},
  year = 2022,
  journal = {Forum of Mathematics, Sigma},
  volume = {10},
  primaryclass = {math.AT},
  pages = {e8},
  issn = {2050-5094},
  doi = {10.1017/fms.2022.1},
  url = {http://arxiv.org/abs/2101.01939},
  urldate = {2026-08-14},
  archiveprefix = {arXiv}
}

@article{christLaxAdditivity2025,
  title={Lax additivity},
  author={Christ, Merlin and Dyckerhoff, Tobias and Walde, Tashi},
  journal={Documenta Mathematica},
  volume={31},
  number={4},
  pages={787},
  year={2026},
  publisher={European Mathematical Society (EMS)}
}

@article{christRelativeCalabiYauStructures2026,
  title={Relative {{Calabi--Yau}} structures and perverse schobers on surfaces},
  author={Christ, Merlin},
  journal={Journal of Noncommutative Geometry},
  volume={20},
  number={3},
  pages={871--951},
  year={2026}
}

@article{christSphericalMonadicAdjunctions2023,
  title = {Spherical Monadic Adjunctions of Stable Infinity Categories},
  author = {Christ, Merlin},
  year = 2023,
  journal = {International Mathematics Research Notices},
  volume = {2023},
  primaryclass = {math},
  pages = {13153--13213},
  issn = {1073-7928, 1687-0247},
  doi = {10.1093/imrn/rnac187},
  url = {http://arxiv.org/abs/2010.05294},
  urldate = {2025-09-27},
  archiveprefix = {arXiv}
}

@article{cotePerverseMicrosheaves2025,
  title = {Perverse {{Microsheaves}}},
  author = {C{\^o}t{\'e}, Laurent and Kuo, Christopher and Nadler, David and Shende, Vivek},
  year = 2025,
  eprint = {2209.12998},
  primaryclass = {math.SG},
  publisher = {arXiv},
  doi = {10.48550/arXiv.2209.12998},
  url = {http://arxiv.org/abs/2209.12998},
  urldate = {2026-07-22},
  archiveprefix = {arXiv}
}

@book{dyckerhoffHigherSegalSpaces2019,
  title = {Higher {{Segal}} Spaces {{I}}},
  author={Dyckerhoff, Tobias and Kapranov, Mikhail},
  series={Lecture Notes in Mathematics},
  volume={2244},
  year={2019},
  publisher={Springer},
  doi={10.1007/978-3-030-27124-4},
  isbn={978-3-030-27122-0},
  url={https://link.springer.com/book/10.1007/978-3-030-27124-4}
}

@article{dyckerhoffNsphericalFunctorsCategorification2023,
  title={N-spherical functors and categorification of {{Euler}}’s continuants: T. Dyckerhoff et al.},
  author={Dyckerhoff, Tobias and Kapranov, Mikhail and Schechtman, Vadim},
  journal={Mathematische Zeitschrift},
  volume={312},
  number={3},
  pages={75},
  year={2026},
  publisher={Springer}
}

@article{dyckerhoffPerverseSchobersCoxeter2025,
  title = {Perverse Schobers of {{Coxeter}} Type $\mathbb{A}$},
  author = {Dyckerhoff, Tobias and Wedrich, Paul},
  year = 2025,
  eprint = {2504.08496},
  primaryclass = {math},
  publisher = {arXiv},
  doi = {10.48550/arXiv.2504.08496},
  url = {http://arxiv.org/abs/2504.08496},
  urldate = {2026-04-14},
  archiveprefix = {arXiv}
}

@article{dyckerhoffSphericalAdjunctionsStable2021,
  title = {Spherical Adjunctions of Stable $\infty$-Categories and the Relative {{S-construction}}},
  author={Dyckerhoff, Tobias and Kapranov, Mikhail and Schechtman, Vadim and Soibelman, Yan},
  journal={Mathematische Zeitschrift},
  volume={307},
  number={4},
  pages={73},
  year={2024},
  publisher={Springer}
}

@article{faergemanNonvanishingGeometricWhittaker2022,
  title = {Non-Vanishing of Geometric {{Whittaker}} Coefficients for Reductive Groups},
  author={Faergeman, Joakim and Raskin, Sam},
  journal={J. Amer. Math. Soc},
  volume={38},
  number={4},
  pages={919--995},
  year={2025}
}

@article{gaitsgoryIndcoherentSheaves2012,
  title = {Ind-Coherent Sheaves},
  author = {Gaitsgory, Dennis},
  year = 2012,
  eprint = {1105.4857},
  primaryclass = {math},
  publisher = {arXiv},
  doi = {10.48550/arXiv.1105.4857},
  url = {http://arxiv.org/abs/1105.4857},
  urldate = {2025-08-09},
  archiveprefix = {arXiv}
}

@article{galvez-carrilloDecompositionSpacesIncidence2018,
  title = {Decomposition Spaces, Incidence Algebras and {{M\"obius}} Inversion {{I}}: Basic Theory},
  author={G{\'a}lvez-Carrillo, Imma and Kock, Joachim and Tonks, Andrew},
  journal={Advances in Mathematics},
  volume={331},
  pages={952--1015},
  year={2018},
  publisher={Elsevier}
}

@article{gammageBettiTatesThesis2025,
  title = {Betti {{Tate}}'s Thesis and the Trace of Perverse Schobers},
  author = {Gammage, Benjamin and Hilburn, Justin},
  year = 2025,
  journal = {Comptes Rendus. Math\'ematique},
  volume = {363},
  primaryclass = {math},
  pages = {169--181},
  issn = {1631-073X, 1778-3569},
  doi = {10.5802/crmath.703},
  url = {http://arxiv.org/abs/2210.06548},
  urldate = {2025-10-20},
  archiveprefix = {arXiv}
}

@article{gammageHypertoric2categoriesSymplectic2025,
  title = {Hypertoric 2-Categories {{O}} and Symplectic Duality},
  author={Gammage, Benjamin and Hilburn, Justin},
  journal={Communications in Mathematical Physics},
  volume={407},
  number={3},
  pages={55},
  year={2026},
  publisher={Springer}
}

@article{gammagePerverseSchobers3d2023,
  title = {Perverse Schobers and 3d Mirror Symmetry},
  author = {Gammage, Benjamin and Hilburn, Justin and {Mazel-Gee}, Aaron},
  year = 2023,
  eprint = {2202.06833},
  primaryclass = {math},
  publisher = {arXiv},
  doi = {10.48550/arXiv.2202.06833},
  url = {http://arxiv.org/abs/2202.06833},
  urldate = {2025-10-14},
  archiveprefix = {arXiv}
}

@article{gargGITRootStacks2026,
  title = {{{GIT}} for Root Stacks and 3d Mirror Symmetry},
  author = {Garg, Swapnil and Li, Ruoxi and Okitani, Yuji},
  year = 2026,
  eprint = {2608.30350},
  primaryclass = {math.RT},
  publisher = {arXiv},
  doi = {10.48550/arXiv.2608.30350},
  url = {http://arxiv.org/abs/2608.30350},
  urldate = {2026-09-01},
  archiveprefix = {arXiv}
}

@article{geiglePerpendicularCategoriesApplications1991,
  title = {Perpendicular Categories with Applications to Representations and Sheaves},
  author = {Geigle, Werner and Lenzing, Helmut},
  year = 1991,
  journal = {Journal of Algebra},
  volume = {144},
  pages = {273--343},
  issn = {00218693},
  doi = {10.1016/0021-8693(91)90107-J},
  url = {https://linkinghub.elsevier.com/retrieve/pii/002186939190107J},
  urldate = {2025-10-22},
  copyright = {https://www.elsevier.com/tdm/userlicense/1.0/},
  langid = {english}
}

@article{gelfandMicrolocalPerverseSheaves2005,
  title = {Microlocal {{Perverse Sheaves}}},
  author = {Gelfand, S. and MacPherson, R. and Vilonen, K.},
  year = 2005,
  eprint = {math/0509440},
  publisher = {arXiv},
  doi = {10.48550/arXiv.math/0509440},
  url = {http://arxiv.org/abs/math/0509440},
  urldate = {2025-07-10},
  archiveprefix = {arXiv}
}

@article{gelfandPerverseSheavesQuivers1996,
  title = {Perverse Sheaves and Quivers},
  author = {Gelfand, Sergei and MacPherson, Robert and Vilonen, Kari},
  year = 1996,
  journal = {Duke Mathematical Journal},
  volume = {83},
  issn = {0012-7094},
  doi = {10.1215/S0012-7094-96-08319-2},
  url = {https://projecteuclid.org/journals/duke-mathematical-journal/volume-83/issue-3/Perverse-sheaves-and-quivers/10.1215/S0012-7094-96-08319-2.full},
  urldate = {2025-08-14},
  langid = {english}
}

@article{halpern-leistnerAutoequivalencesDerivedCategories2016,
  title = {Autoequivalences of Derived Categories via Geometric Invariant Theory},
  author = {{Halpern-Leistner}, Daniel and Shipman, Ian},
  year = 2016,
  journal = {Advances in Mathematics},
  volume = {303},
  pages = {1264--1299},
  issn = {00018708},
  doi = {10.1016/j.aim.2016.06.017},
  urldate = {2026-09-12},
  langid = {english}
}

@article{harderPerverseSheavesCategories2019,
  title = {Perverse Sheaves of Categories and Some Applications},
  author = {Harder, Andrew and Katzarkov, Ludmil},
  year = 2019,
  journal = {Advances in Mathematics},
  volume = {352},
  pages = {1155--1205},
  issn = {00018708},
  doi = {10.1016/j.aim.2019.05.024},
  urldate = {2026-09-12},
  langid = {english}
}

@article{haugsengLaxTransformationsAdjunctions2021,
  title = {On Lax Transformations, Adjunctions, and Monads in $(\infty,2)$-Categories},
  author = {Haugseng, Rune},
  year = 2021,
  eprint = {2002.01037},
  primaryclass = {math.CT},
  publisher = {arXiv},
  doi = {10.48550/arXiv.2002.01037},
  url = {http://arxiv.org/abs/2002.01037},
  urldate = {2026-07-19},
  archiveprefix = {arXiv}
}

@article{kapranovPerverseSchobers2015,
  title = {Perverse {{Schobers}}},
  author = {Kapranov, Mikhail and Schechtman, Vadim},
  year = 2015,
  eprint = {1411.2772},
  primaryclass = {math},
  publisher = {arXiv},
  doi = {10.48550/arXiv.1411.2772},
  url = {http://arxiv.org/abs/1411.2772},
  urldate = {2025-09-27},
  archiveprefix = {arXiv}
}

@article{kapranovPerverseSchobersAlgebra2020,
  title = {Perverse Schobers and the {{Algebra}} of the {{Infrared}}},
  author = {Kapranov, Mikhail and Soibelman, Yan and Soukhanov, Lev},
  year = 2020,
  eprint = {2011.00845},
  primaryclass = {math},
  publisher = {arXiv},
  doi = {10.48550/arXiv.2011.00845},
  url = {http://arxiv.org/abs/2011.00845},
  urldate = {2025-09-16},
  archiveprefix = {arXiv}
}

@book{kashiwaraSheavesManifolds1990,
  title = {Sheaves on {{Manifolds}}},
  author = {Kashiwara, Masaki and Schapira, Pierre},
  editor = {Artin, M. and Chern, S. S. and Coates, J. and Fr{\"o}hlich, J. M. and Hironaka, H. and Hirzebruch, F. and H{\"o}rmander, L. and MacLane, S. and Moore, C. C. and Moser, J. K. and Nagata, M. and Schmidt, W. and Scott, D. S. and Sinai, {\relax Ya}. G. and Tits, J. and Waldschmidt, M. and Watanabe, S. and Berger, M. and Eckmann, B. and Varadhan, S. R. S.},
  year = 1990,
  series = {Grundlehren Der Mathematischen {{Wissenschaften}}},
  volume = {292},
  publisher = {Springer Berlin Heidelberg},
  address = {Berlin, Heidelberg},
  doi = {10.1007/978-3-662-02661-8},
  url = {http://link.springer.com/10.1007/978-3-662-02661-8},
  urldate = {2025-09-09},
  copyright = {http://www.springer.com/tdm},
  isbn = {978-3-642-08082-1 978-3-662-02661-8},
  langid = {english}
}

@incollection{kiehlLefschetzTheoryBrylinskiRadon2001,
  title = {Lefschetz {{Theory}} and the {{Brylinski-Radon Transform}}},
  booktitle = {Weil {{Conjectures}}, {{Perverse Sheaves}} and l'adic {{Fourier Transform}}},
  author = {Kiehl, Reinhardt and Weissauer, Rainer},
  editor = {Kiehl, Reinhardt and Weissauer, Rainer},
  year = 2001,
  pages = {203--224},
  publisher = {Springer},
  address = {Berlin, Heidelberg},
  doi = {10.1007/978-3-662-04576-3_5},
  url = {https://doi.org/10.1007/978-3-662-04576-3_5},
  urldate = {2026-08-06},
  isbn = {978-3-662-04576-3},
  langid = {english}
}

@article{kuwagakiCategorificationLegendrianKnots2019,
  title = {Categorification of {{Legendrian}} Knots},
  author={Kuwagaki, Tatsuki},
  journal={Pure and Applied Mathematics Quarterly},
  volume={16},
  number={3},
  pages={421--437},
  year={2020},
  publisher={International Press of Boston, Inc. Somerville, MA 02143, USA}
}

@article{lehmerJacobisExtensionContinued1918,
  title = {On {{Jacobi}}'s {{Extension}} of the {{Continued Fraction Algorithm}}},
  author = {Lehmer, D. N.},
  year = 1918,
  journal = {Proceedings of the National Academy of Sciences},
  volume = {4},
  pages = {360--364},
  publisher = {Proceedings of the National Academy of Sciences},
  doi = {10.1073/pnas.4.12.360},
  url = {https://www.pnas.org/doi/10.1073/pnas.4.12.360},
  urldate = {2026-09-05}
}

@article{lurieHigherAlgebra2017,
  title = {Higher {{Algebra}}},
  author = {Lurie, Jacob},
  year = 2017,
  note = {Available at \url{https://www.math.ias.edu/~lurie/}}
}

@article{lurierotationinvariance,
  title = {Rotation Invariance in Algebraic {{K-theory}}},
  author = {Lurie, Jacob},
  note = {Available at \url{https://www.math.ias.edu/~lurie/}}
}

@article{macphersonPerverseSheavesSingularities1988,
  title = {{Perverse sheaves with singularities along the curve Yn = Xm.}},
  author = {MacPherson, R. and Vilonen, K.},
  year = 1988,
  journal = {Commentarii mathematici Helvetici},
  volume = {63},
  pages = {89--102},
  issn = {0010-2571; 1420-8946/e},
  url = {https://eudml.org/doc/140110},
  urldate = {2026-07-06},
  langid = {und}
}

@article{nadlerSheafQuantizationWeinstein2022,
  title = {Sheaf Quantization in {{Weinstein}} Symplectic Manifolds},
  author = {Nadler, David and Shende, Vivek},
  year = 2022,
  eprint = {2007.10154},
  primaryclass = {math},
  publisher = {arXiv},
  doi = {10.48550/arXiv.2007.10154},
  url = {http://arxiv.org/abs/2007.10154},
  urldate = {2026-04-12},
  archiveprefix = {arXiv}
}

@article{oblomkovHilbertSchemePlane2012,
  title = {The {{Hilbert}} Scheme of a Plane Curve Singularity and the {{HOMFLY}} Polynomial of Its Link},
  author = {Oblomkov, Alexei and Shende, Vivek},
  year = 2012,
  journal = {Duke Mathematical Journal},
  volume = {161},
  primaryclass = {math.AG},
  issn = {0012-7094},
  doi = {10.1215/00127094-1593281},
  url = {http://arxiv.org/abs/1003.1568},
  urldate = {2026-08-10},
  archiveprefix = {arXiv}
}

@article{riehlHomotopyCoherentAdjunctions2015,
  title = {Homotopy Coherent Adjunctions and the Formal Theory of Monads},
  author={Riehl, Emily and Verity, Dominic},
  journal={Advances in Mathematics},
  volume={286},
  pages={802--888},
  year={2016},
  publisher={Elsevier}
}

@article{saitoSingularSupportsMixed2025,
  title = {On Singular Supports in Mixed Characteristic},
  author = {Saito, Takeshi},
  year = 2025,
  eprint = {2501.07965},
  primaryclass = {math.AG},
  publisher = {arXiv},
  doi = {10.48550/arXiv.2501.07965},
  url = {http://arxiv.org/abs/2501.07965},
  urldate = {2026-08-06},
  archiveprefix = {arXiv}
}

@article{segalAllAutoequivalencesAre2020,
  title = {All Autoequivalences Are Spherical Twists},
  author={Segal, Ed},
  journal={International Mathematics Research Notices},
  volume={2018},
  number={10},
  pages={3137--3154},
  year={2018},
  publisher={Oxford University Press}
}

@article{seidelBraidGroupActions2000,
  author = {Paul Seidel and Richard Thomas},
  title = {{Braid group actions on derived categories of coherent sheaves}},
  volume = {108},
  journal = {Duke Mathematical Journal},
  number = {1},
  publisher = {Duke University Press},
  pages = {37 -- 108},
  year = {2001},
  doi = {10.1215/S0012-7094-01-10812-0},
  URL = {https://doi.org/10.1215/S0012-7094-01-10812-0}
}

@article{seidelSymplecticHomologyHochschild2008,
  title = {Symplectic Homology as {{Hochschild}} Homology},
  author={Seidel, Paul},
  journal={Algebraic geometry—Seattle 2005},
  pages={415--434},
  year={2009}
}

@article{shendeClusterVarietiesLegendrian2019,
  title = {Cluster Varieties from {{Legendrian}} Knots},
  author = {Shende, Vivek and Treumann, David and Williams, Harold and Zaslow, Eric},
  year = 2019,
  journal = {Duke Mathematical Journal},
  volume = {168},
  primaryclass = {math.SG},
  issn = {0012-7094},
  doi = {10.1215/00127094-2019-0027},
  urldate = {2026-09-05},
  archiveprefix = {arXiv}
}

@article{stefanichPresentable$inftyN$categories2020,
  title = {Presentable $(\infty, n)$-Categories},
  author = {Stefanich, Germ{\'a}n},
  year = 2020,
  eprint = {2011.03035},
  primaryclass = {math.AT},
  publisher = {arXiv},
  doi = {10.48550/arXiv.2011.03035},
  url = {http://arxiv.org/abs/2011.03035},
  urldate = {2026-07-29},
  archiveprefix = {arXiv}
}

@article{waschkiesMicrolocalPerverseSheaves2002,
  title = {Microlocal Perverse Sheaves},
  author = {Waschkies, Ingo},
  year = 2002,
  eprint = {math/0209341},
  publisher = {arXiv},
  doi = {10.48550/arXiv.math/0209341},
  url = {http://arxiv.org/abs/math/0209341},
  urldate = {2026-07-11},
  archiveprefix = {arXiv}
}

@article{waschkiesStackMicrolocalPerverse2004,
  title = {The Stack of Microlocal Perverse Sheaves},
  author = {Waschkies, Ingo},
  year = 2004,
  journal = {Bulletin de la soci{\'e}t{\'e} math{\'e}matique de France},
  volume = {132},
  pages = {397--462},
  issn = {0037-9484, 2102-622X},
  doi = {10.24033/bsmf.2469},
  url = {http://www.numdam.org/item?id=BSMF_2004__132_3_397_0},
  urldate = {2026-08-10},
  langid = {english}
}

\end{document}